\documentclass[11pt,a4paper]{article}
\usepackage[utf8]{inputenc}
\usepackage{amsmath}
\usepackage{amsfonts}
\usepackage[colorlinks,linkcolor=blue,anchorcolor=blue,citecolor=blue]{hyperref}
\usepackage{amssymb}
\usepackage{extarrows}
\usepackage{amsthm}
\usepackage{bm}
\usepackage{tikz}
\usetikzlibrary{arrows.meta}       
\usetikzlibrary{positioning}       
\usetikzlibrary{calc}              
\usetikzlibrary{shapes}  
\usepackage{verbatim}
\usepackage{hyperref}
\usepackage{cite}
\usepackage{enumerate}
\PassOptionsToPackage{normalem}{ulem}
\usepackage{ulem}
\usepackage[margin=1in]{geometry} 
\usepackage{array}
\usepackage{cases}
\usepackage{mathrsfs}
\usepackage{longtable}
\allowdisplaybreaks[4]

\newcolumntype{L}[1]{>{\raggedright\let\newline\\\arraybackslash\hspace{0pt}}m{#1}}
\newcolumntype{C}[1]{>{\centering\let\newline\\\arraybackslash\hspace{0pt}}m{#1}}
\newcolumntype{R}[1]{>{\raggedleft\let\newline\\\arraybackslash\hspace{0pt}}m{#1}}

\usepackage{natbib}
\usepackage{color}
\usepackage{dsfont}
\usepackage{marginnote}
\usepackage{enumitem}
\usepackage{graphicx}
\theoremstyle{plain}
\newtheorem{theorem}{\protect Theorem}[section]
\newtheorem{prop}[theorem]{\protect Proposition}
\newtheorem{definition}[theorem]{\protect Definition}

\newtheorem{lemma}[theorem]{\protect Lemma}
\newtheorem{remark}[theorem]{\protect Remark}
\newtheorem{ass}{\protect Assumption}
\newtheorem{corollary}[theorem]{\protect Corollary}

\def\d{\mathrm{d}}
\def\ie{\mathrm{i.e.}}
\def\as{\mathrm{a.s.}}
\def\RV{\mathrm{r.v.}}

\newcommand{\Lip}{\mathrm{Lip}}
\newcommand{\R}{\mathbb{R}}
\newcommand{\E}{\mathbb{E}}

\newcommand{\T}{\top}
\newcommand{\F}{\mathcal{F}}
\newcommand{\Fb}{\mathbb{F}}
\newcommand{\Pb}{\mathbb{P}}
\newcommand{\Pc}{\mathcal{P}}

\newcommand{\tr}{\mathrm{tr}}
\newcommand{\Law}{\mathcal{L}}

\newcommand{\Gb}{\mathbb{G}}
\newcommand{\G}{\mathcal{G}}
\newcommand{\Qb}{\mathscr{Q}}
\newcommand{\pa}{\partial}
\newcommand{\ps}{\mathscr{P}}
\newcommand{\M}{\mathcal{M}}
\newcommand{\supp}{\mathrm{supp}}
\newcommand{\U}{\mathscr{U}}
\newcommand{\tb}{\tilde{b}}
\newcommand{\ts}{\tilde{\sigma}}
\newcommand{\tg}{\tilde{\gamma}}
\newcommand{\tf}{\tilde{f}}
\newcommand{\tp}{\tilde{p}}
\newcommand{\tP}{\tilde{P}}
\newcommand{\tK}{\tilde{K}}
\newcommand{\tH}{\tilde{H}}
\newcommand{\lv}{\lVert}
\newcommand{\rv}{\rVert}
\newcommand{\J}{\mathcal{J}}
\newcommand{\tN}{\tilde{N}}

\newcommand{\assref}[1]{\hyperref[#1]{Assumption \ref*{#1}}}
\newcommand{\thmref}[1]{\hyperref[#1]{Theorem \ref*{#1}}}
\newcommand{\propref}[1]{\hyperref[#1]{Proposition \ref*{#1}}}
\newcommand{\remref}[1]{\hyperref[#1]{Remark \ref*{#1}}}
\newcommand{\lemref}[1]{\hyperref[#1]{Lemma \ref*{#1}}}
\newcommand{\defref}[1]{\hyperref[#1]{Definition \ref*{#1}}}
\newcommand{\corref}[1]{\hyperref[#1]{Corollary \ref*{#1}}}
\newcommand{\equref}[1]{\hyperref[#1]{(\ref*{#1})}}
\newcommand{\exaref}[1]{\hyperref[#1]{Example \ref*{#1}}}

\title{Extended mean-field control problems with Poissonian common noise: Stochastic maximum principle and Hamiltonian-Jacobi-Bellman equation}

\author{Lijun Bo \thanks{Email: lijunbo@ustc.edu.cn, School of Mathematics and Statistics, Xidian University, Xi'an, 710126, China.}
	\and
	Jingfei Wang \thanks{Email:wjf2104296@mail.ustc.edu.cn, School of Mathematical Sciences, University of Science and Technology of China, Hefei, 230026, China.}
	\and
	Xiaoli Wei \thanks{Email: xiaoli.wei@hit.edu.cn, Institute for Advance Study in Mathematics, Harbin Institute of Technology, Harbin, China.}
	\and
	Xiang Yu \thanks{Email: xiang.yu@polyu.edu.hk, Department of Applied Mathematics, The Hong Kong Polytechnic University, Kowloon, Hong Kong.}
}
\date{\vspace{-0.8cm}}

\begin{document}
	\maketitle
	
	\begin{abstract}
		This paper studies mean-field control problems with state-control joint law dependence and Poissonian common noise. We develop the stochastic maximum principle (SMP) and establish its connection to the Hamiltonian-Jacobi-Bellman (HJB) equation on the Wasserstein space. The presence of the conditional joint law and its discontinuity under Poissonian common noise bring new technical challenges. To develop the SMP when the control domain is not necessarily convex, we first consider a strong relaxed control formulation that allows us to perform the first-order variation. We propose the technique of extension transformation to overcome the compatibility issues arising from the joint law in the relaxed control formulation. By further establishing the equivalence between the relaxed control and the strict control formulations, we obtain the SMP for the original problem with strict controls. In the part to investigate the HJB equation, we formulate an equivalent controlled Fokker-Planck problem subjecting to a controlled measure-valued dynamics with Poisson jumps, which allows us to derive the HJB equation of the original problem under open-loop strict controls. We also establish the connection between the SMP and the HJB equation.

		\vspace{0.1in}
	\end{abstract}
	
	\section{Introduction}\label{sec:intro}
	
	Mean-field control (MFC) problem, also known as the optimal McKean-Vlasov control problem, has attracted great attention in recent years. This type of control problem is closely related to mean-field games (MFG) initially introduced by Larsy-Lions \cite{LL2007} and Huang-Caines-Malham\'e \cite{Huangetal2006} as both problems are used to approximate the asymptotic behavior of large population systems of {interacting agents}. On the other hand, it is well known that these two problems focus on different types of interactions and objectives. The MFG problem models the competitive interactions where each agent aims to maximize his own cost that results in the Nash equilibrium as the solution, while the MFC problem models the cooperative interactions where all agents jointly optimize the social optimum that leads to the optimal control by the social planner. The scope of this paper is to investigate MFC problems by featuring both the state-control joint law dependence and Poissonian common noise.

	Stochastic Maximum Principle (SMP) and Dynamic Programming Principle (DPP) are two fundamental methods in tackling a wide variety of optimal control problems in different contexts. In particular, SMP aims at establishing some necessary and sufficient conditions for optimality of control by using the techniques in calculus of variations. It states that any optimal control and the resulting controlled state process must solve the so-called extended Hamiltonian system comprised of an adjoint process in terms of the backward stochastic differential equation (BSDE),  the controlled state SDE, and the maximum condition. Bensoussan \cite{Bensoussan1981} derived a form of SMP via the first-order variation for controlled diffusion processes when the domain of control is convex. For systems of controlled diffusion and possibly non-convex control domain, Peng \cite{Peng1990} established a generalized SMP with two adjoint equations, usually referred as Peng's SMP, by utilizing the ``spike-variation method" and the second-order expansion. Bahlali~\cite{Bahlali} considered a relaxed (measure-valued) control formulation and established SMP through the first-order expansion when the control domain is not convex thanks to the fact that the space of relaxed controls is inherently convex. The SMP is further derived for problems with strict controls in Bahlali~\cite{Bahlali} by showing the equivalence between the relaxed and strict control formulations under certain assumptions. The methodology in Bensoussan \cite{Bensoussan1981} and Peng \cite{Peng1990} have also been generalized to cope with MFC problems in mean-field models recently. For example, \cite{AD2010,Buckdahn2011, MB2012, Djehicheetal2015} addressed MFC problems with coefficients depending on the moments of the population state distribution. Li~\cite{Li2012} explored MFC problems and the SMP in scalar interaction forms. Carmona and Delarue \cite{Carmona2015} solved a MFC problem with general dependence on the distribution of the state process. Bo et al. \cite{BLY} examined a linear quadratic (LQ) MFC problem with Brownian common noise by focusing on the Gamma convergence of the optimal controls from the N-player games to the mean-field model. Acciaio et al.~\cite{Carmona1} considered extended MFC problems involving the joint law of the controlled state process and the control process without common noise. Some previous studies, for instance \cite{Shenetal2014, Hafayed2014, Zhangetal2018}, also addressed MFC problems in jump diffusion or regime switching models without common noise. Recently, Nie and Yan~\cite{Nie2022} investigated extended MFC problems with partial observations.
	
	On the other hand, DPP is to decompose a global optimization problem into a series of recursive local optimization problems. Due to the distribution dependence, the value function needs to be defined on the lifted space of probability measures to recover the time consistency. The DPP in the mean-field model has been established in various contexts, such as  \cite{Lauriere2014, Bensoussan2015} under the assumption that the distribution of the state admits a density function; Pham and Wei \cite{pham2017} { in a Brownian common noise framework with semi-open-loop controls  adapted only to common noise}; \cite{Bayraktaretal2018, cosso2019} for MFC control problem with open-loop controls and coefficients relying on the marginal laws of the state and the control; Djete et al. \cite{Djeteetal2022} in a general framework of MFC with common noise, open-loop controls and coefficients depending on the joint conditional law of the path of the state-control; Cosso et al.~\cite{Cossoetal2023} for path-dependent Hilbert space valued MFC problem with open-loop controls and coefficients depend on the marginal laws of the state and the control. On the strength of the established DPP, one can apply some recent developments of stochastic calculus on flow of measures to derive and study the associated HJB equation on the Wasserstein space of probability measures. Different notions of solutions to the HJB equation and some existence results can be found in some recent studies. To name a few, we refer to \cite{BLPR2017, Chassagneux2022} for discussions on the classical solution and refer to \cite{Wuzhang2020, Burzoni2020, Soner2022, Bayraktar2023, Cossoetal2024, Zhou2024} for some investigations on the viscosity solutions on the Wasserstein space and different versions of comparison principles.

	All aforementioned studies only considered the common noise modeled by a Brownian motion while the idiosyncratic noise can either be a Brownian motion or a Poisson random measure.  Recently, \cite{HHRG2023, HHRG2024} introduced some interesting mean-field game problems where the common noise is driven by a Poisson random measure. Unlike the common Brownian noise, Poissonian common noise can effectively capture unexpected or rare events that simultaneously impact the entire system of players, leading to distinct phenomena and mathematical conclusions comparing with the existing results with common Brownian noise. Inspired by these MFG problems in \cite{HHRG2023, HHRG2024}, we are interested in MFC problems in the setting of Poissonian common noise from the social planner's perspective and aim to lay new theoretical foundations to develop SMP and HJB equation in the presence of both Poissonian common noise and the conditional joint law of state-control.

	Our first main contribution is to establish the SMP for the extended MFC problem in both relaxed and strict (open-loop) control formulations with a general control domain (not necessarily convex) under minimal assumptions. We generalize the methodology in Bahlali~\cite{Bahlali} to the mean-field model to cope with the joint law dependence and the Poissonian common noise. We first establish SMP in the strong relaxed control formulation (see \thmref{RNEC} and \thmref{RSUF}), which does not require the original domain of strict control to be convex and the first order variation can be exercised (see Subsection \ref{subsec:first-orderva}). However, contrary to classical single agent's control problems in Bahlali~\cite{Bahlali}, the dependence on the conditional joint law significantly hinders some standard arguments within the relaxed control formulation. In response, we propose several new key ingredients to overcome these issues, which are illustrated in the roadmap of our methodology in Figure \ref{SMP}. Firstly, we propose the \textit{extension transformation} to resolve the joint law formulation—specifically, transforming a mapping $h:\Pc_2(\R^n\times U)\to E$ to $\tilde h: \Pc_2(\R^n\times\Pc(U))\to E$ to ensure the consistency of the joint distribution in the relaxed formulation (see its definition in \eqref{eq:extensionh}) such that we can work with the joint conditional law of the state process and the relaxed control (instead of the strict control). Moreover, we demonstrate that this extension transformation maintains smooth properties of original functionals, see \lemref{extensioneq} and \lemref{lem:condexpect}. Secondly, we choose to work with the { linear functional derivative} of functionals on Banach product spaces equipped with suitable norms. Thirdly, we establish the key equivalence between the strong relaxed formulation and the strict control formulation by using the Chattering Lemma and other technical arguments; see \lemref{valueeq} and \lemref{equvalent}. As a consequence, we obtain the SMP in the strict control formulation and establish the sufficient and necessary condition on the optimal strict control as desired; see  \thmref{SNEC} and \thmref{SSUF}. 
	
	\begin{figure}
		\centering
		\scalebox{0.9}
		{
			\begin{tikzpicture}[
				box/.style={rectangle, draw=blue, thick, minimum width=3cm, minimum height=1.5cm, align=center},
				arrow/.style={->, thick},
				label/.style={black, font=\small, sloped, above}
				]
				
				\node[box] (strict) {Strict control formulation\\ \textcircled{1} non-convex control domain;\\
					\textcircled{2} state-control joint-law dependence};
				\node[box, below=6cm of strict] (relaxed) {Relaxed control formulation};
				\node[box, right=5.4cm of strict] (strictSMP) {SMP for strict control};
				\node[box, right=6cm of relaxed] (relaxedSMP) {SMP for relaxed control};
				
				\draw[arrow] (strict) -- node[label, above=5pt] {\parbox{5cm}{Extension Transformation}} (relaxed);
				\draw[arrow] (strict) -- node[label, below=5pt] {\parbox{5cm}{addressing the consistency in \textcircled{2}}} (relaxed);
				\draw[arrow] (relaxed) -- node[label, below=5pt] {first-order variation} (relaxedSMP);
				\draw[arrow] (relaxed) -- node[label, above=5pt] {convex set $\Pc(U)$ addressing \textcircled{1}} (relaxedSMP);
				\draw[arrow] (relaxedSMP) -- node[label, above=5pt] {\parbox{3cm}{chattering lemma}}  (strictSMP);
				
			\end{tikzpicture}
		}
		\caption{Our methodology for SMP}
		\label{SMP}
	\end{figure}

	Our second main contribution is the derivation of HJB equation via a controlled Fokker-Planck problem and its equivalence to the original MFC problem.  To the best of our knowledge, {Djete et al. \cite{Djeteetal2022} is the only work that discussed the HJB equation under Brownian common noise involving the conditional joint law. They first rigorously established the DPP in various formulations, and then derived the HJB equation in the general strong formulation with state-dependent coefficients}.  Our problem involves both state-control joint-law dependence and Poissonian common noise. Unlike existing studies, the presence of Poissonian common noise gives rise to the discontinuity of the conditional law of the state, thereby posing new challenges in deriving the HJB equation. In particular, it becomes crucial to understand how the jumps of conditional law of state affect the value function. Motivated by Motte and Pham \cite{Pham}, which addressed the extended MFC problems in the discrete time setting with common noise and open-loop controls, we first construct a new  controlled Fokker-Planck problem with kernel-valued controls only adapted to Poissonian common noise filtration (see \eqref{equ:extended-FP-mu_t}-\eqref{liftMFC-value}). We reveal that the conditional state law has jumps whose sizes are characterized by {\it generalized} measure shifts in terms of some adjoint operator; see \lemref{measure_shift} and~\eqref{extended-operatorA_0*}. This fact, combined with the It\^o's formula on flows of conditional probability measures (see \lemref{Ito}), leads to an associated HJB equation on the space of probability measures in \eqref{extended-HJB-1-poisson-common-noise}. Our main objective is to show that the value function of the original problem {coincides with} that of the controlled Fokker Planck problem.  The roadmap of our method is illustrated in Figure~\ref{HJB}. In Subsection \ref{LCP}, we prove that the value function of the original problem might be larger than that of the controlled Fokker-Planck problem (see \lemref{lem:inequality-from-op-fp}).  The idea is mainly based on the disintegration of the joint conditional law by Bayes' formula. This allows us to  find a corresponding kernel-valued control of the controlled Fokker-Plank problem for every open-loop strict control of the original problem.  The challenge is the way back from the kernel-valued controls to the open-loop strict controls.  In Subsection \ref{sec:the controlled Fokker-Planck problem-less-original}, our strategy is to first prove the verification theorem for the controlled Fokker-Planck problem, which allows us to focus on near-optimal feedback kernel-valued controls for the controlled Fokker-Planck problem (see \propref{thm:verification}). Moreover, when the near optimal feedback optimal policy takes the form of the lifted randomized policy (see \defref{randomized_feedback}), we  demonstrate that there exists a sequence of piecewise constant controls such that the error between two associated cost functionals diminishes as the size of the time grid tends to zero (see \propref{prop:diff-ol-fp}).   As a consequence, in \thmref{equivalencevalue}, we show that the value functions of these two problems coincide, and both are given by the classical solution of the HJB equation, by combining  \lemref{lem:inequality-from-op-fp},  \propref{thm:verification}, and \propref{prop:diff-ol-fp}. As a byproduct, the value function of the original problem satisfies the {\it conditional law invariance} property; see \remref{vlsequivalence}.  We highlight that our new approach, based on the new controlled Fokker-Planck problem, can effectively resolve all issues caused by the discontinuity of the joint law and is applicable to more general settings beyond the Poissonian common noise such as L\'{e}vy common noise. Moreover, we also show in the present paper that the derived HJB equation differs substantially from the counterpart in jump diffusion models without common noise.  
	
	Finally, our third contribution is the connection of HJB equation to SMP. When the HJB equation has a smooth solution (which coincides with the value function under a measurable assumption), we verify that the solution to the adjoint BSDE can be expressed in terms of the value function and its derivatives with respect to the probability measures along the optimal state trajectory, and the connection between the SMP and HJB equation holds; see \thmref{smp_HJB}. We further illustrate this connection with explicit results in a LQ-type extended MFC problem with Poissonian common noise, in which we can fully understand how the Poissonian common noise affects the forms of the derived BSDE and the HJB equation.
	
	\begin{figure}
		\centering
		\scalebox{0.9}
		{
			\begin{tikzpicture}[
				box/.style={rectangle, draw=blue, thick, minimum width=4cm, minimum height=2cm, align=center},
				arrow/.style={thick, ->, >=Stealth},
				every label/.style={black, font=\footnotesize}
				]
				
				\node[box] (strict1) {Original problem (\equref{extended-MFC-dynamics}-\equref{valuefunc})\\
					\textcircled{1} Open-loop strict controls\\
					\textcircled{2}  Piecewise constant controls};
				\node[box, right=4cm of strict1] (strict2) {Controlled Fokker-Planck\\
					problem (\eqref{equ:extended-FP-mu_t}-\eqref{liftMFC-value})\\
					\textcircled{1} kernel-valued controls\\
					\textcircled{2}near optimal lifted randomized \\
					feedback policies};
				
				\draw[arrow] ([yshift=5mm]strict1.east) --  node[above] {Section~\ref{sec:original-less-the controlled Fokker-Planck problem} (``$\geq$")} ([yshift=5mm]strict2.west);
				\draw[arrow] ([yshift=-5mm]strict2.west) --  node[below] {Section~\ref{sec:the controlled Fokker-Planck problem-less-original}(``$\leq$")} ([yshift=-5mm]strict1.east);
				
			\end{tikzpicture}
		}
		\caption{Our methodology for HJB equation}
		\label{HJB}
	\end{figure}
	
	The rest of the paper is organized as follows. Section \ref{sec:model} introduces some technical preparations and two formulations of the extended MFC problems with Poissonian common noise in both the strict and relaxed sense. In Section \ref{sec:SMP}, we first develop the SMP for the extended MFC problems with Poissonian common noise in the relaxed formulation using the first order variation. We then establish the equivalence result between two different formulations to drive the SMP for the strict extended MFC problems. Section \ref{sec:HJB} derives the HJB equation of the extended MFC problem and reveals its connection to the SMP. Section \ref{sec:example} studies a LQ-type example of the MFC problem with Poissonian common noise using both the SMP and the HJB equation methods, and further illustrates the connections between these two methods. Some proofs of auxiliary results in previous sections are reported in Appendices \ref{appendix_for_sec2}, \ref{appendix} and \ref{sec:appendix-B}.

	\vspace{0.1in}
	\noindent{\it Notations.}\quad We list below some notations that will be used frequently throughout the paper:
	\vspace{-0.2in}
	\begin{center}
		\begin{longtable}{l l}
			$T\in(0,\infty)$ & Terminal time horizon \\
			$\tt m$ & The Lebesgue measure on $\R$\\
			${\rm Lip}_1(E)$ & Set of Lipschitz continuous functions on $E$ with Lipschitz coefficient\\
			& no more than $1$\\
			${\rm Lip}_b(E)$ & Set of Lipschitz continuous functions on $E$ { with bounded supremum}\\ &{norm}\\
			$|\cdot|$ & Euclidean norm on $\R^n$\\
			$a\cdot b$ & Dot product of vectors $a\in\R^n$ and $b\in\R^n$\\
			$C_b(\R^n)$ & Set of bounded and continuous real-valued functions on $\R^n$\\
			$C(I;E)$ & Set of $E$-valued continuous functions defined on $I$\\
			$L^2((A,\mathcal{B}(A),\lambda_A);E)$ & Set of square-integrable $E$-valued random variable (r.v.) defined on\\
			& measure space  $(A,\mathcal{B}(A),\lambda_A)$. We shall abbreviate it as $L^2(A; E)$ \\
			&or $L^2(A)$ when there is no confusion. \\
			${\cal P}(E)$ (${\cal P}_p(E)$)  & Set of probability measures on $E$ (with finite $p$-order moments)\\
			$d_{p,E}$ & $p$-order Wasserstein metric on $\Pc_p(E)$\\
			$\partial_x$ & Partial derivative w.r.t. the argument $x\in\R$\\
			$\nabla_x$ ($\nabla^2_x$) & Gradient (Hessian) operator w.r.t. the argument $x\in\R^n$\\
			${\cal L}^P(\xi)$ & Law of a $E$-valued random variable (r.v.) $\xi$ under $P$.\\
			&We will drop the superscript and abbreviate it as $\Law(\xi)$ when $P=\Pb$\\
			${\cal L}^P(\xi|{\cal G})$ & Regular conditional law of a $E$-valued r.v. $\xi$ given a $\sigma$-field ${\cal G}$ under $P$.\\
			&We will drop the superscript and abbreviate it as $\Law(\xi|\G)$ when $P=\Pb$\\
			$\E^P$ & Expectation operator under probability measure $P$\\
			&and we shall abbreviate it $\E^{\Pb}(\E^{\Pb'})$ as $\E(\E')$ respectively\\
		\end{longtable}
	\end{center}
	
	\section{Setup and Problem Formulation}\label{sec:model}
	
	In this section, we first introduce two formulations of the extended MFC problem with Poissonian common noise, namely the strict control formulation and the (strong) relaxed control formulation. The main feature of the extended MFC problem lies in the presence of the joint law of the state and control in the controlled McKean-Vlasov dynamics and the objective functional. More precisely, in the strict control formulation of our MFC problem, the mean-field interaction is captured by
	the joint law of the state and strict control that belongs to $\Pc(\R^n\times U)$, where $U$ is referred to as control space. In the relaxed control formulation, we consider the joint law of the state and relaxed control that belongs to $\Pc(\R^n\times\Pc(U))$.  This essential difference motivates us to construct an extension transformation from a mapping $h:\Pc(\R^n\times U)\mapsto\R$ to the corresponding mapping $\tilde h:\Pc(\R^n\times \Pc(U))\mapsto\R$, which can preserve some smooth properties of $h$. This key extension plays an important role in our analysis, which will be introduced with details in \eqref{eq:affinemaping}-\eqref{eq:extensionh}.
	
	\subsection{Basic spaces and the extension transformation}\label{subsecB}
	
	Let $(B,\|\cdot\|)$ be a Banach space and $U\subset B$ be a compact subset. Define the product space $K:=\R^n\times B$, endowed with the product norm $\|(x,u)\|_K:=|x|+\|u\|$ for $(x,u)\in K$. Thus,  $(K,\lv\cdot\rv_K)$ is again a Banach space. Denote by $\mathcal{M}(U)$ the set of finite signed Radon measures on the measurable space $(U,\mathscr{B}(U))$, {equipped with the Fortet-Mourier norm} 
	\begin{equation*}
		\|q\|_{U,{\rm FM}}:=\sup_{\substack{f\in\Lip_1(U)\\ \| f\|_{\infty}\leq 1}} \int_U f(u)q(\d u),\quad \forall q\in\mathcal{M}(U),
	\end{equation*}
	where $\|f\|_{\infty}:=\sup_{u\in U}|f(u)|$ for $f:U\mapsto\R$ in $C(U;\R)$. Then $(\M(U),\|\cdot\|_{U,{\rm FM}})$ is a separable Banach space. Let $d_{U,{\rm FM}}$ be the metric induced by the norm $\|\cdot\|_{U,{\rm FM}}$ on $\Pc(U)$, and it follows that $(\Pc(U),d_{U,{\rm FM}})$ is a compact Polish space. Here, we consider $\M(U)$ instead of $\Pc(U)$ directly because we are going to define partial $L$-derivative with respect to $q\in\M(U)$.
	
	Let $V:=\R^n\times\M(U)$ be a separable Banach product space  equipped with the product norm $\|\cdot\|_{V}$ defined by $\|(x,q)\|_V:=|x|+\|q\|_{U,{\rm FM}}$ for $(x,q)\in V$. For any $\xi_1,\xi_2\in{\cal P}_2(V)$, consider the Kantorovich-Rubinstein metric (which is equivalent to 1-Wasserstein metric according to Kantorovich duality) 
	\begin{equation}\label{eq:metric-dKR}
		d_{\rm KR}(\xi_1,\xi_2)=\sup_{f\in\Lip_1(V)}\left(\int_Vf(x,q)\xi_1(\d x,\d q)-\int_Vf(x,q)\xi_2(\d x,\d q)\right).
	\end{equation}
	Then, $(\Pc_2(V),d_{\rm KR})$ is a Polish space and we have the next result.
	
	\begin{lemma}\label{metriceq}
		For any $\xi_1,\xi_2\in\Pc_2(\R^n\times\Pc(U))\subset \Pc_2(V)$ (one can identify $\xi\in\Pc_2(\R^n\times\Pc(U))$ as an element in $\Pc_2(V)$ with support in $\R^n\times\Pc(U)$), it holds that
		\begin{equation*}
			d_{\rm KR}(\xi_1,\xi_2)=\sup_{f\in\Lip_1(\R^n\times\Pc(U))}\left(\int_{\R^n}\int_{\Pc(U)}f(x,q)\xi_1(\d x,\d q)-\int_{\R^n}\int_{\Pc(U)}f(x,q)\xi_2(\d x,\d q)\right).
		\end{equation*}
	\end{lemma}

	Let $\mathcal{M}(K)$ be the set of finite signed Radon measures on the measurable space $(K,\mathscr{B}(K))$ equipped with the following Fortet-Mourier norm:
	\begin{equation*}
		\| \rho\|_{K,{\rm FM}}=\sup_{\substack{f\in\Lip_1(K)\\ \lv f\rv_{\infty}\leq 1}}\int_K f(w)\rho(\d w),\quad\forall\rho\in\mathcal{M}(K).
	\end{equation*}
	
	In the sequel, we denote by $x,u,q,\rho,\xi$ the generic elements in $\R^n,U,\M(U),\M(K),\Pc_2(V)$ respectively, and denote by $\mu,\nu$ the marginals of $\xi$. Note that $\Pc_2(\R^n\times U)\subset\Pc_2(K)\subset\M(K)$. We endow $\Pc_2(\R^n\times U)$ with the metric induced by $\lv\cdot\rv_{K,{\rm FM}}$, and hence it becomes a Polish space. Similar to \lemref{metriceq}, the topology induced by this metric is equivalent to the weak convergence topology on $\Pc_2(\R^n\times U)$. Because $U$ is compact, it holds that, for all $\xi\in\Pc_2(\R^n\times\Pc(U))$,
	\begin{equation*}
		\int_{\mathcal{M}(U)}q(\d u)\xi(\d x,\d q)=\int_{\Pc(U)}q(\d u)\xi(\d x,\d q)\in\Pc_2(\R^n\times U).
	\end{equation*}
	We can then define an affine mapping $\ps:\Pc_2(\R^n\times\Pc(U))\mapsto\Pc_2(\R^n\times U)$ as a Bochner integral 
	\begin{equation}\label{eq:affinemaping}
		\ps(\xi)(\d x,\d u):=\int_{\mathcal{M}(U)}q(\d u)\xi(\d x,\d q),\quad \forall \xi\in\Pc_2(\R^n\times\Pc(U)).
	\end{equation}
	
	For any mapping $h:\mathcal{P}_2(\R^n\times U)\to\R$, we define $\tilde{h}:\Pc_2(\R^n\times\Pc(U))\mapsto\R$ as its \textit{extension transformation} to $\mathcal{P}_2(\R^n\times\Pc(U))$ in the sense that
	\begin{equation}\label{eq:extensionh}
		\tilde{h}(\xi):=h(\ps(\xi)),\quad\forall\xi\in\Pc_2(\R^n\times {\cal P}(U)).
	\end{equation}
	Then, we have the next results.
	\begin{lemma}\label{extensioneq}
		If $h:{\cal P}_2(\R^n\times U)\mapsto\R$ is Lipschitz continuous, so is its extension transformation $\tilde{h}:{\cal P}_2(\R^n\times {\cal P}(U))\mapsto\R$  defined by \eqref{eq:extensionh}.
	\end{lemma}
	
	\begin{lemma}\label{lem:condexpect}
		Let $(\Omega,\F,\mathbb{P})$ be a given probability space and $\G\subset\F$ be a sub-$\sigma$-algebra. Then, for any $\R^n$-valued square-integrable random variable $X$ and $U$-valued random variable $\alpha$, it holds that, $\Pb$-a.s.
		\begin{equation*}
			\tilde{h}(\Law((X,\delta_{\alpha})|\G))=h(\Law((X,\alpha)|\G)),
		\end{equation*}
		where $\delta_{\alpha}$ denotes the Dirac measure concentrated on $\alpha$.
	\end{lemma}

	\subsection{{ linear functional derivatives} w.r.t. measures of the extension}
	
	We adopt the definition of { linear functional derivative} in Banach space  in Buckdahn et al.~\cite{Buckdahn}.
	\begin{definition}\label{def:derivativelinear}
		Let	$(\mathcal{K},\lv\cdot\rv_{\mathcal{K}})$ be a Banach space and $I\subset\Pc_2(\mathcal{K})$ be a convex subset. We say that a mapping $h:I\mapsto\R$ has a { linear functional derivative} $\frac{\delta h}{\delta m}:I\times\mathcal{K}\mapsto\R$, if $\frac{\delta h}{\delta m}$ is a continuous function over $I\times\mathcal{K}$ such that, for all $m,m'\in I$,
		\begin{equation*}
			h(m')-h(m)=\int_0^1\int_{\mathcal{K}}\frac{\delta h}{\delta m}(m+\lambda(m'-m),y)(m'-m)(\d y)\d\lambda.
		\end{equation*}
		Moreover, there exists a constant $C>0$ such that $\left|\frac{\delta h}{\delta m}(m,y)\right|\leq C(1+\lv y\rv_{\mathcal{K}}^2)$ for all $y\in\bigcup\limits_{m\in I}\supp(m)$.
	\end{definition}
	
	Then, we have the following result regarding the extension transformation.
	
	\begin{lemma}\label{extensionLdif}
		Assume the existence of the { linear functional derivative} $\frac{\delta h}{\delta\rho}:\Pc_2(\R^n\times U)\times K\mapsto\R$ for a given mapping $h:\Pc_2(\R^n\times U)\mapsto\R$. Then, the associated extended mapping $\tilde{h}:\Pc_2(\R^n\times\Pc(U))\subset\Pc_2(V)\mapsto\R$ in \eqref{eq:extensionh} also has a { linear functional derivative} $\frac{\delta \tilde{h}}{\delta \xi}:\Pc_2(\R^n\times\Pc(U))\times V\mapsto \R$ such that, for all $(x,q)\in V$,
		\begin{equation}\label{eq:chain}
			\frac{\delta \tilde{h}}{\delta \xi}(\xi)(x,q)=\int_U \frac{\delta h}{\delta\rho}(\ps(\xi))(x,u)q(\d u),
		\end{equation}
		where $\ps:\Pc_2(\R^n\times\Pc(U))\mapsto\Pc_2(\R^n\times U)$ is the affine mapping given in \eqref{eq:affinemaping}.
	\end{lemma}

	In what follows, we consider a mapping $h:\Pc_2(\R^n\times U)\mapsto\R$ such that its { linear functional derivative}  $\frac{\delta h}{\delta\rho}$ exists. We also assume that the partial derivative $\pa_x(\frac{\delta h}{\delta\rho}(\rho)(x,u))$ w.r.t. the argument $x$ exists, is continuous and of at most linear growth, i.e., there exists a constant $C>0$ such that
	\begin{equation*}
		\left|\pa_x\left(\frac{\delta h}{\delta\rho}(\rho)(x,u)\right)\right|\leq C(1+|x|),\quad \forall (x,u)\in \R^n\times U.
	\end{equation*}
	The next result is a direct consequence of \lemref{extensionLdif}.
	
	\begin{lemma}\label{extensiondif}
		The { linear functional derivative} $\frac{\delta \tilde{h}}{\delta \xi}$ is Frech\'{e}t differentiable, and the vector of partial derivatives:
		\begin{equation*}
			\pa\frac{\delta \tilde{h}}{\delta \xi}=\left(\pa_x\frac{\delta \tilde{h}}{\delta \xi},\pa_q\frac{\delta \tilde{h}}{\delta \xi}\right):\Pc_2(\R^n\times\Pc_2(U))\times V\to V^*:=\R^n\times\M^*(U)  \end{equation*}
		is continuous. Furthermore, it holds that
		\begin{equation*}
			\begin{cases}
				\displaystyle\pa_x\left(\frac{\delta \tilde{h}}{\delta \xi}(\xi)(x,q)\right)=\int_U\pa_x\left(\frac{\delta h}{\delta\rho}(\mathscr{P}(\xi))(x,u)\right)q(\d u)\in\R^n;\\
				\displaystyle\pa_q\left(\frac{\delta \tilde{h}}{\delta \xi}(\xi)(x,q)\right)(\cdot)=\int_U\frac{\delta h}{\delta\rho}(\mathscr{P}(\xi))(x,u)(\cdot)(\d u)=\frac{\delta\tilde h}{\delta \xi}(\xi)(x,\cdot)\in\M^*(U).
			\end{cases}
		\end{equation*}
		In particular, there exists a constant $C>0$ such that, for all $(x,q)\in\R^n\times\Pc(U)$,
		\begin{equation*}
			\left|\pa_x\left(\frac{\delta \tilde{h}}{\delta \xi}(\xi)(x,q)\right)\right|\leq C(1+|x|).
		\end{equation*}
	\end{lemma}
	Notably, by \lemref{extensiondif}, the differentiability of $\frac{\delta\tilde h}{\delta\xi}$ w.r.t. $q\in\Pc(U)$ after extension does not require the differentiability of $\frac{\delta h}{\delta\rho}$ w.r.t. $u\in U$.  Hence, the extension lifts the differentiablity in this sense.
	
	Similar to Proposition 4.1 in  Buckdahn et al.~\cite{Buckdahn}, we have the following result.
	\begin{lemma}\label{diff}
		Let $(\Omega,\F,\Pb)$ be a probability space. For any $(X,q)\in L^2((\Omega,\F,\Pb);\R^n\times\Pc(U))$, as $\epsilon\downarrow0$, it holds that
		\begin{equation*}
			\begin{aligned}
				\tilde{h}(\Law((X,q)&+\epsilon(X'-X,q'-q)))-\tilde{h}(\Law(X,q))=\E\left[\pa_x \left(\frac{\delta \tilde{h}}{\delta \xi}(\Law(X,q))(X,q)\right)\epsilon\cdot (X'-X)\right]\\
				&+\E\left[\int_U\frac{\delta h}{\delta \rho}(\mathscr{P}(\Law(X,q)))(X,q)\epsilon(q'-q)(\d u)\right]+o(\epsilon),
			\end{aligned}
		\end{equation*}
		for all $(X',q')\in L^2((\Omega,\F,\Pb);\R^n\times\Pc(U))$.  	
	\end{lemma}
	
	We also give the definition of the (partial) $L$-derivative as below.
	
	\begin{definition}\label{Ldif}
		The partial $L$-derivative of the mapping $h:\Pc_2(\R^n\times U)\mapsto\R$ w.r.t. the probability measure $\mu\in{\cal P}_2(\R^n)$ is defined by, for all $(\rho,x,u)\in\Pc_2(\R^n\times U)\times \R^n\times U$,
		\begin{equation*}
			\pa_{\mu}h(\rho)(x,u):=\pa_x\frac{\delta h}{\delta\rho}(\rho)(x,u).
		\end{equation*}
		Similarly, we can also define the partial $L$-derivative w.r.t. the probability measure $\nu\in\Pc_2(\M(U))$ for the mapping $\tilde h:\Pc_2(\R^n\times\Pc(U))\mapsto\R$. The partial $L$-derivative of $\tilde h$ with respect to $\nu$ is defined by, for all $(\xi,x,q)\in\Pc_2(\R^n\times\Pc(U))\times \R^n\times \Pc(U)$,
		\begin{equation*}
			\pa_{\nu}\tilde h(\xi)(x,q):=\pa_q\frac{\delta\tilde h}{\delta\xi}(\xi)(x,q).
		\end{equation*}
	\end{definition}
	
	Note that we are not using the classical definition of $L$-differentiability here (more precisely, define the $L$-derivative via lifting functions) because we are considering the joint law involving the law of controls such that the classical methods are not applicable. It follows from \lemref{extensiondif} that, the (partial) $L$-derivative of the extension of $h$ is given by
	\begin{align}\label{eq:tildehmu=}
		\pa_{\mu}\tilde h(\xi)(x,q)&=\pa_x\frac{\delta\tilde h}{\delta \xi}(\xi)(x,q)=\int_U\pa_{\mu}h(\mathscr{P}(\xi))(x,u)q(\d u),\\
		\pa_{\nu}\tilde h(\xi)(x,q)(\cdot)&=\pa_q\frac{\delta\tilde h}{\delta \xi}(\xi)(x,q)(\cdot)=\int_U\frac{\delta h}{\delta\rho}(\mathscr{P}(\xi))(x,u)(\cdot)(\d u).
	\end{align}

	\subsection{Formulations of extended MFC with Poissonian common noise}\label{sec:extendedMFCPoiCommon}
	
	
	{ Consider a  standard $d$-dimensional Brownian motion $W=(W_t)_{t\in[0,T]}$ defined on the probability space $(\Omega^1,\F^1,\Pb^1)$ and a Poisson random measure $N(\d z,\d t)$ on some measurable space $(Z,\mathscr{Z})$ with intensity $\lambda(\d z)\d t$ satisfying $\lambda(Z)<\infty$ on the probability space $(\Omega^2,\F^2,\Pb^2)$. Define
		\begin{equation}\label{eq:probabspace}
			\Omega=\Omega^1\times\Omega^2,\quad \F=\F^1\otimes\F^2,\quad \Pb=\Pb^1\otimes\Pb^2.
		\end{equation}
		To simplify the notations, we denote by $W$ and $N$ the natural extensions of $W$ and $N$ to $\Omega$, respectively. 
		Let $\mathcal{H}_0\subset\mathcal{F}^1$ be a $\sigma$-field independent of $W$ that is sufficiently rich in the sense that for any $\mu\in\mathcal{P}_2(\mathbb{R}^n)$, there exists an $\mathcal{H}_0$-measurable random variable $X$ such that $\mathcal{L}^{\mathbb{P}^1}(X)=\mu$.
		
		Consider the filtration
		$\Fb=(\F_t)_{t\in[0,T]}$ by $\F_t=\mathcal{H}_t\otimes\F_t^N$ with $\mathcal{H}_t=\mathcal{H}_0\vee\sigma(W_s;s\leq t)\subset \F^1$ and $\F_t^N=\sigma(N((0,s]\times A);s\leq t,A\in \mathscr{Z})\subset\F^2$.}
	Let $\Gb=(\G_t)_{t\in[0,T]}$ be the natural extensions of $\Fb^N=(\F_t^N)_{t\in[0,T]}$ to $\Omega$, and also denote by $\mathbb{H}=(\mathcal{H}_t)_{t\in [0,T]}$ the natural extension of $(\mathcal{H}_t)_{t\in [0,T]}$.
	
	We also assume that $(\Omega,\F,\Fb,\Pb)$, $(\Omega,\F,\mathbb{H},\Pb)$ and $(\Omega,\F,\Gb,\Pb)$ satisfy the usual conditions without loss of generality. By construction, we can deduce that $N$ is $\Pb$-independent of $\mathcal{H}_t$ for all $t\in [0,T]$. Therefore, for any $\Fb$-adapted process $Y=(Y_t)_{t\in[0,T]}$, it holds that, $\Pb$-a.s.
	\begin{equation}\label{compatible}
		\Law(Y_t|N)=\Law(Y_t|\G_t),\quad \forall t\in[0,T],
	\end{equation}
	where the notation $\Law(\cdot|N)$ stands for the conditional distribution given the Poisson random measure $N$ under the probability measure $\Pb$.
	
	Consider measurable functions $b:\R^n\times\Pc_2(\R^n\times U)\times U\mapsto\R^n$, $\sigma:\R^n\times\Pc_2(\R^n\times U)\times U\mapsto\R^{n\times d}$, $\gamma:\R^n\times\Pc_2(\R^n\times U)\times U\times Z\mapsto\R^n$, $f:\R^n\times\Pc_2(\R^n\times U)\times U\mapsto\R^{n}$ and $g:\R^n\times\Pc_2(\R^n)\mapsto\R$ as the coefficients of the underlying controlled state process and the objective functional. We impose the following assumptions throughout the paper.
	
	\begin{ass}\label{ass} We make the following assumptions on model coefficients:
		\begin{itemize}
			\item[{\rm(A.1)}] $b(x,\rho,u),\sigma(x,\rho,u),\gamma(x,\rho,u,z),f(x,\rho,u),g(x,\mu)$ are continuous functions on their domains of definition; $b(x,\rho,u),\sigma(x,\rho,u),\gamma(x,\rho,u,z)$
			are uniformly Lipschitz continuous in $(x,\mu)$ in the sense that, there exists a constant $L>0$ independent of { $(u,z)\in U\times Z$} such that, for all $(x,\rho),(x',\rho')\in\R^n\times\Pc_2(\R^n\times U)$,
			\begin{equation*}
				\left|\phi(x',\rho',u)-\phi(x,\rho,u)\right|\leq L\left(|x-x'|+\|\rho-\rho'\|_{K,\rm FM}\right),
			\end{equation*}
			where $\phi=b(\cdot),\sigma(\cdot)$ or $\gamma(\cdot,z)$.
			
			\item[{\rm(A.2)}] the partial derivatives $\nabla_xb,\nabla_x\sigma,\nabla_x\gamma$ are uniformly bounded and continuous. $\nabla_xf,\nabla_xg$ have at most uniform linear growth in $x$ and continuous.
			
			\item[{\rm(A.3)}] there exists a constant $K>0$ such that, for all $(x,\rho,u)\in \R^n\times\Pc_2(\R^n\times U)\times U$,
			\begin{equation*}
				|\phi(x,\rho,u)|\leq K\left(1+|x|+M_2(\rho)\right),
			\end{equation*}
			with $\phi=b(\cdot),\sigma(\cdot)$, $f(\cdot)$ or $\gamma(\cdot,z)$ and $M_2(\rho):=(\int_{\R^n\times U} (|x|^2+\lv u\rv^2)\rho(\d x,\d u))^{\frac{1}{2}}$. Moreover, it holds that $
			\int_Z|\gamma(x,\rho,u,z)|^2\lambda(\d z)\leq L(1+|x|^2+M_2(\rho)^2)$.
			
			\item[{\rm(A.4)}] the coefficients $b,\sigma,\gamma$ have { linear functional derivative}s $\frac{\delta b}{\delta\rho},\frac{\delta \sigma}{\delta\rho},\frac{\delta \gamma}{\delta\rho}$, {respectively}. These { linear functional derivative}s have bounded continuous partial derivatives w.r.t. the state variables. That is, the L-derivatives of $b,\sigma,\gamma$ w.r.t. $\mu$ defined in \defref{Ldif} are bounded and continuous.
			
			\item[{\rm(A.5)}] the coefficients $f,g$ have { linear functional derivative}s $\frac{\delta f}{\delta\rho},\frac{\delta g}{\delta\mu}$, respectively. These { linear functional derivative}s have continuous partial derivatives w.r.t. the state variables, $\ie$ $f,g$ admit continuous $L$-derivative w.r.t. $\mu$. Furthermore, the growth condition holds for some constant $L>0$ independent of $(x,\rho,u)\in\R^n\times\Pc_2(\R^n\times U)\times U$,
			{\small\begin{align*}
					\left(\int_{\R^n\times U}|\pa_{\mu}f(x,\rho,u)(x',u')|^2\rho(\d x',\d u')\right)^{\frac12}+\left(\int_{\R^n}|\pa_{\mu}g(x,\mu)(x')|^2\mu(\d x')\right)^{\frac12}\leq L(1+|x|+M_2(\rho)).
			\end{align*}}
		\end{itemize}

	\end{ass}

	We are at the position to formulate the extended MFC problem with Poissonian common noise in both strict and relaxed senses. Let $\U$ be the set of $\Fb$-adapted process $\alpha=(\alpha_t)_{t\in[0,T]}$ taking values in $U$. We first introduce the extended MFC problem with Poissonian common noise in the strict control formulation, which is given by
	\begin{equation}\label{eq:svalue1}
		J(\alpha):=\E\left[\int_0^Tf(X_t,\Law((X_t,\alpha_t)|\G_t),\alpha_t)\d t+g(X_T,\Law(X_T|\G_T))
		\right]
	\end{equation}
	subject to the constraint:
	\begin{equation}\label{eq:seq1}
		\begin{cases}
			\displaystyle \d X_t =b(X_t,\Law((X_t,\alpha_t)|\G_t),\alpha_t)\d t+\sigma(X_t,\Law((X_t,\alpha_t)|\G_t),\alpha_t)\d W_t\\
			\displaystyle \quad\qquad+\int_Z\gamma(X_{t-},\Law((X_{t-},\alpha_{t-})|\G_{t-}),{\alpha_{t-}},z)\tN(\d t,\d z),\\
			\displaystyle \Law(X_0)=\mu\in{{\cal P}_p(\R^n).}
		\end{cases}
	\end{equation}
	Here,  {$X_0\in\mathcal{H}_0$, $p>2$ and} $\tN(\d z,\d t):=N(\d z,\d t)-\lambda(\d z)\d t$ is the compensated Poisson random measure. It is not difficult to show that problem \eqref{eq:svalue1}-\eqref{eq:seq1} is well-defined under \assref{ass}. An adapted process $\alpha^*\in\U$ is an optimal (strict) control of problem \eqref{eq:svalue1}-\eqref{eq:seq1} if it holds that $J(\alpha^*)=\inf_{\alpha\in\U}J(\alpha)$.

	
	Next, we consider the relaxed control formulation of the extended MFC problem. { Denote by $\mathcal{Q}$ the collection of finite measures on $[0,T]\times U$ with the first marginal measure equal to the Lebesgue measure, $\ie$, every element $q\in\mathcal{Q}$ can be identified with a measurable mapping $[0,T]\in t\mapsto q_t\in \Pc(U)$, defined up to $\as$ by $q(\d t,\d u)=q_t(\d u)\d t$.  In the sequel, we will always refer to the measurable mapping $q=(q_t)_{t\in [0,T]}$ to an element in $\mathcal{Q}$. Moreover, we endow $\mathcal{Q}$ with the weak convergence topology, $\ie$, $q^n\to q$ in $\mathcal{Q}$ if and only if for every $f\in C([0,T]\times U)$, the convergence $\int_0^T\int_Uf(t,u)q_t^n(\d u)\d t\to\int_0^T\int_Uf(t,u)q_t(\d u)$ holds.
		
		We call a $\mathcal{Q}$-valued $\RV$ $q=(q_t)_{t\in [0,T]}$ an admissible relaxed control if $q_t$ is $\mathbb{F}$-adapted (such adaptedness can be upgraded to progressive measurability due to Lemma 3.2 in Lacker \cite{Lacker}) and denote by $\Qb$ the collection of all admissible relaxed controls.  } Furthermore, we also define
	\begin{equation}\label{eq:deltaU}
		\delta(\mathscr{U}):=\left\{q=(q_t)_{t\in[0,T]}~\text{with}~q_t=\delta_{\alpha_t};~\alpha=(\alpha_t)_{t\in[0,T]}\in\mathscr{U}\right\}.
	\end{equation}
	Then, it automatically holds that $\delta(\mathscr{U})\subset\Qb$.
	
	Let $\tb,\ts,\tg,\tf$ be the respective extensions of the coefficients $b,\sigma,\gamma,f$ according to \eqref{eq:extensionh}. Then, under \assref{ass}, by applying \lemref{extensioneq}, \lemref{lem:condexpect}, \lemref{extensionLdif} and \lemref{extensiondif}, we can obtain the corresponding extensions $\tb,\ts,\tg,\tf$ satisfy the following properties:
	\begin{itemize}
		\item[{\rm(B.1)}] $\tb(x,\xi,u),\ts(x,\xi,u),\tg(x,\xi,u,z),\tf(x,\xi,u),g(x,\mu)$ are continuous functions on their domains of definition; $\tb(x,\xi,u),\ts(x,\xi,u),\tg(x,\xi,u,z)$
		are uniformly Lipschitz continuous in $(x,\xi)$ in the sense that, there is a constant $L>0$ independent of { $(u,z)\in U\times Z$} such that, for all $(x,\xi),(x',\xi')\in\R^n\times\Pc_2(\R^n\times \Pc(U))$,
		\begin{equation*}
			\left|\phi(x',\xi',u)-\phi(x,\xi,u)\right|\leq L(|x-x'|+d_{\rm KR}(\xi,\xi')),
		\end{equation*}
		where $\phi=\tb(\cdot),\ts(\cdot)$ or $\tg(\cdot,z)$.
		
		\item[{\rm(B.2)}] the partial derivatives $\nabla_x\tb,\nabla_x\ts,\nabla_x\tg,\nabla_x\tf,\nabla_x g$ are uniformly bounded and continuous.
		
		\item[{\rm(B.3)}] there exists a constant $K>0$ such that, for all $(x,\xi,u)\in \R^n\times\Pc_2(\R^n\times \Pc(U))\times U$,
		\begin{equation*}
			|\phi(x,\xi,u)|\leq K\left(1+|x|+M_2(\xi)\right)
		\end{equation*}
		with $\phi=\tb(\cdot),\ts(\cdot)$, $\tf(\cdot)$ or $\tg(\cdot,z)$ and $M_2(\xi):=(\int_{\R^n\times U} (|x|^2+\lv u\rv^2)\mathscr{P}(\xi)(\d x,\d u))^{\frac{1}{2}}$. Moreover, it holds that $\int_Z|\tg(x,\xi,u,z)|^2\lambda(\d z)\leq K(1+|x|^2+M_2(\xi)^2)$.
		
		\item[{\rm(B.4)}] the extensions $\tb,\ts,\tg$ have { linear functional derivative}s $\frac{\delta \tb}{\delta\xi},\frac{\delta \ts}{\delta\xi},\frac{\delta \tg}{\delta\xi}$, respectively. These { linear functional derivative}s have bounded continuous partial derivatives w.r.t. the state variables. That is, the partial $L$-derivatives of $\tb,\ts,\tg$ w.r.t. $\mu$ defined in \defref{Ldif} are bounded and continuous.
		
		\item[{\rm(B.5)}] the coefficients $\tf,g$ have { linear functional derivative}s $\frac{\delta f}{\delta\xi},\frac{\delta g}{\delta\mu}$, respectively. These { linear functional derivative}s have continuous partial derivatives w.r.t. the state variables, $\ie$ $\tf,g$ admit continuous $L$-derivative w.r.t. $\mu$. Furthermore, the growth condition holds for some constant $L>0$ independent of $(x,\xi,u)\in\R^n\times\Pc_2(\R^n\times \Pc(U))\times U$,
		{\small\begin{align*}
				\left(\int_{\R^n\times U}|\pa_{\mu}\tf(x,\xi,u)(x',u')|^2\xi(\d x',\d u')\right)^{\frac12}+\left(\int_{\R^n}|\pa_{\mu}g(x,\mu)(x')|^2\mu(\d x')\right)^{\frac12}\leq L(1+|x|+M_2(\xi)).
		\end{align*}}
		
		\item[{\rm(B.6)}] for any $(x,x',u,q)\in\R^n\times\R^n\times U\times\Pc(U)$, any square-integrable $\R^n\times U$-valued r.v. $(X,\alpha)$ on some probability space and a $\sigma$-field $\G$ on it, we have
		\begin{equation*}
			\begin{cases}
				\displaystyle \phi(x,\Law((X,\alpha)|\G),u)=\tilde\phi(x,\Law((X,\delta_{\alpha})|\G),u),\\
				\displaystyle \int_U\frac{\delta \phi}{\delta \rho}(x,\Law((X,\alpha)|\G),u)(x',u')q'(\d u')=\frac{\delta\tilde\phi}{\delta\xi}(x,\Law((X,\delta_{\alpha})|\G),u)(x',q'),
			\end{cases}
		\end{equation*}
		where $\phi=b(\cdot),\sigma(\cdot)$, $f(\cdot)$ or $\gamma(\cdot,z)$.
	\end{itemize}
	
	For any $q\in\Qb$, the state process under relaxed control is governed by, for $t\in[0,T]$,
	\begin{equation}\label{eq:req1state}
		\begin{cases}
			\displaystyle \d X_t=\int_U\tb (X_t,\Law((X_t,q_t)|\G_t),u)q_t(\d u)\d t+\int_U\ts(X_t,\Law((X_t,q_t)|\G_t),u)q_t(\d u) dW_t\\
			\displaystyle\qquad\quad+\int_U\int_Z \tg(X_{t-},\Law((X_{t-},q_{t-})|\G_{t-}),u,z)q_{t-}(\d u)\tN(\d t,\d z),\\
			\displaystyle\Law(X_0)=\mu.
		\end{cases}
	\end{equation}
	We aim to minimize the following cost functional over $q\in\Qb$ that
	\begin{equation}\label{eq:rvalue1}
		\J(q):=\E\left[\int_0^T\int_U\tf(X_t,\Law((X_t,q_t)|\G_t),u)q_t(\d u)\d t+g(X_T,\Law(X_T|\G_T))\right]\to\inf_{q\in\Qb}.
	\end{equation}
	
	A standard moment estimate under \assref{ass} yields $\E[\sup_{t\in[0,T]}|X_t|^2]<\infty$, which implies that the control problem \eqref{eq:req1state}-\eqref{eq:rvalue1} is well-defined under \assref{ass}. It is straightforward to see that the problem \eqref{eq:req1state}-\eqref{eq:rvalue1} will reduce to problem \equref{eq:svalue1}-\equref{eq:seq1} when $q\in\delta(\mathscr{U})$. If one can find a control $q^*\in\Qb$ such that
	$$\inf_{q\in\Qb}\J(q)=\J(q^*),$$
	we call $q^*$ an optimal relaxed extended MFC.

	\section{Stochastic Maximum Principle}\label{sec:SMP}
	
	In this section, we will first prove the SMP for the relaxed extended MFC problem using the first order variation and then derive an equivalence result of the value functions between the strict control formulation  and the relaxed control formulation. Building upon these two key results, we finally establish the SMP using the first order adjoint process for the original extended MFC problem with strict controls on general control domain that may not be convex.

	\subsection{First-order variation}\label{subsec:first-orderva}
	
	For the relaxed control problem, $\Pc(U)$ is compact and convex, and hence we can apply the first-order variation. That is, for any two relaxed controls $q,v\in\Qb$, we define a new relaxed control as follows $q^{\epsilon}:=q+\epsilon(v-q)\in\Qb$ for $\epsilon\in[0,1]$. Denote by $X^{\epsilon}=(X_t^{\epsilon})_{t\in[0,T]}$ the state process under the relaxed control $q^{\epsilon}$ according to the dynamics \eqref{eq:req1state}. We first have the following lemma, whose proof is standard and omitted.
	
	\begin{lemma}\label{statedif}
		Let $X=(X_t)_{t\in[0,T]}$ be the state process satisfying \eqref{eq:req1state}. Then, under \assref{ass}, as $\epsilon\downarrow0$,
		\begin{equation}\label{eq:dif}
			\sup_{t\in[0,T]}\E\left[|X_t^{\epsilon}-X_t|^2\right]=O(\epsilon^2).
		\end{equation}
	\end{lemma}
	
	For $\omega\in\Omega$, let $Q_{\omega}$ be the r.c.p.d. of the probability measure $\Pb$ in \eqref{eq:probabspace} given $\G_T$. It is clear that, under $Q_{\omega}$, the law of $(X_t,q_t)$ coincides with $\Law((X_t,q_t)|\G_t)(\omega)$ for $\Pb$-a.s..(recall \equref{compatible}) We consider the copy measurable space $(\Omega',\F')\equiv(\Omega,\F)$. For any $\omega\in\Omega$, let us define by r.c.p.d. $\Pb':=Q_{\omega}$,
	which is a probability measure on $(\Omega',\F')$. We can define a copy random variable $X'$ on $(\Omega',\F')$ for every random variable $X$ on $(\Omega,\F)$ in the sense that $X'(\omega)=X(\omega)$ for all $\omega\in\Omega'=\Omega$. It is easy to see that $X'$ is indeed a random variable on $(\Omega',\F',\Pb')$, and moreover we have
	\begin{equation}\label{eq:lawtrans}
		\Law^{Q_{\omega}}(X')=\Law(X|\G_T)(\omega),
	\end{equation}
	where $\Law^{Q_{\omega}}(X')$ denotes the law of the random variable $X'$ under the probability measure $Q_{\omega}$. In the sequel, $\E'$ refers to the expectation in $(\Omega',\F')$ under the probability measure $\Pb'$. 
	\begin{remark}
		We also stress the next property of the expectation w.r.t. $\Pb'$. For random variables $X$ and its copy $X'$ respectively defined on $(\Omega,\F,\Pb)$ and $(\Omega',\F',\Pb')$, and let $F:\R^n\times\R^n\mapsto\R^n$ be measurable, the expectation is understood in the following sense:
		\begin{equation*}
			\E'\left[F(X(\omega),X')\right]=\E^{\Pb'}\left[F(x,X')\right]|_{x=X(\omega)}=\E^{Q_{\omega}}\left[F(x,X')\right]|_{x=X(\omega)} = \E[F(x, X)|\G_T]|_{x = X(\omega)}.
		\end{equation*}
		It is emphasized that the above equality should not be limited to Euclidean-valued random variables. For any random variables defined on the new space $(\Omega',\F',\Pb')$ taking values in any measurable space, the above equality still holds.
	\end{remark}
	Thus, we have the following result on the variational equation whose proof is reported in Appendix \ref{appendix}. Recall that $\xi_t=\Law((X_t,q_t)|\G_t)$ for $t\in[0,T]$ is defined in the proof of \lemref{statedif}.
	\begin{lemma}\label{dif}
		Let \assref{ass} hold. Then there exists a unique solution $V = (V_t)_{t \in [0,T]}$, taking values in $\R^n$, to the self-dependent variational equation
		\begin{equation}\label{eq:variation}
			\begin{aligned}
				\d V_t&=\{\lambda_{b, t}V_t+\beta_{b, t}+\E'[\eta_{b, t}V_t'+\zeta_{b, t}]\}\d t+\{{\lambda}_{\sigma, t}V_t+{\beta}_{\sigma, t}+\E'[{\eta}_{\sigma, t}V_t'+{\zeta}_{\sigma,t}]\}\d W_t\\
				&\quad+\int_Z\{{\lambda}_{\gamma, t-}(z)V_{t-}+{\beta}_{\gamma, t-}(z)+\E'[\eta_{\gamma, t-}(z)V_{t-}'+{\zeta}_{\gamma, t-}(z)]\}\tN(\d t,\d z)
			\end{aligned}
		\end{equation}
		with the initial condition $V_0 = 0$, {where $V'=(V_t')_{t\in[0,T]}$ is a copy of $V=(V_t)_{t\in[0,T]}$ on $(\Omega',\F')$. Here, the equation is called self-dependent because it involves the random process $V$ through its conditional copy $V'$. 
			Furthermore, it holds that
			\begin{equation*}
				\lim_{\epsilon\downarrow 0}\sup_{t\in[0,T]}\E\left[\left|\frac{X_t^{\epsilon}-X_t}{\epsilon}-V_t\right|^2\right]=0.
			\end{equation*}
			Here, for notational convenience, let $\phi \in \{b, \sigma, \gamma\}$ be a mapping from $\R^n\times\Pc_2(\R^n\times U)\times U$ to $E_\phi$, with $E_b = \R^n$, $E_\sigma = \R^{n \times d}$ and $E_\gamma = L^2(Z; \R^n)$. The coefficient processes are given by
			\begin{equation*}
				\begin{aligned}
					\lambda_{\phi, t}&=\int_U\pa_x\tilde\phi(X_t,\xi_t,u)q_t(\d u),~ \beta_{\phi, t}=\int_U\tilde\phi(X_t,\xi_t,u) v_t(\d u)-\int_U\tilde\phi(X_t,\xi_t,u)q_t(\d u),\\
					\eta_{\phi,t}&=\int_U\pa_x\left(\frac{\delta\tilde\phi}{\delta\xi}(X_t,\xi_t,u)\right)(X_t',q_t')q_t(\d u),~\zeta_{\phi,t}=\int_U\pa_q\left(\frac{\delta\tilde\phi}{\delta\xi}(X_t,\xi_t,u)\right)(X_t',q_t')(v_t'-q_t'){q_t(\d u)},
				\end{aligned}
			\end{equation*}
			where $(X', q')$ is a copy of $(X, q)$ on $(\Omega', \mathcal{F}')$.}
	\end{lemma}

	\begin{remark}\label{consistent}
		In the variational equation~\eqref{eq:variation}, {the coefficient process $\zeta_{\phi}$ admits the presentations:
			\begin{equation*}
				\displaystyle \zeta_{\phi,t}=\int_U\frac{\delta\tilde\phi}{\delta\xi}(X_t,\xi_t,u)(X_t',v_t')q_t(\d u)-\int_U\frac{\delta \tilde\phi}{\delta\xi}(X_t,\xi_t,u)(X_t',q_t')q_t(\d u),\\
			\end{equation*}
			which is consistent with $\beta_{\phi}$ according to \lemref{extensiondif}.}
	\end{remark}
	
	We also need the following auxiliary result whose proof is delegated to Appendix~\ref{appendix}.
	\begin{lemma}\label{valuedif}
		Let $q\in\Qb$ be an optimal relaxed control that  minimizes the cost functional $\J$ in \eqref{eq:rvalue1} over $\Qb$ and let $X=(X_t)_{t\in[0,T]}$ be its resulting state process satisfying the dynamics \eqref{eq:req1state}. For any $v\in\Qb$, we have
		\begin{equation}\label{eq:value_variation}
			\begin{aligned}
				0&\leq\E\left\{\pa_xg(X_T,\mu_T)\cdot V_T+\int_0^T\int_U\pa_x\tf(X_t,\xi_t,u)\cdot V_tq_t(\d u)\d t \right.\\
				&\qquad\qquad+\int_0^T\left(\int_U \tf(X_t,\xi_t,u)v_t(\d u)-\int_U\tf(X_t,\xi_t,u)q_t(\d u)\right)\d t\\
				&\quad+\E'\left[\pa_x\left(\frac{\delta g}{\delta\mu}(X_T,\mu_T)\right)(X_T')\cdot V_T'+\int_0^T\int_U\pa_x\left(\frac{\delta\tf}{\delta\xi}(X_t,\xi_t,u)\right)(X_t',q_t')\cdot V_t'q_t(\d u)\d t\right.\\
				&\qquad\qquad+\left.\left.\int_0^T\int_U\pa_q\left(\frac{\delta\tf}{\delta\xi}(X_t,\xi_t,u)\right)(X_t',q_t')(v_t'-q_t')q_t(\d u)\d t\right]\right\}.
			\end{aligned}
		\end{equation}
		Here, the process $V=(V_t)_{t\in[0,T]}$ is given in Lemma~\ref{dif}, and $\mu_T=\Law(X_T|\G_T)$ is the marginal law of $\xi_T$, and the copies $(X',q',V')$ of $(X,q,V)$ are constructed on $(\Omega',\F',\Pb')$.
	\end{lemma}
	
	\subsection{Hamiltonian and first-order adjoint process}
	
	In this subsection, we introduce the so-called relaxed Hamiltonian and the first-order adjoint process. Let us first define the relaxed Hamiltonian $\mathcal{H}:\R^n\times\Pc(U)\times\Pc_2(\R^n\times\Pc(U))\times\R^n\times\R^{n\times d}\times L^2((Z,\mathscr{Z},\lambda);\R^n)\mapsto\R$ by
	\begin{align}\label{RHamiltonian}
		\mathcal{H}(x,q,\xi,p,P,K)&:=\int_U\tb(x,\xi,u)q(\d u)\cdot p+\tr\left(\int_U\ts(x,\xi,u)q(\d u)P^{\T}\right)+\int_U\tf(x,\xi,u)q(\d u)\nonumber\\
		&\quad +\int_Z\int_U\tg(x,\xi,u,z)q(\d u)\cdot K(z)\lambda(\d z),
	\end{align}
	and the relaxed $\delta$-Hamiltonian $\delta\mathcal{H}:\R^n\times\Pc(U)\times\Pc_2(\R^n\times\Pc(U))\times\R^n\times\Pc(U)\times\R^n\times\R^{n\times d}\times L^2((Z,\mathscr{Z},\lambda);\R^n)\mapsto\R$ that
	\begin{align}\label{deltaRdHamiltonian}
		&\delta\mathcal{H}(x,q,\xi,x',q',p,P,K):=\int_U\frac{\delta\tb}{\delta\xi}(x,\xi,u)(x',q')q(\d u)\cdot p+\tr\left(\int_U\frac{\delta\ts}{\delta\xi}(x,\xi,u)(x',q')q(\d u)P^{\T}\right)\nonumber\\
		&\qquad+\int_U\frac{\delta\tf}{\delta\xi}(x,\xi,u)(x',q')q(\d u)+\int_Z\int_U\frac{\delta\tg}{\delta\xi}(x,\xi,u,z)(x',q')q(\d u)\cdot K(z)\lambda(\d z).
	\end{align}
	Then, for any $(x,q,\xi,x',q',p,P,K)\in\R^n\times\Pc(U)\times\Pc_2(\R^n\times\Pc(U))\times\R^n\times\Pc(U)\times\R^n\times\R^{n\times d}\times L^2((Z,\mathscr{Z},\lambda);\R^n)$, it holds that
	{
		\begin{equation}\label{eq:deltaH=deriH}
			\delta\mathcal{H}(x,q,\xi,x',q',p,P,K)=\frac{\delta\mathcal{H}}{\delta\xi}(x,q,\xi,p,P,K)(x',q').
	\end{equation}}
	{ Here, the definition of the linear functional derivative $\frac{\delta\mathcal{H}}{\delta\xi}$ is given in \defref{def:derivativelinear}}
	
	The adjoint process is defined as a (triplet) $\Fb$-adapted process $(\tp,\tP,\tK)=(\tp_t,\tP_t,\tK_t)_{t\in[0,T]}$ taking values in $\R^n\times\R^{n\times d}\times L^2((Z,\mathscr{Z},\lambda);\R^n)$ that satisfies the integrability condition
	\begin{equation}\label{eq:cond}
		\E\left[ \sup_{t\in[0,T]}|\tp_t|^2+\int_0^T|\tP_t|^2\d t+\int_0^T\int_Z|\tK_t|^2\lambda(\d z)\d t
		\right]<\infty,
	\end{equation}
	and the BSDE:
	\begin{equation}\label{RBSDE}
		\begin{cases}
			\displaystyle  \d \tp_t=-\left\{\pa_x\mathcal{H}(X_t,q_t,\xi_t,\tp_t,\tP_t,\tK_t)+\E'\left[\pa_{x'}\delta\mathcal{H}(X_t',q_t',\xi_t,X_t,q_t,\tp_t',\tP_t',\tK_t')\right] \right\}\d t\\
			\displaystyle\qquad\quad+\tP_t\d W_t+\int_Z\tK_{t-}\tilde{N}(\d t,\d z),\\
			\displaystyle \tp_T=\pa_xg(X_T,\mu_T)+\E'\left[\pa_x\left(\frac{\delta g}{\delta\mu}(X_T',\mu_T)\right)(X_T)\right].
		\end{cases}
	\end{equation}
	Here, recall that $X'$ is a copy of the random variable $X$ constructed on the probability space $(\Omega',\F',\Pb')$ as before. The same fashion applies to the notations $\tp'$, $\tP'$ and $\tK'$. In addition, $W=(W_t)_{t\in[0,T]}$ is a $d$-dimensional Brownian motion and  $\tilde{N}(\d t,\d z)$ is a compensated Poisson random measure.
	\begin{remark}
		It can be observed that the BSDE \eqref{RBSDE} is a linear BSDE. Hence, with the help of \assref{ass}, for any relaxed control $q\in\Qb$ and the corresponding state process $X=(X_t)_{t\in[0,T]}$ satisfying the dynamics \eqref{eq:req1state}, the BSDE \eqref{RBSDE} always admits an $\Fb$-adapted and $\R^n\times\R^{n\times d}\times L^2((Z,\mathscr{Z},\lambda);\R^n)$-valued solution $(\tp,\tP,\tK)=(\tp_t,\tP_t,\tK_t)_{t\in[0,T]}$ (c.f. Hao~\cite{TaoHao}).
	\end{remark}

	\subsection{Necessary and sufficient conditions for optimal relaxed control}
	
	For the probability space $(\Omega,\F,\Pb)$,  we recall the definition of its copy $(\Omega',\F',\Pb')$ given in Section \ref{subsec:first-orderva}. We then define a unique probability measure $\tilde{\Pb}$ on $(\tilde{\Omega}:=\Omega\times\Omega',\tilde{\F}:=\F\otimes\F')$ by
	\begin{equation*}
		\tilde{\Pb}(A\times B):=\int_{A\times B}\Pb'(\d \omega')\Pb(\d\omega),\quad \forall A\in\F,B\in\F'.
	\end{equation*}
	Note that, for any $\omega\in\Omega$, $\Pb':=Q_{\omega}$ is the r.c.p.d., which is a probability measure on $(\Omega',\F')$. Denote by $R$ the second marginal law of $\tilde{\Pb}$. Then, for all $A\in\F$,
	\begin{equation}\label{eq:Req}
		R(A)=\int_{\Omega\times A}\Pb'(\d\omega')\Pb(\d\omega)=\int_{\Omega\times A}Q_{\omega}(\d \omega')\Pb(\d \omega)=\int_{\Omega}Q_{\omega}(A)\Pb(\d\omega)=\Pb(A).
	\end{equation}
	We can thus write $\tilde{\Pb}$ in the disintegration form thanks to Radon-Nikodym theorem $\tilde{\Pb}(\d\omega,\d\omega')=Q'_{\omega'}(\d\omega)R(\d\omega')$ where $Q'_{
		\omega'}(\cdot)$ is the Radon-Nikodym derivative of $\tilde{\Pb}$ w.r.t. $R$ given $\omega'$.
	Recall that $(\Omega',\F')=(\Omega,\F)$ is Polish, and is hence countably determined, we can conclude that $Q_{\omega}=Q'_{\omega}$ on $\F$, $\Pb$-a.s.. Then the next result follows.
	
	\begin{theorem}[Necessary Condition]\label{RNEC}
		Let $q\in\Qb$ be an optimal relaxed control attaining the minimum of the cost function $J$ in \eqref{eq:svalue1} over $\Qb$ and $X=(X_t)_{t\in[0,T]}$ be the associated controlled state process satisfying the dynamics \eqref{eq:req1state}. Then, there exists an $\Fb$-adapted solution $(\tp,\tP,\tK)=(\tp_t,\tP_t,\tK_t)_{t\in[0,T]}$ to the BSDE \eqref{RBSDE}. Furthermore, for the relaxed Hamiltonian $\mathcal{H}$ defined by \eqref{RHamiltonian} and  the relaxed $\delta$-Hamiltonian $\delta\mathcal{H}$ defined by \eqref{deltaRdHamiltonian}, we have, $\d t\times\d\Pb$-a.s.
		\begin{equation}\label{eq:RSMP}
			\begin{aligned}
				&\mathcal{H}(X_t,q_t,\xi_t,\tp_t,\tP_t,\tK_t)+\E'\left[\delta\mathcal{H}(X_t',q_t',\xi_t,X_t,q_t,\tp_t',\tP_t',\tK_t')\right]\\
				&\quad\quad\leq\mathcal{H}(X_t,v,\xi_t,\tp_t,\tP_t,\tK_t)+\E'\left[\delta\mathcal{H}(X_t',q_t',\xi_t,X_t,v,\tp_t',\tP_t',\tK_t')\right],~~\forall v\in\Pc(U).
			\end{aligned}
		\end{equation}
		Here, $X'_t,q_t',\tilde{p}_t',\tilde{P}_t',\tilde{K}_t'$ with $t\in[0,T]$ are the corresponding copies defined on the  space $(\Omega',\F')$.
	\end{theorem}

	
	\noindent{\it Proof.}\quad
	Let $\E^{R}$, $\E^{Q_{\omega}}$ and $\E^{Q_{\omega'}}$ represent expectation operators under probability measures $R$, $Q_{\omega}$ and $Q_{\omega'}$ introduced above, respectively (recall that $\E$ is the expectation w.r.t. $\Pb$ but here we use $\E^{\Pb}$ to emphasize which probability measure we are taking expectations with respect to). Then, it follows from Fubini's theorem that
	\begin{equation*}
		\begin{aligned}
			\E^{\Pb}  \left[\tp_T\cdot V_T\right]&=\E^{\Pb}\left[ \pa_xg(X_T,\mu_T){\cdot V_T}+\E^{Q_{\omega}}\left[\pa_x\left(\frac{\delta g}{\delta\mu}(X_T',\mu_T)\right)(X_T)\cdot V_T\right]\right]\\
			&=\E^{\Pb}\left[ \pa_xg(X_T,\mu_T)\cdot V_T\right]+\E^{\Pb}\left[\E^{Q_{\omega}}\left[\pa_x\left(\frac{\delta g}{\delta\mu}(X_T(\omega'),\mu_T)\right)(X_T(\omega))\cdot V_T(\omega)\right]\right]\\
			&=\E^{\Pb}\left[\pa_xg(X_T,\mu_T)\cdot V_T\right]+\E^R\left[\E^{Q_{\omega'}}\left[\pa_x\left(\frac{\delta g}{\delta\mu}(X_T(\omega'),\mu_T)\right)(X_T(\omega))\cdot V_T(\omega)\right]\right]\\
			&=\E^{\Pb}\left[\pa_xg(X_T,\mu_T)\cdot V_T\right]+\E^{\Pb}\left[\E^{Q_{\omega}}\left[\pa_x\left(\frac{\delta g}{\delta\mu}(X_T(\omega),\mu_T)\right)(X_T(\omega'))\cdot V_T(\omega')\right]\right]\\
			&=\E^{\Pb}\left[\pa_xg(X_T,\mu_T)\cdot V_T\right]+\E^{\Pb}\left[\E'\left[\pa_x\left(\frac{\delta g}{\delta\mu}(X_T,\mu_T)\right)(X_T')\cdot V_T'\right]\right].
		\end{aligned}
	\end{equation*}
	Here, the process $V=(V_t)_{t\in[0,T]}$ satisfies \eqref{eq:variation} provided in \ref{dif}. The fourth equality in the above display results from the representation \equref{eq:Req} and the fact $Q_{\omega}=Q'_{\omega}$, on $\F$, $\Pb$-a.s. and $\Pb=R$ on $\F$. On the other hand, we have from It\^{o}'s formula 
	\begin{equation*}
		\begin{aligned}
			\E\left[\tp_T\cdot V_T\right]&=\E\Bigg[\int_0^T \left\{\tp_t\cdot \{\alpha_tV_t+\beta_t+\E'[\eta_tV_t'+\zeta_t]\}-\{\pa_x \mathcal{H}+\E'[\pa_{x'}\delta\mathcal{H}]\}\cdot V_t\right.\\
			&\quad+\tr\left[\tP_t^{\T}(\hat{\alpha}_tV_t+\hat{\beta}_t+\E'[\hat{\eta}_tV_t'+\hat{\zeta}_t])\right]\\
			&\quad+\left.\int_Z\tK_{t-}\cdot\{\tilde{\alpha}_{t-}(z)V_{t-}+\tilde{\beta}_{t-}(z)+\E'[\tilde{\eta}_{t-}(z)V_{t-}'+\tilde{\zeta}_{t-}(z)]\}\lambda(\d z)\right\}\d t\Bigg].
		\end{aligned}
	\end{equation*}
	By combining the above two equalities with \equref{RHamiltonian} and \eqref{RBSDE}, inserting them into \equref{eq:value_variation} and recalling \ref{consistent}, we obtain for all $v\in \Qb$,
	\begin{equation*}
		\begin{aligned}
			&\E\left[\int_0^T\left\{\mathcal{H}(X_t,q_t,\xi_t,\tp_t,\tP_t,\tK_t)+\E'[\delta\mathcal{H}(X_t',q_t',\xi_t,X_t,q_t,\tp_t',\tP_t',\tK_t')]\right\}\d t\right]\\
			&\quad\leq\E\left[\int_0^T\left\{\mathcal{H}(X_t,v_t,\xi_t,\tp_t,\tP_t,\tK_t)+\E'[\delta\mathcal{H}(X_t',q_t',\xi_t,X_t,v_t,\tp_t',\tP_t',\tK_t')]\right\}\d t\right].
		\end{aligned}
	\end{equation*}
	The desired result then follows from the arbitrariness of $v\in\Qb$.\hfill$\Box$
	
	Before introducing our sufficient condition, let us follow Acciaio et al. \cite{Carmona1} to give a definition of so-called $L$-convexity (note that here we are considering the joint law):
	\begin{definition}[$L$-convexity]\label{L-convex} A continuously differentiable function $l:\R^n\times\Pc_2(\R^n\times \Pc(U))\mapsto\R$ is said to be $L$-convex, if for every $(x_1,\xi_1),(x_2,\xi_2)\in \R^n\times\Pc_2(\R^n\times \Pc(U))$, it holds that
		\begin{equation*}
			\begin{aligned}
				&l(x_2,\xi_2)-l(x_1,\xi_1)\\
				&\geq \pa_x l(x_1,\xi_1)\cdot(x_2-x_1)+\E^{\mathbb{P}}\left[\pa_{\mu}l(x_1,\xi_1)(X_1,q_1)\cdot (X_2-X_1)+\pa_{\nu}l(x_1,\xi_1)(X_2,q_2)(q_2-q_1)
				\right],
			\end{aligned}
		\end{equation*}
		where $(X_1,q_1)$ and $(X_2,q_2)$ are $\R^n\times\Pc(U)$-valued r.v.s defined on some probability space $(\Omega,\F,\Pb)$ such that their distributions coincide with $\xi_1$ and $\xi_2$, respectively.
	\end{definition}

	\begin{remark}
		Note that there are several notions of ``convexity". For example, by McCann~\cite{McCann} and Villani~\cite{Villani}, we have
		\begin{itemize}
			\item[{\rm(i)}] A function $l:\R^n\times\Pc_2(\R^n\times \Pc(U))\mapsto\R$ is said to be convex in the classical sense if for any $(x_1,\xi_1),(x_2,\xi_2)\in\R^n\times\Pc_2(\R^n\times\Pc(U))$ and $\theta\in[0,1]$, we have
			\begin{equation}\label{calssical_convex}
				l(\theta x_1+(1-\theta)x_2,\theta\xi_1+(1-\theta)\xi_2)\leq \theta l(x_1,\xi_1)+(1-\theta)l(x_2,\xi_2).
			\end{equation}
			\item[{\rm(ii)}] A function $l:\R^n\times\Pc_2(\R^n\times\Pc(U))\mapsto\R$ is said to be convex in the displacement convex sense if for any $x_1,x_2\in\R^n$, $\R^n\times\Pc(U)$-valued random variables $(X_1,q_1), (X_2,q_2)$ on some probability space and $\theta\in[0,1]$, we have
			\begin{equation}\label{displacement_convex}
				\begin{aligned}
					& l(\theta x_1+(1-\theta)x_2,\Law(\theta X_1+(1-\theta)X_2,\theta q_1+(1-\theta)q_2))\\
					&\qquad\leq \theta l(x_1,\Law(X_1,q_1))+(1-\theta)l(x_2,\Law(X_2,q_2)).
				\end{aligned}
			\end{equation}
		\end{itemize}
		The above two definitions do not require the $L$-differentiability of the mapping $l$. However, we choose to use the $L$-convexity in \defref{L-convex} because it is most suitable in our setting. {By using Proposition 5.79 in Carmona and Delarue \cite{Carmona}, if $l$ is $L$-differentiable, then $L$-convexity and displacement convexity are equivalent in the space of measures on Euclidean space. }
	\end{remark}

	\begin{theorem}[Sufficient Condition]\label{RSUF}
		Let $q\in\Qb$ be a relaxed control, $X=(X_t)_{t\in[0,T]}$ be the resulting controlled state process, and $(\tp,\tP,\tK)=(\tp_t,\tP_t,\tK_t)_{t\in[0,T]}$ be the adjoint process satisfying BSDE~\eqref{RBSDE}. Assume that the Hamiltonian $\mathcal{H}$ is $L$-convex in $(x,\xi)\in\R^n\times\Pc_2(\R^n\times\Pc(U))$ and $g$ is $L$-convex in $(x,\mu)\in\R^n\times\Pc_2(\R^n)$. Then, this $q\in\Qb$ is an optimal relaxed control provided that the inequality \eqref{eq:RSMP} holds.
	\end{theorem}
	
	\noindent{\it Proof.}\quad For any $v\in\Qb$, denote by $X^v=(X^v_t)_{t\in[0,T]}$ the resulting state process. By the convexity of $g$ on $\R^n\times{\cal P}_2(\R^n)$, we have
	\begin{align*}
		&\E\left[g(X_T^v,\mu_T^v)-g(X_T,\mu_T)\right]\\
		&\geq\E\left[\pa_xg(X_T,\mu_T)\cdot(X_T^v-X_T^q)+\E'\left[\pa_x\left(\frac{\delta g}{\delta\mu}(X_T,\mu_T)\right)(X_T')\cdot(X_T^{v'}-X_T')\right]\right]\\
		&=\E\left[\pa_xg(X_T,\mu_T)\cdot(X_T^v-X_T)+\E'\left[\pa_x\left(\frac{\delta g}{\delta\mu}(X_T',\mu_T)\right)(X_T)\cdot(X_T^v-X_T)\right] \right]\nonumber\\
		&=\E\left[\tp_T\cdot(X_T^v-X_T)\right].
	\end{align*}
	Applying It\^{o}'s formula to $\tp_t\cdot(X_t^v-X_t^q)$, we arrive at
	\begin{align*}
		&\E\left[\tp_T\cdot(X_T^v-X_T)\right]=\E\Bigg[\int_0^T\left[\tp_t\cdot\left(\int_U\tb(X_t^v,\xi_t^v,u)v_t(\d u)-\int_U\tb(X_t,\xi_t,u)q_t(\d u)\right)\right.\\
		&\qquad-(X_t^v-X_t)\cdot\left(\pa_x\mathcal{H}(X_t,q_t,\xi_t,\tp_t,\tP_t,\tK_t)+\E'[\pa_{x'}\delta\mathcal{H}(X_t',q_t',\xi_t,X_t,q_t,\tp_t',\tP_t',\tK_t')]\right)\\
		&\qquad+\tr\left(\tP_t^{\T}\left(\int_U\ts(X_t^v,\xi_t^v,u)v_t(\d u)-\int_U\ts(X_t,\xi_t,u)q_t(\d u)\right)\right)\\
		&\qquad+\left.\tK_t\cdot\int_Z\left(\int_U\tg(X_t^v,\xi_t^v,u,z)v_t(\d u)-\int_U\tg(X_t,\xi_t,u,z)q_t(\d u)\right)\lambda(\d z)\right]\d t\Bigg],
	\end{align*}
	where we recall $\xi_t^v=\Law((X_t^v,v_t)|\G_{t})$ for $t\in[0,T]$. Hence, we have
	\begin{align*}
		&J(v)-J(q)\geq \E\Big[\mathcal{H}(X_t^v,v_t,\xi_t^v,\tp_t,\tP_t,\tK_t)-\mathcal{H}(X_t,q_t,\xi_t,\tp_t,\tP_t,\tK_t)\\
		&\quad-\left.\pa_x\mathcal{H}(X_t,q_t,\xi_t,\tp_t,\tP_t,\tK_t)\cdot(X_t^v-X_t)-\E'[\pa_{x'}\delta\mathcal{H}(X_t',q_t',\xi_t,X_t,q_t,\tp_t',\tP_t',\tK_t')\cdot(X_t^v-X_t)]\right]\\
		&~=\E\left[\mathcal{H}(X_t^v,v_t,\xi_t^v,\tp_t,\tP_t,K_t)-\mathcal{H}(X_t,q_t,\xi_t,\tp_t,\tP_t,\tK_t)\right.\\
		&\quad-\left.\pa_x\mathcal{H}(X_t,q_t,\xi_t,\tp_t,\tP_t,\tK_t)\cdot(X_t^v-X_t)-\E'[\pa_{x'}\delta\mathcal{H}(X_t,q_t,\xi_t,X_t',q_t',\tp_t,\tP_t,\tK_t)\cdot(X_t^{v'}-X_t')]\right].
	\end{align*}
	Note that $\mathcal{H}$ is $L$-convex in $(x,\xi)$ and is linear in $q$. Using the Clarke generalized gradient of $\mathcal{H}$, we then have from  \lemref{extensiondif} 
	\begin{align*}
		& J(v)-J(q)\geq\E\left[\mathcal{H}(X_t^v,v_t,\xi_t^v,\tp_t,\tP_t,K_t)-\mathcal{H}(X_t,q_t,\xi_t,\tp_t,\tP_t,\tK_t)\right.\\
		&\quad-\left.\pa_x\mathcal{H}(X_t,q_t,\xi_t,\tp_t,\tP_t,\tK_t)\cdot(X_t^v-X_t)-\E'[\pa_{x'}\delta\mathcal{H}(X_t,q_t,\xi_t,X_t',q_t',\tp_t,\tP_t,\tK_t)\cdot(X_t^{v'}-X_t')]\right]\\
		&\quad \geq \E\left[\pa_x\mathcal{H}(X_t,q_t,\xi_t,\tp_t,\tP_t,\tK_t)\cdot(X_t^v-X_t)-\pa_x\mathcal{H}(X_t,q_t,\xi_t,\tp_t,\tP_t,\tK_t)\cdot(X_t^v-X_t)\right.\\
		&\qquad\qquad+\E'[\pa_{x'}\delta\mathcal{H}(X_t,q_t,\xi_t,X_t',q_t',\tp_t,\tP_t,\tK_t)\cdot(X_t^{v'}-X_t')]\\
		&\qquad\qquad-\E'[\pa_{x'}\delta\mathcal{H}(X_t,q_t,\xi_t,X_t',q_t',\tp_t,\tP_t,\tK_t)\cdot(X_t^{v'}-X_t')]\\
		&\qquad\qquad+\mathcal{H}(X_t,v_t,\xi_t,\tp_t,\tP_t,\tK_t)-\mathcal{H}(X_t,q_t,\xi_t,\tp_t,\tP_t,\tK_t)\\
		&\qquad\qquad\left.+\delta\mathcal{H}(X_t',q_t',\xi_t,X_t,v,\tp_t',\tP_t',\tK_t')-\delta\mathcal{H}(X_t',q_t',\xi_t,X_t,q_t,\tp_t',\tP_t',\tK_t')\right]\\
		&\quad=\E\left[\mathcal{H}(X_t,v_t,\xi_t,\tp_t,\tP_t,\tK_t)-\mathcal{H}(X_t,q_t,\xi_t,\tp_t,\tP_t,\tK_t)\right.\\
		&\qquad\qquad\left.+\delta\mathcal{H}(X_t',q_t',\xi_t,X_t,v,\tp_t',\tP_t',\tK_t')-\delta\mathcal{H}(X_t',q_t',\xi_t,X_t,q_t,\tp_t',\tP_t',\tK_t')\right]\geq 0.
	\end{align*}
	The last inequality follows from \equref{eq:RSMP}.  As $v\in\Qb$ is arbitrary, $q\in\Qb$ is an optimal relaxed control. \hfill$\Box$
	
	{ \subsection{Equivalence between relaxed control and strict control formulations}
		This subsection is devoted to establishing the value function equivalence between the strict and relaxed control formulations. Once the equivalence is in force, the stochastic maximum principle for relaxed controls can be transferred to the framework of strict controls. Let us first give the next result whose proof is similar to that of Lemma 16 in Bahlali \cite{Bahlali}.	
		\begin{lemma}\label{equvalent}
			Let $\alpha=(\alpha_t)_{t\in[0,T]}\in\mathscr{U}$. Then, $\alpha$ minimizes the objective functional \eqref{eq:svalue1} of the strict MFC problem over $\U$ if and only if $q=(\delta_{\alpha_t})_{t\in[0,T]}$ minimizes the objective functional \eqref{eq:rvalue1} of the relaxed MFC problem over $\delta(\U)$.
		\end{lemma}
		The following lemma shows that the conditional joint law of state–control given common noise is intrinsic and its representation does not depend on the choice of the underlying probability space.
		\begin{lemma}\label{law_representation}
			Let $q\in\Qb$ and $X=(X_t)_{t\in [0,T]}$ be the controlled process. We denote by $\xi=\Law((X,q)|N)$ and let $\xi_t$ be the $t$-marginal of $\xi$, $\ie$, $\xi_t=\xi\circ\pi_t^{-1}$, where the projection mapping $\pi_t:D([0,T];\R^n)\times\mathcal{Q}\to \R^n\times \Pc(U)$ is defined by $\pi_t(x,q)=(x(t),q_t)$ for $(x,q)\in D([0,T];\R^n)\times\mathcal{Q}$. Moreover, suppose that $(\hat X,\hat q,\hat W,\hat N,\hat\xi)$ is defined on some probability space $(\hat\Omega,\hat\F,\hat P)$ and has the same joint law as $(X,q,W,N,\xi)$. Then $\hat P(\hat\xi=\Law^{\hat{P}}((\hat X,\hat q)|\hat N)=1$ and $\hat X$ solves the following controlled SDE on $(\hat\Omega,\hat\F,\hat P)$:
			\begin{align}\label{SDE_hat}
				\d \hat X_t&=\int_U\tb(\hat X_t,\hat\xi_t,u)\hat q_t(\d u)\d t+\int_U\ts(\hat X_t,\hat\xi_t,u)\hat q_t(\d u)\d \hat W_t\nonumber\\
				&\quad+\int_U\int_Z\tg(\hat X_t,\hat\xi_t,u,z)\hat q_t(\d u)\tilde{\hat N}(\d t,\d z),
			\end{align}
			where $\hat\xi_t=\hat\xi\circ\pi_t^{-1}$. Furthermore, if $q$ has continuous paths $\ie$ $q\in\Qb_c$ (see the definition of $\Qb_c$ below), then it holds that $\hat P(\hat\xi_t=\Law^{\hat P}((\hat X_t,\hat q_t)|\hat \G_t),\forall t\in [0,T])=1$, where $\hat\G_t=\sigma(\hat N((0,s]\times A),s\leq t,A\in\mathscr{Z})$.
		\end{lemma}
		
		We provide below the so-called \textit{chattering lemma} without proof (c.f. Lemma 4.2 in Mezerdi~\cite{Mohamed}), which will be used later.
		\begin{lemma}[Chattering Lemma]\label{chattering}
			Let $q\in \Qb$. Then, there exists a sequence of adapted process $(\alpha^n)_{n\geq1}$ taking values in $U$ such that the sequence of random measures $\delta_{\alpha_t^n}(\d u)\d t$ $\as$ converges to $q$ as $n\to\infty$ in $\mathcal{Q}$, $\ie$ for any continuous function $f:[0,T]\times U\to\R$, it holds that, a.s.
			\begin{equation*}
				\lim_{n\to\infty}\int_0^Tf(t,\alpha_t^n)\d t=\int_0^T\int_Uf(t,u)q_t(\d u) \d t.
			\end{equation*}
		\end{lemma}
		
		We impose the following additional assumption such that all coefficients satisfy the separable condition. In particular, they are (partially) affine in the conditional joint law of the state and the strict control.
		\begin{ass}\label{extra_ass}
			The coefficients $b,\sigma,\gamma,f$ satisfy the separable condition. To be more precise, there exist Borel mappings $(b_1,\sigma_1,f_1):\R^n\times\Pc_2(\R^n)\times U\to\R^n\times\R^{n\times d}\times\R$, $\gamma_1:\R^n\times\Pc_2(\R^n)\times U\times Z\to\R^n$, $(b_2,\sigma_2,f_2):\R^n\times\Pc_2(\R^n)\times U\times\R^n\times U\to\R^n\times\R^{n\times d}\times\R$ and $\gamma_2:\R^n\times\Pc_2(\R^n)\times U\times Z\times\R^n\times U\to\R^n$ such that, for every $(x,\rho,u,z)\times\R^n\times\Pc_2(\R^n\times U)\times U\times Z$,
			\begin{align*}
				(b,\sigma,f)(x,\rho,u)&=(b_1,\sigma_1,f_1)(x,\mu,u)+\int_{\R^n\times U}(b_2,\sigma_2,f_2)(x,\mu,u)(x',u')\rho(\d x',\d u'),\\
				\gamma(x,\rho,u,z)&=\gamma_1(x,\mu,u,z)+\int_{\R^n\times U}\gamma_2(x,\mu,u,z)(x',u')\rho(\d x',\d u'),
			\end{align*}
			where $\mu\in\Pc_2(\R^n)$ denotes the first marginal distribution of $\rho$. Moreover, $(b_1,\sigma_1,\gamma_1)$, $(b_2,\sigma_2,\gamma_2)$ are uniformly Lipschitz continuous in $(x,\mu,x')\in\Pc_2(\R^n)\times\R^n$.
			$f_1$, $f_2$ and $g$ are locally Lipschitz in the sense that there exists a constant $L>0$ independent of $(x_i,\mu_i,u,x_i',u')\in \R^n\Pc_2(\R^n)\times U\times\R^n\times U$ such that
			{\small\begin{align}\label{local_Lip}
					&|f_1(x_1,\mu_1,u)-f_1(x_2,\mu_2,u)|+|f_2(x_1,\mu_1,u)(x_1',u')-f_2(x_2,\mu_2,u)(x_2',u)|+|g(x_1,\mu_1)-g(x_2,\mu_2)|\nonumber\\
					&\leq L\left(1+|x_1|+|x_2|+|x_1'|+|x_2'|+M_2(\mu)+M_2(\mu'))(|x_1-x_2|+|x_1'-x_2'|+\mathcal{W}_{1,\R^n}(\mu_1,\mu_2)\right),
			\end{align}}
			$(b_2,\sigma_2,\gamma_2,f_2)$ are uniformly continuous in $u'\in U$ and have at most quadratic growth in $(x',u')\in \R^n\times U$.
		\end{ass}
		The next lemma allows us to restrict our attention to continuous relaxed controls, $\ie$, the class $\Qb_c$ consisting of all relaxed controls $q=(q_t)_{t\in[0,T]}\in\Qb$ whose sample paths are continuous outside a $\Pb$-null set. The proof is deferred to Appendix~\ref{appendix}.
		\begin{lemma}\label{continuous_control}
			Let \assref{extra_ass} hold. For any $q\in \Qb$, there exists a sequence continuous relaxed controls $\{q^n=(q_t^n)_{t\in [0,T]};~n\in\mathbb{N}\}\subset\Qb_c$ converging to $q$ strongly in the sense that $q_t^n(A)\to q_t(A)$ for all $A\in\mathcal{B}(U)$ and $t\in [0,T]$ $\as$. Moreover, let $X=(X_t)_{t\in [0,T]}$ and $X^n=(X_t^n)_{t\in [0,T]}$ be the McKean-Vlasov dynamics \eqref{eq:req1state} controlled by $q$ and $q^n$, respectively, then 
			\begin{align*}
				\lim_{n\to\infty}\E^{\Pb}\left[\sup_{t\in [0,T]}|X_t-X_t^n|^2\right]=0.
			\end{align*}
			Consequently, we have the value function equivalence:
			\begin{align}\label{value_continuity}
				\inf_{q\in\Qb}\J(q)=\inf_{q\in\Qb_c}\J(q).
			\end{align}
		\end{lemma}
		Similar to Lemma 3.3 in Ma and Yong \cite{Majin}, we have the following weak convergence result, whose proof is delegated to Appendix~\ref{appendix}.
		\begin{lemma}\label{L2con}
			Let $q\in\Qb_c$ and $q^n$ be a sequence of relaxed controls such that $q^n\to q$ in $\mathcal{Q}$ $\as$. Let $X=(X_t)_{t\in [0,T]}$ be the  Mckean-Vlasov dynamics \eqref{eq:req1state} controlled by $q$ and set $\xi_t(\omega)=\Law((X_t,q_t)|\G_t)$. Furthermore, we define
			\begin{align}\label{eq:req1auxiliarytes}
				M_{\tb}^n(t,\omega)&:=\int_U\tb(X_t(\omega),\xi_t(\omega),u)q_t^n(\omega,\d u),~M_{\tb}(t,\omega):=\int_U\tb(X_t(\omega),\xi_t(\omega),u)q(\omega,\d u);\nonumber\\
				M_{\ts}^n(t,\omega)&:=\int_U\ts(X_t(\omega),\xi_t(\omega),u)q_t^n(\omega,\d u),~M_{\ts}(t,\omega):=\int_U\ts(X_t(\omega),\xi_t(\omega),u)q(\omega,\d u);\\
				M_{\tg}^n(t,\omega,z)&:=\int_U\tg(X_t(\omega),\xi_t(\omega),u,z)q_t^n(\omega,\d u),~M_{\tg}(t,\omega,z):=\int_U\tg(X_t(\omega),\xi_t(\omega),u,z)q(\omega,\d u).\nonumber
			\end{align}
			Then, for all $z\in Z$, as $n\to\infty$,
			\begin{equation*}
				M_{\tb}^n\overset{w}{\to} M_{\tb},~~M_{\ts}^n\overset{w}{\to} M_{\ts},~~M_{\tg}^n(\cdot,z)\overset{w}{\to} M_{\tg}(\cdot,z),~~\text{in}~L^2([0,T]\times\Omega).
			\end{equation*}
			{Or equivalenly, it holds for every test function $\psi\in L^2([0,T]\times\Omega)$ that\begin{align*}
					\E\left[\int_0^TM_{\tilde \phi}^n(t)\psi(t)\d t\right]\to\E\left[\int_0^TM_{\tilde \phi}(t)\psi(t)\d t\right],~n\to\infty, ~\forall\phi\in\{b,\sigma,\gamma(\cdot,z)\}.
			\end{align*}}
		\end{lemma}

		Now, we can show an equivalence result on the value functions between two formulations in the next result whose proof is deferred to Appendix~\ref{appendix}.
		\begin{prop}\label{valueeq}
			Let \assref{ass} and \assref{extra_ass} hold. The (MFC) value function defined by \eqref{eq:svalue1} in the strict control formulation coincides with the one defined by \eqref{eq:rvalue1} in the relaxed control formulation, i.e., it holds that
			\begin{equation}\label{eq:eqvalue}
				\inf_{\alpha\in\U}J(\alpha)=\inf_{q\in\Qb}\J(q).
			\end{equation}
		\end{prop}
		{
			\begin{remark}[Outline of the proof of Proposition~3.16]
				The proof follows the route developed in Ma and Yong~\cite{Majin}, but it requires technical efforts to handle the convergence of the joint law of state and control in our context, which calls for \assref{extra_ass}. The key steps are as follows:
				\begin{enumerate}
					\item By \assref{extra_ass} and \lemref{continuous_control}, it suffices to consider relaxed controls with continuous trajectories, denoted by $q \in \mathcal{Q}_c$.  For any such $q$, the Chattering Lemma (\lemref{chattering}) yields a sequence of strict controls $\alpha^n$ converging weakly to $q$ in $\mathcal{Q}$.
					\item We consider the sequence of state process $X^n$, control $\alpha^n$, the joint distribution of state and control $\xi^n$, as well as the idiosyncratic noise $W^n$ and common noises $N^n$. It is shown that the collection of their joint laws is tight.  Prokhorov's theorem then provides a weakly convergent subsequence.  Using the Skorokhod representation (\lemref{law_representation}), we upgrade this convergence to almost sure convergence on a new probability space $(\hat\Omega, \hat{\mathcal{F}}, \hat{\mathbb{P}})$.
					\item On this new space, \lemref{L2con} guarantees convergence $M_{\widetilde\phi}^n\to M_{\widetilde\phi}$ weakly in $L^2([0,T]\times\hat\Omega)$ for $\phi\in\{b,\sigma,\gamma(\cdot,z)\}$.  Mazur's theorem implies that a sequence of convex combinations of $\{M_{\widetilde\phi}^n\}$ converges \emph{strongly} in $L^2$ to $M_{\widetilde\phi}$.
					\item By exploiting \assref{extra_ass}, we can pass to the limit in the martingale problem and identify the limiting probability measure as the solution corresponding to $q$ thanks to the weak uniqueness. The continuity of the cost functional with respect to the weak topology then yields the desired value equivalence.
				\end{enumerate}
		\end{remark}}

	}
	\subsection{Necessary and sufficient condition for optimal strict control}
	
	Define the Hamiltonian $H:\R^n\times U\times\Pc_2(\R^n\times U)\times\R^n\times\R^{n\times d}\times L^2((Z,\mathscr{Z},\lambda);\R^n)\mapsto\R$ in the strict control formulation by
	\begin{equation}\label{SHamiltonian}
		\begin{aligned}
			H(x,u,\rho,p,P,K)&:=b(x,\rho,u)\cdot p+\tr\left(\sigma(x,\rho,u)P^{\T}\right)+f(x,\rho,u)\\
			&\quad+\int_Z\gamma(x,\rho,u,z)\cdot K(z)\lambda(\d z),
		\end{aligned}
	\end{equation}
	and consider the strict $\delta$-Hamiltonian $\delta H:\R^n\times U\times\Pc_2(\R^n\times U)\times\R^n\times U\times\R^n\times\R^{n\times d}\times L^2((Z,\mathscr{Z},\lambda);\R^n)\mapsto \R$ defined by
	\begin{equation}\label{SdHamiltonian}
		\begin{aligned}
			\delta H(x,u,\rho,x',u',p,P,K)&=\frac{\delta b}{\delta\rho}(x,\rho,u)(x',u')\cdot p+\tr\left(\frac{\delta\sigma}{\delta\rho}(x,\rho,u)(x',u')P^{\T}\right)\\
			&\quad+\frac{\delta f}{\delta\rho}(x,\rho,u)(x',u')+\int_Z\frac{\delta \gamma}{\delta\rho}(x,\rho,u,z)(x',u')\cdot K(z)\lambda(\d z).
		\end{aligned}
	\end{equation}
	It then holds that $\delta H(x,u,\rho,x',u',p,P,K)=\frac{\delta H}{\delta\rho}(x,u,\rho,p,P,K)(x',u')$.
	
	Consider the triplet $(p,P,K)=(p_t,P_t,K_t)_{t\in[0,T]}$ as an $\R^n\times\R^{n\times d}\times L^2((Z,\mathscr{Z},\lambda);\R^n)$-valued $\Fb$-adapted process that satisfies
	\begin{equation}\label{eq:scond}
		\E\left[ \sup_{t\in[0,T]}|p_t|^2+\int_0^T|P_t|^2\d t+\int_0^T\int_Z|K_t|^2\lambda(\d z)\d t\right]<\infty,
	\end{equation}
	and solves the BSDE under the strict control 
	\begin{equation}\label{SBSDE}
		\begin{cases}
			\displaystyle \d p_t=-\left\{\pa_xH(X_t,\alpha_t,\rho_t,p_t,P_t,K_t)+\E'\left[\pa_{x'}\delta H(X_t',\alpha_t',\rho_t,X_t,\alpha_t,p_t',P_t',K_t')\right] \right\}\d t\\
			\displaystyle \qquad\qquad+P_t\d W_t+\int_ZK_{t-}\tilde{N}(\d t,\d z),\\
			\displaystyle p_T=\pa_xg(X_T,\mu_T)+\E'\left[\pa_x\left(\frac{\delta g}{\delta\mu}(X_T',\mu_T)\right)(X_T)\right],
		\end{cases}
	\end{equation}
	where $\rho_t=\Law((X_t,\alpha_t)|\G_t)$ for $t\in[0,T]$. Then, we call the above triplet $(p,P,K)=(p_t,P_t,K_t)_{t\in[0,T]}$ a solution to BSDE~\eqref{SBSDE}.
	
	It is stressed here that, when $q=(\delta_{\alpha_t})_{t\in[0,T]}\in\delta(\U)$ with $\alpha=(\alpha_t)_{t\in[0,T]}\in\mathscr{U}$, the dynamics in the relaxed control formulation is the same as the one in the strict control formulation, and hence the two state processes must be indistinguishable. Moreover, it holds that, for $\tilde P$-$\as~\tilde\omega=(\omega,\omega')\in\tilde\Omega$, $\forall (x,x',p,P,K)\in\R^n\times\R^n\times\R^n\times\R^{n\times d}\times L^2((Z,\mathscr{Z},\lambda);\R^n)$,
	\begin{equation*}
		\begin{cases}
			\displaystyle H(x,\alpha_t,\rho_t,p,P,K)=\mathcal{H}(x,{q_t, \xi_t},p,P,K),\\
			\displaystyle \delta H(x,\alpha_t,\rho_t,x',\alpha_t',p,P,K)=\delta \mathcal{H}(x,q_t,\xi_t,x',q_t',p,P,K),\\
			\displaystyle	\pa_{x'}\delta H(x,\alpha_t,\rho_t,x',\alpha_t',p,P,K)=\pa_{x'}\delta \mathcal{H}(x,q_t,\xi_t,x',q_t',p,P,K).
		\end{cases}
	\end{equation*}
	Here, for $t\in[0,T]$, $\alpha_t'$ and $q_t'$ defined on the new space $(\Omega',\F')$ are copies of $\alpha_t$ and $q_t$, respectively. Therefore, by using the pathwise uniqueness of BSDE \equref{SBSDE}, the solution $(p,P,K)=(p_t,P_t,K_t)_{t\in[0,T]}$ to BSDE \equref{SBSDE} is indistinguishable from the solution $(\tp,\tP,\tK)=(\tp_t,\tP_t,\tK_t)_{t\in[0,T]}$ to BSDE \equref{RBSDE}.
	
	\begin{remark}\label{Hamiltonian_relation}
		If we set $\xi=\Law(X,\delta_{\alpha})$ when $\Law(X,\alpha)=\rho$ with $(X,\alpha)$ being an $\R^n\times U$-valued random variable on some probability space, then $\ps(\xi)=\rho$ and for any $(x,q,p,P,K)\in\R^n\times\Pc(U)\times\R^n\times\R^{n\times d}\times L^2((Z,\mathscr{Z},\lambda);\R^n)$, it holds that
		\begin{equation*}
			\mathcal{H}(x,q,\xi,p,P,K)=\int_U \tH(x,u,\xi,p,P,K)q(\d u)=\int_U H(x,u,\rho,p,P,K)q(\d u).
		\end{equation*}
		Here, $\mathcal{H}$ is the relaxed Hamiltonian defined by \eqref{RHamiltonian} and $\tH$ is the extension of $H$ defined by \eqref{eq:extensionh}.
	\end{remark}

	Based on \propref{valueeq}, we have the following necessary condition for the existence of an optimal strict control.
	\begin{theorem}[Necessary Condition for Strict Control]\label{SNEC}
		Let $\alpha\in\mathscr{U}$ be an optimal strict control minimizing the cost $J$ in \eqref{eq:svalue1} over $\mathscr{U}$ and let $X=(X_t)_{t\in[0,T]}$ be the resulting controlled state process. Consider the solution $(p,P,K)=(p_t,P_t,K_t)_{t\in[0,T]}$ of BSDE~\eqref{SBSDE}. Then, we have, $\d t\times\d\Pb\text{-}\as$
		\begin{equation}\label{eq:SSMP}
			\begin{aligned}
				&H(X_t,\alpha_t,\rho_t,p_t,P_t,K_t)+\E'\left[\delta H(X_t',\alpha_t',\rho_t,X_t,\alpha_t,p_t',P_t',K_t')\right]\\
				&\qquad\leq H(X_t,u,\rho_t,p_t,P_t,K_t)+\E'\left[\delta H(X_t',\alpha_t',\rho_t,X_t,u,p_t',P_t',K_t')\right],~\forall u\in U,
			\end{aligned}
		\end{equation}
		with $\rho_t=\Law((X_t,\alpha_t)|\G_t)$ with $t\in[0,T]$.
	\end{theorem}
	
	\noindent{\it Proof.}\quad
	It follows from \lemref{valueeq}, \lemref{equvalent} and the prescribed optimal condition that, $q=(\delta_{\alpha_t})_{t\in[0,T]}$ minimizes the objective functional $J$ defined by \eqref{eq:rvalue1} over $\Qb$. Hence, we have from \thmref{RNEC}, $\d t\times\d\Pb\text{-}\as$
	\begin{equation*}
		\begin{aligned}
			&\mathcal{H}(X_t,q_t,\xi_t,\tp_t,\tP_t,\tK_t)+\E'\left[\delta\mathcal{H}(X_t',q_t',\xi_t,X_t,q_t,\tp_t',\tP_t',\tK_t')\right]\\
			&\quad\leq\mathcal{H}(X_t,v,\xi_t,\tp_t,\tP_t,\tK_t)+\E'\left[\delta\mathcal{H}(X_t',q_t',\xi_t,X_t,v,\tp_t',\tP_t',\tK_t')\right],~\forall v\in\Pc(U).
		\end{aligned}
	\end{equation*}
	Note that $\delta(U)=\{\delta_u;~u\in U\}\subset\Pc(U)$, it holds that, $\d t\times\d\Pb\text{-}\as$
	\begin{equation*}
		\begin{aligned}
			&\mathcal{H}(X_t,q_t,\xi_t,\tp_t,\tP_t,\tK_t)+\E'\left[\delta\mathcal{H}(X_t',q_t',\xi_t,X_t,q_t,\tp_t',\tP_t',\tK_t')\right]\\
			&\quad\leq\mathcal{H}(X_t,v,\xi_t,\tp_t,\tP_t,\tK_t)+\E'\left[\delta\mathcal{H}(X_t',q_t',\xi_t,X_t,v,\tp_t',\tP_t',\tK_t')\right],~\forall v\in\delta(U).
		\end{aligned}
	\end{equation*}
	Therefore, we arrive at, $\d t\times\d\Pb\text{-}\as$
	\begin{equation*}
		\begin{aligned}
			&H(X_t,\alpha_t,\rho_t,p_t,P_t,K_t)+\E'\left[\delta H(X_t',\alpha_t',\rho_t,X_t,\alpha_t,p_t',P_t',K_t')\right]\\
			&\quad\leq H(X_t,u,\rho_t,p_t,P_t,K_t)+\E'\left[\delta H(X_t',\alpha_t',\rho_t,X_t,u,p_t',P_t',K_t')\right],~\forall u\in U.
		\end{aligned}
	\end{equation*}
	which yields \eqref{eq:SSMP} as desired. \hfill$\Box$

	The next result provides a sufficient condition for the strict control problem.
	\begin{theorem}[Sufficient Condition for Strict Control]\label{SSUF}
		Let $\alpha\in\U$, and $X=(X_t)_{t\in[0,T]}$ be the corresponding state process under this strict control $\alpha\in\U$, and $(p,P,K)=(p_t,P_t,K_t)_{t\in[0,T]}$ be the corresponding adjoint process satisfying \eqref{eq:scond}-\eqref{SBSDE}. Assume that the extension of the Hamiltonian $H$ in \remref{Hamiltonian_relation}, denoted by $\tH$, is $L$-convex in $(x,\xi)\in\R^n\times\Pc_2(\R^n\times\Pc(U))$ and $g$ is $L$-convex in $(x,\mu)\in\R^n\times\Pc_2(\R^n)$. Then, this $\alpha\in\U$ is an optimal strict control in the sense that it minimizes the cost functional \eqref{eq:svalue1} over $\U$ provided the inequality \eqref{eq:SSMP} holds.
	\end{theorem}
	
	\noindent{\it Proof.}\quad
	In view of \remref{Hamiltonian_relation} and the discussion above \remref{Hamiltonian_relation}, we deduce, for any $v\in\Pc(U)$,
	\begin{equation*}
		\begin{aligned}
			&\mathcal{H}(X_t,v,\xi_t,\tp_t,\tP_t,\tK_t)+\E'[\delta\mathcal{H}(X_t',\delta_{\alpha_t'},\xi_t,X_t,v,\tp_t',\tP_t',\tK_t')]\\
			&\qquad \geq \inf_{u\in\U}\left\{H(X_t,u,\rho_t,p_t,P_t,K_t)+\E'[\delta\mathcal{H}(X_t',\alpha_t',\rho,X_t,u,p_t',P_t',K_t')]\right\}\\
			&\qquad = H(X_t,\alpha_t,\rho_t,p_t,P_t,K_t)+\E'\left[\delta H(X_t',\alpha_t',\rho_t,X_t,\alpha_t,p_t',P_t',K_t')\right]\\
			&\qquad =\mathcal{H}(X_t,q_t,\xi_t,\tp_t,\tP_t,\tK_t)+\E'[\delta H(X_t',q_t',\xi_t,X_t,q_t,\tp_t',\tP_t',\tK_t')],
		\end{aligned}
	\end{equation*}
	where $q_t=\delta_{\alpha_t}$, $\rho_t=\Law((X_t,\alpha_t)|\G_t)$ and $\xi_t=\Law((X_t,\delta_{\alpha_t})|\G_t)$ for $t\in[0,T]$. Using \thmref{RSUF}, we conclude that $q=(q_t)_{t\in[0,T]}$ is an optimal relaxed control and $\alpha=(\alpha_t)_{t\in[0,T]}\in\U$ is an optimal strict control. \hfill$\Box$

	\begin{remark}
		It is noted that we do not require the convexity of the Hamiltonian w.r.t. the control variable and the $L$-differentiability w.r.t. the second marginal law (the law of control). Moreover, the control space $U$, although restricted to be compact in our framework, is not necessarily convex. Compared with Acciaio et al.~\cite{Carmona1}, we can get the equivalence of minimizing the Hamiltonian and minimizing the original problem thanks to the relaxed formulation. However, one can also follow the method in Acciaio et al.~\cite{Carmona1} to directly prove the SMP for strict control by assuming the convexity of the policy space $U$. Hence, the SMP holds when either convexity or compactness is fulfilled by the policy space $U$.
	\end{remark}
	
	\subsection{Poissonian common noise vs McKean-Vlasov jump-diffusion model}
	
	To better elaborate how the Poissonian common noise affects the SMP and the adjoint equation, we next apply our previous result to the extended MFC problem in the jump-diffusion model without common noise, where the jump term only stands for idiosyncratic noise. The McKean-Vlasov controlled state process is given by
	\begin{equation}\label{eq:seq3}
		\begin{cases}
			\displaystyle \d X_t =b(X_t,\Law(X_t,\alpha_t),\alpha_t)\d t+\sigma(X_t,\Law(X_t,\alpha_t),\alpha_t)\d W_t\\
			\displaystyle \qquad\qquad+\int_Z\gamma(X_{t-},\Law(X_{t-},\alpha_{t-}),\alpha_{t-},z)\tN(\d t,\d z),\\
			\displaystyle \Law(X_0)=\mu\in{\cal P}_2(\R^n),
		\end{cases}
	\end{equation}
	where $\tN(\d z,\d t):=N(\d z,\d t)-\lambda(\d z)\d t$ is the compensated Poisson random measure and $\alpha\in \U$ is an admissible strict control. Now, the Brownian motion and Poisson random measure are both idiosyncratic noises and there is no common noise. Hence, the mean field term appeared in the controlled dynamics \eqref{eq:seq3} is merely joint law of state and control instead of their conditional law. We aim to minimize the following cost functional over $\alpha\in\U$:
	\begin{equation}\label{eq:svalue3}
		J(\alpha):=\E\left[\int_0^Tf(X_t,\Law(X_t,\alpha_t),\alpha_t)\d t+g(X_T,\Law(X_T))\right]\to\inf_{\alpha\in\U}.
	\end{equation}
	It is not difficult to show that problem \eqref{eq:svalue3}-\eqref{eq:seq3} is well-defined under \assref{ass}. We then replace the probability measure on the copy measurable space with $\Pb$. In other words, we use $\Pb'=\Pb$ as the probability measure on the copy measurable space instead of $\Pb'=Q_{\omega}$ and other notations follow as before. Then, we can derive a similar result for the jump diffusion case and the adjoint BSDE can be obtained as 
	\begin{equation}\label{SBSDE_jump}
		\begin{cases}
			\displaystyle \d p_t=-\left\{\pa_xH(X_t,\alpha_t,\rho_t,p_t,P_t,K_t)+\E'[\pa_{x'}\delta H(X_t',\alpha_t',\rho_t,X_t,\alpha_t,p_t',P_t',K_t')] \right\}\d t\\
			\displaystyle\qquad\qquad+P_t\d W_t+\int_ZK_{t-}\tilde{N}(\d t,\d z),\\
			\displaystyle p_T=\pa_xg(X_T,\mu_T)+\E'\left[\pa_x\left(\frac{\delta g}{\delta\mu}(X_T',\mu_T)\right)(X_T)\right],
		\end{cases}
	\end{equation}
	where $\rho_t=\Law(X_t,\alpha_t)$ for $t\in[0,T]$.
	\begin{corollary}[Necessary Condition]\label{SNECj}
		Let $\alpha\in\U$ be an optimal strict control minimizing the cost $J$ in \eqref{eq:svalue3} over ${\U}$ and $X=(X_t)_{t\in[0,T]}$ be the state process \eqref{eq:seq3} under this $\alpha$. Then, there exists an $\Fb$-adapted and $\R^n\times\R^{n\times d}\times L^2((Z,\mathscr{Z},\lambda);\R^n)$-valued solution $(p,P,K)=(p_t,P_t,K_t)_{t\in[0,T]}$ to BSDE \eqref{SBSDE_jump}. Furthermore, we have, $\d t\times\d\Pb\text{-}\as$,
		\begin{equation}\label{eq:SSMP_j}
			\begin{aligned}
				&H(X_t,\alpha_t,\rho_t,p_t,P_t,K_t)+\E'\left[\delta H(X_t',\alpha_t',\rho_t,X_t,\alpha_t,p_t',P_t',K_t')\right]\\
				&\qquad\leq H(X_t,u,\rho_t,p_t,P_t,K_t)+\E'\left[\delta H(X_t',\alpha_t',\rho_t,X_t,u,p_t',P_t',K_t')\right],~\forall u\in U
			\end{aligned}
		\end{equation}
		with $\rho_t=\Law(X_t,\alpha_t)$ for $t\in[0,T]$.
	\end{corollary}
	
	On the other hand, we also have
	\begin{corollary}[Sufficient Condition]\label{SSUFj}
		Let $\alpha\in\U$ be a strict control, $X=(X_t)_{t\in[0,T]}$ be the corresponding controlled state process in \eqref{eq:seq3} under $\alpha\in\U$, and $(p,P,K)=(p_t,P_t,K_t)_{t\in[0,T]}$ be the corresponding adjoint process satisfying \eqref{eq:scond} and \eqref{SBSDE_jump}. If the extension of the Hamiltonian $H$, denoted by $\tH$, is $L$-convex in $(x,\xi)\in\R^n\times\Pc_2(\R^n\times\Pc(U))$ and $g$ is $L$-convex in $(x,\mu)\in\R^n\times\Pc_2(\R^n)$. Then, $\alpha\in\U$ is an optimal strict control if the equality in \eqref{eq:SSMP_j} holds.
	\end{corollary}
	Although it looks like that the adjoint BSDE \eqref{SBSDE_jump} of the jump diffusion resembles \eqref{SBSDE}, the solutions are in fact fundamentally different, primarily due to the different interpretations of $\E'$ in \eqref{SBSDE} and \eqref{SBSDE_jump}, see \thmref{smp_HJB} and Section~\ref{sec:example} for more details.
	
	\section{The HJB Equation under Poissonian Common Noise}\label{sec:HJB}
	
	This section is devoted to deriving the HJB equation for our extended MFC problem with strict open-loop controls and relate the PDE to our BSDE under the previous SMP. To derive the HJB equation, we generalize the method proposed in Motte and Pham~\cite{Pham} in a discrete-time setting to our continuous-time setting with Poissonian common noise. Firstly, we shall propose a new optimal control problem of Fokker-Planck equation whose state process is related to the Fokker-Planck equation of the original problem. Thus, we can heuristically establish the HJB equation from the new problem. 
	Secondly, by imposing some mild assumptions, we show that, given a smooth solution to the HJB equation, this candidate coincides with the value function of the original extended MFC problem, which can be almost viewed as a verification theorem. As a result, we  prove the equivalence between the value functions of these two control problems, and the two value functions should solve the HJB equation we derived in the last step.  The conditional law invariance can be obtained as a byproduct. Thirdly, we will relate the HJB equation to the BSDE derived from our previous stochastic maximum principle.
	
	Throughout the section, consider an enlarged probability space $(\Omega \times \Omega^3 \times \Omega^4, \mathcal{F} \otimes \mathcal{F}^3 \otimes \mathcal{F}^4, \mathbb{P} \times \mathbb{P}^3 \times \mathbb{P}^4)$, where $(\Omega, \mathcal{F}, \mathbb{P})$ is given in~\eqref{eq:probabspace}, $(\Omega^3, \mathcal{F}^3, \mathbb{P}^3)$ is a Polish space to support the initial state $\vartheta$, a square-integrable random variable in $L^2((\Omega,\F^3,\Pb),\R^n)$, and $(\Omega^4, \mathcal{F}^4, \mathbb{P}^4)$ is another Polish space to support an $\mathcal{F}^4$-measurable random variable $I_0$ with uniform distribution on $[0, 1]$. We still denote the enlarged probability space as $(\Omega, \mathcal{F}, \mathbb{P})$ with a slightly abuse of notation. By standard separation of the decimals of $I_0$ (c.f. Lemma 2.21 in Kallenberg \cite{Kallenberg2002}), there exists an i.i.d. sequence of $\mathcal{F}^4$-adapted uniform random variables $(I_k)_{k \in \mathbb{N}}$, independent of the initial state $\vartheta$, $W$ and $N$.

	Let us consider the dynamic version of extended MFC problem \eqref{eq:svalue1}-\eqref{eq:seq1} by varying the initial time and state. Given a strict open-loop control $\alpha \in \U$, the dynamic is given by, for $s\in(t,T]$,
	\begin{equation}\label{extended-MFC-dynamics}
		\begin{aligned}
			dX_s &= b(X_s, \Law((X_s, \alpha_s)|\G_s), \alpha_s)\d s + \sigma(X_s, \Law((X_s, \alpha_s)|\G_s), \alpha_s) \d W_s \\
			&\quad + \int_Z \gamma(X_{s-}, \Law((X_{s-}, \alpha_{s-}|\G_{s-}),\alpha_{s-}, z) \widetilde N(\d s, \d z),\quad X_t = \vartheta \in L^2(\Omega;\R^n).
		\end{aligned}
	\end{equation}
	The cost functional is given by, for $(t,\vartheta,\alpha)\in[0,T]\times L^2(\Omega;\R^n)\times\U$,
	\begin{equation}\label{eq:HJBobjectfcn}
		J_{ol}(t,\vartheta;\alpha):=\E\left[\int_t^Tf(X_s,\Law((X_s,\alpha_s)|\G_s),\alpha_s)\d s+g(X_T,\Law(X_T|\G_T))\right].
	\end{equation}
	We aim to minimize the above cost functional $J_{ol}$, and the value function is given by
	\begin{equation}\label{valuefunc}
		J^*_{ol}(t,\vartheta):=\inf_{\alpha\in\U}J_{ol}(t,\vartheta;\alpha).
	\end{equation}
	Throughout Section \ref{sec:HJB}, we denote the cost functional with strict open-loop controls as $J_{ol}$ to distinguish it from that of the new problem below.
	
	\begin{remark}\label{inva}
		{Due to Poissoian common noise and the joint conditional law dependence, it is not easy to prove the law invariance property of the value function \eqref{valuefunc} and establish the DPP. Therefore, we follow a different approach to derive the HJB equation. Our idea is to introduce a new controlled Fokker-Planck problem. It is easier to show that the value function of this new problem is law invariant and establish the corresponding HJB equation. The key difficulty is the equivalence between the original problem and the new controlled Fokker-Planck problem.}  
	\end{remark}
	
	The introduction of the controlled Fokker-Planck problem needs some notations. Denote by $\hat\U(\R^n)$ the set of transition kernels on $\R^n \times U$. That is, an element of $\hat\U(\R^n)$ is a measurable mapping $\pi:\R^n\mapsto\Pc(U)$ in the sense that $x\to \pi(x)(A)$ is measurable for all $A\in\mathscr{B}(U)$. Denote by $\widehat \U$ the set of $\mathbb{G}$-adapted process $\hat\alpha$ valued in $\hat\U(\R^n)$, and introduce measurable functions $\hat{b}: \R^n \times \Pc_2(\R^n) \times \hat\U(\R^n) \to \R^n$, $\hat\sigma: \R^n \times \Pc_2(\R^n) \times \hat\U(\R^n) \to \R^{n\times d}$ and $\hat{\gamma}:\R^n \times \Pc_2(\R^n) \times \hat\U(\R^n){\times Z\to\R^n}$ by
	\begin{equation}\label{aggregate}
		\begin{cases}
			\displaystyle \hat b(x, \mu, \hat u) = \int_U b(x, \mu \cdot \hat u, u)\hat u(x, \d u),\\
			\displaystyle~\hat \sigma\hat\sigma^{\T}(x, \mu, \hat u) = \int_U \sigma\sigma^{\T}(x, \mu \cdot \hat u, u)\hat u(x, \d u),\\
			\displaystyle ~\hat \gamma(x, \mu, \hat u,z) = \int_U \gamma(x, \mu \cdot \hat u, u,z)\hat u(x, \d u).
		\end{cases}
	\end{equation}
	For any $(\mu,\hat u,z)\in\Pc_2(\R^n)\times\hat \U(\R^n)\times Z$, let us consider the following linear mapping $I^{\mu, \hat u, z}:C_b(\R^n)\to C_b(\R^n)$ defined by
	\begin{equation}\label{eq:mappingI}
		I^{\mu, \hat u, z}(g)(\cdot):=\int_Ug\left(\cdot+\gamma(\cdot,\mu\cdot\hat u,u,z)\right)\hat u(\cdot,\d u),\quad \forall g\in C_b(\R^n).
	\end{equation}
	Denote by $I^{\mu, \hat u, z, *}:C_b^*(\R^n)\mapsto C_b^*(\R^n)$ the adjoint operator of the mapping $I^{\mu, \hat u, z}$ with $C_b^*(\R^n)$ being the dual space of $C_b(\R^n)$. 
	{ It will used for the characterization of the jump of the conditional state distribution flow. The following lemma provides an explicit characterization of the adjoint operator \( I^{\mu,\hat u,z,*} \), whose proof is postponed to Appendix~\ref{sec:appendix-B}.
		
		\begin{lemma}\label{measure_shift}
			For any \( (\mu,\hat u,z)\in \Pc_2(\R^n)\times\hat\U(\R^n)\times Z \) and \( \nu\in\Pc_2(\R^n) \), the image \( I^{\mu,\hat u,z,*}(\nu) \) belongs to \( \Pc_2(\R^n) \). Moreover, let \( (X,\alpha) \) be an \( \R^n\times U \)-valued random variable defined on a probability space \( (\Omega,\F,\Pb) \) with a sub-\(\sigma\)-algebra \( \G\subset\F \). Set \( \rho:=\Law((X,\alpha)\mid\G) \), \( \mu:=\Law(X\mid\G) \), and \( \hat\alpha(x)(\d u):=\Law(\alpha\mid\G,X=x)(\d u) \). Then it holds that, \(\Pb\)-a.s.,
			\begin{equation}\label{eq:shift}
				I^{\mu,\hat\alpha,z,*}\mu=\Law\bigl(X+\gamma(X,\rho,\alpha,z)\mid\G\bigr).
			\end{equation}
			Here \( I^{\mu,\hat\alpha,z,*}\mu(\omega) \) should be interpreted as \( I^{\mu(\omega),\hat\alpha(\omega),z,*}(\mu(\omega)) \) for any \( \omega\in\Omega \). Thus the adjoint operator \( I^{\mu,\hat\alpha,z,*} \) can be identified as a generalized measure shift.
		\end{lemma}
		
	}
	
	We also have the following remark  when the jump coefficient in the state process {depends on the control only via the conditional joint law.}
	\begin{remark}
		{When the jump coefficient $\gamma(\cdot)$ does not depend on the control variable, $I^{\mu, \hat u, z, *}$ becomes an explicit measure shift operator in terms of the push forward of the measure in the sense that $(I^{\mu, \hat u, z, *}\mu)(x \in B) = \mu(x + \gamma(x, \mu \cdot \hat u, z) \in B)$ for any Borel measurable set $B \subset \R^n$. }
\end{remark}

\paragraph{The controlled Fokker-Planck problem} The dynamics of the the controlled Fokker-Planck problem problem is a controlled stochastic Fokker-Planck equation (in the sense of distributions), which is described as 
\begin{equation}\label{equ:extended-FP-mu_t}
	d\mu_s = A_0^{\hat\alpha}\mu_s \d s + \int_{Z} A_1^{\hat\alpha} \mu_{s-} N(\d s, \d z),\quad \mu_t = \Law(\vartheta|\G_t).
\end{equation}
Here, $\hat\alpha  \in \widehat \U$ with $\hat\alpha_s(\cdot)=\Law(\alpha_s|X_s=\cdot,\G_s)$, and $A_0^{\hat\alpha}$ and $A_1^{\hat\alpha}$ are operators defined by
\begin{equation}\label{extended-operatorA_0*}
	\begin{cases}
		\displaystyle A_0^{\hat\alpha}\mu_s := -\pa_x\left(\left(\hat b(x, \mu_s, \hat\alpha_s) -  \langle\hat\gamma(x, \mu_s, \hat\alpha_s,\cdot), \lambda\rangle \right)\mu_s\right)+ \frac12 \pa_{xx} \left(\hat \sigma
		\hat\sigma^{\T}(x, \mu_s, \hat\alpha_s) \mu_s\right),\\
		\displaystyle A_1^{\hat\alpha} \mu_{s} := I^{\mu_{s},\hat\alpha_{s}, z, *}\mu_{s}-\mu_{s}.
	\end{cases}
\end{equation}
The cost functional and the value function are given by
\begin{equation}\label{liftMFC-value}
	J_{fp}(t,\mu;{\hat\alpha}) = \E\left[\int_t^T\hat f(\mu_s, \hat\alpha_{s})\d s + \hat g(\mu_T)\right],\quad  J_{fp}^*(t, \mu) = \inf_{\hat\alpha \in \widehat\U} J_{fp}(t, \mu; {\hat\alpha}),
\end{equation}
where we define
\begin{equation}\label{running-cost}
	\hat f(\mu, \hat\alpha) := \int_{\R \times U} f(x, \mu \cdot \hat\alpha, u)(\mu \cdot \hat\alpha)(\d x, \d u),~~\text{and}~~~ \hat g(\mu) := \int_{\R^n} g(x)\mu(\d x).
\end{equation}
Note that the the controlled Fokker-Planck problem problem treats the conditional state distribution as the new state and regards the common noise as the true sole noise affecting the system, hence both the new state and the new control are only adapted to Poissonian common noise filtration. Furthermore, the value function of the the controlled Fokker-Planck problem problem is law invariant.

\begin{remark}\label{jumps}
	By virtue of \eqref{equ:extended-FP-mu_t}, we highlight that the measure $\mu_s$ has jumps caused by the Poissonian common noise, which differs substantially from the case of jump diffusion model in which the measure $\mu_s$ is continuous in time $s$ (see the discussion in Burzoni et al. \cite{Burzoni2020} and Guo et al.~\cite{XWei}). Moreover, from \eqref{equ:extended-FP-mu_t}, we can observe that,  the state $X_s$ and the conditional law $\mu_s$ have exactly the same jump times.
\end{remark}

\subsection{Formal derivation of the HJB equation}
In this section, we drive the HJB equation from the the controlled Fokker-Planck problem problem. Let us introduce the space $C^{1;1,1}([0,T]\times\Pc_2(\R^n))$ as the set of mapping $J:[0,T]\times\Pc_2(\R^n)\mapsto\R$ such that $\pa_tJ(t,\mu),\pa_{\mu}J(t,\mu)(x)$ and $\pa_x\pa_{\mu}J(t,\mu)(x)$ are jointly continuous, where $\pa_{\mu}J$ stands for the $L$-derivative of the mapping $J$. We first give below the It\^{o}'s rule, whose proof is standard and hence omitted.

\begin{lemma}[It\^{o}'s formula]\label{Ito}
	Let the measure-valued process $\mu=(\mu_s)_{s\in[t,T]}$ satisfy the dynamics \eqref{equ:extended-FP-mu_t}, and assume that $J\in C^{1;1,1}([0,T]\times\Pc_2(\R^n))$. Then, it holds that
	\begin{equation*}
		\begin{aligned}
			\d J(s,\mu_s)&=\left[\pa_sJ(s,\mu_s)+\int_{\R^n}\left((\hat b(x,\mu_s,\hat\alpha_s)-\langle\hat\gamma(x,\mu_s,\hat\alpha_s,\cdot),\lambda\rangle)\cdot\pa_{\mu}J(s,\mu_s)(x)\right.\right.\\
			&\quad\left.\left.+\frac12\tr(\hat\sigma\hat\sigma^{\T}(x,\mu_s,\hat\alpha_s)\pa_x\pa_{\mu} J(s,\mu_s)(x))\right)\mu_s(\d x)\right]\d s\\
			&\quad+\int_Z\left(J(s,I^{\mu_{s-},\hat\alpha_{s-},z, *}\mu_{s-})-J(s,\mu_{s-})\right)N(\d s,\d z).
		\end{aligned}
	\end{equation*}
\end{lemma}

\begin{remark}\label{Itoremark}
	It follows from \lemref{measure_shift} that, when the control $\hat\alpha$ can be related to an original control $\alpha$ as in \eqref{eq:hatalpha}, the jump term can be rewritten as a measure shift form:
	\begin{equation*}
		\begin{aligned}
			\d J(s,\mu_s)&=\left[\pa_sJ(s,\mu_s)+\int_{\R^n}\left((\hat b(x,\mu_s,\hat\alpha_s)-\langle\hat\gamma(x,\mu_s,\hat\alpha_s,\cdot),\lambda\rangle)\cdot\pa_{\mu}J(s,\mu_s)(x)\right.\right.\\
			&\quad\left.\left.+\frac12\tr(\hat\sigma\hat\sigma^{\T}(x,\mu_s,\hat\alpha_s)\pa_x\pa_{\mu}J(s,\mu_s)(x))\right)\mu_s(\d x)\right]\d s\\
			&\quad+\int_Z\left(J(s,\Law(X_{s-} + \gamma(X_{s-},\mu_{s-}\cdot\hat\alpha_{s-},\alpha_{s-},z)|\G_{s-}))-J(s,\Law(X_{s-}|\G_{s-}))\right)N(\d s,\d z).
		\end{aligned}
	\end{equation*}
\end{remark}

It is easy to verify that $\mu=(\mu_t)_{t\in[0,T]}$ satisfies the flow property whenever \eqref{equ:extended-FP-mu_t} has a unique solution. We assume that the following DPP for the value function $J_{fp}^*$ holds in order to heuristically derive the HJB equation, for $h>0$,
\begin{equation}\label{dpp:extended-mfc-poisson-common-noise}
	J_{fp}^*(t, \mu) = \inf_{\hat\alpha  \in \widehat \U}\E\left[\int_t^{t + h} \hat f(\mu_s, \hat\alpha_s)\d s + J_{fp}^*(t + h, \mu_{t + h})\right].
\end{equation}
Suppose that $J_{fp}^*$ has sufficient regularity. Then, from \lemref{Ito}, DPP~\eqref{dpp:extended-mfc-poisson-common-noise},
we arrive at the following dynamic programming equation satisfied by the value function $J_{fp}^*(t,\mu)$:
\begin{equation}\label{extended-HJB-1-poisson-common-noise}
	\begin{aligned}
		&\pa_tJ_{fp}^*(t,\mu) +\inf_{\hat\alpha \in \hat\U(\R^n)}\biggl[\int_{\R^n}\Big(\hat f(\mu, \hat\alpha) + \big(\hat b(x, \mu, \hat\alpha) - \langle \hat\gamma(x, \mu , \hat\alpha,\cdot), \lambda \rangle \big)\cdot \pa_{\mu}J_{fp}^*(t,\mu)(x)\\
		&~~+ \frac{1}{2} \tr\left(\hat\sigma\hat\sigma^{\T}(x, \mu, \hat\alpha)\pa_x\pa_{\mu}J_{fp}^*(t,\mu)(x) \right)\Big)\mu(dx)+ \int_Z \left(J_{fp}^*(t,I^{\mu,\hat\alpha,z,*}\mu)  - J_{fp}^*(t,\mu)\right) \lambda(dz)\biggl]= 0,\\
		&J_{fp}^*(T,\mu)=\hat g(\mu),
	\end{aligned}
\end{equation}
where $\hat b$, $\hat\sigma$, $\hat\gamma$ and $\hat f$ are given in \eqref{aggregate} and \eqref{running-cost}, respectively.
Similar to the result in Motte and Pham~\cite{Pham}, we provide several alternative forms of the HJB equation \eqref{extended-HJB-1-poisson-common-noise}.
\begin{lemma}\label{lemma1}
	Let us define the following operator on $C^{1;1,1}([0,T]\times\Pc_2(\R^n))$ that
	\begin{align}\label{operator-T-hatalpha}
		\mathcal{T}^{\hat\alpha}J(t,\mu)&:=\left\{\int_{\R}\left[\hat f(\mu,\hat\alpha)+(\hat b(x,\mu,\hat\alpha)-\langle\hat\gamma(x,\mu,\hat\alpha,\cdot),\lambda\rangle)\pa_{\mu}J(t,\mu)(x)\right.\right.\\
		&~~\left.\left.+\frac12\tr\left(\hat\sigma\hat\sigma^{\T}(x,\mu,\hat\alpha)\pa_x\pa_{\mu}J(t,\mu)(x)\right)\right]\mu(\d x)+\int_Z \left[J(t,I^{\mu,\hat{\alpha},z, *}\mu)-J(t,\mu)\right]\lambda(\d z)\right\}. \nonumber
	\end{align}
	Define $\mathcal{T} J(t, \mu) =\inf_{\hat\alpha\in\hat \U(\R)} \mathcal{T}^{\hat\alpha} J(t, \mu)$. Then, for any $J\in C^{1;1,1}([0,T]\times\Pc_2(\R^n))$, it holds that $\mathcal{T}J(t,\mu)={\rm T}J(t,\mu)=\mathbb{T}J(t,\mu)$ for all $(t,\mu)\in [0,T]\times\Pc_2(\R^n)$.
	Here, the operators ${\rm T}$ and $\mathbb{T}$ are defined respectively by, for $(t,\mu)\in [0,T]\times\Pc_2(\R^n)$,
	{\small
		\begin{equation*}
			\begin{aligned}
				&{\rm T}J(t,\mu):=\inf_{\alpha\in L(\R\times[0,1];U)}\Bigg\{\E\bigg[f(\vartheta,\Law(\vartheta,\alpha(\vartheta,I)),\alpha(\vartheta,I))+(b(\vartheta,\Law(\vartheta,\alpha(\vartheta,I)),\alpha(\vartheta,I))\\
				&-\langle\gamma(\vartheta,\Law(\vartheta,\alpha(\vartheta,I)),\alpha(\vartheta,I),\cdot),\lambda\rangle)\pa_{\mu}J(t,\mu)(\vartheta)\!\!+\!\!\frac12\tr\left(\sigma\sigma^{\T}(\vartheta,\Law(\vartheta,\alpha(\vartheta,I)),\alpha(\vartheta,I))\pa_x\pa_\mu J(t,\mu)(\vartheta)\right)\bigg]\\
				&\quad+\int_Z \left(J(t,\Law(\vartheta+\gamma(\vartheta,\Law(\vartheta,\alpha(\vartheta,I)),\alpha(\vartheta,I),z)))-J(t,\mu)\right)\lambda(\d z)\Bigg\},\\
				&\mathbb{T}J(t,\mu):=\inf_{\alpha\in L(\Omega,U)}\Bigg\{\E\bigg[f(\vartheta,\Law(\vartheta,\alpha),\alpha)+(b(\vartheta,\Law(\vartheta,\alpha),\alpha)-\langle\gamma(\vartheta,\Law(\vartheta,\alpha),\alpha,\cdot),\lambda\rangle)\pa_{\mu}J(t,\mu)(\vartheta)\\
				& +\frac12\tr\left(\sigma\sigma^{\T}(\vartheta,\Law(\vartheta,\alpha),\alpha)\pa_x\pa_\mu J(t,\mu)(\vartheta)\right)\bigg]\!\!+\!\!\int_Z \left(J(t,\Law(\vartheta+\gamma(\vartheta,\Law(\vartheta,\alpha),\alpha,z)))-J(t,\mu)\right)\lambda(\d z)\bigg\},
			\end{aligned}
		\end{equation*}
	}
	where the random variable $(\vartheta,I)\sim \mu\times {\rm U}(0,1)$, we denote by $L(\R\times[0,1];U)$ the set of measurable functions from $\R \times [0, 1]$ to $U$, and $L(\Omega, U)$ is denoted as the set of all measurable random variables valued in $U$.
\end{lemma}

\begin{remark}\label{joint_law}
	It is straightforward to verify that the operator $\mathbb{T}$ can be equivalently written by
	\begin{equation*}
		\begin{aligned}
			\mathbb{T}J(t,\mu)&=\inf_{\substack{\rho\in\Pc_2(\R^n\times U)\\ \rho_1=\mu}}\int_{\R^n\times U}\Big(f(x,\rho,u)+(b(x,\rho,u)-\langle\gamma(x,\rho,u,\cdot),\lambda\rangle\pa_{\mu}J(t,\mu)(x) \\
			&\left.+\frac12\tr\left(\sigma\sigma^{\T}(x,\rho,u)\pa_x\pa_{\mu}J(t,\mu)(x)\right)\right)\rho(\d x,\d u)+\int_Z\left(J(t,I^{\mu,\hat\alpha,z, *}\mu)-J(t,\mu)\right)\lambda(\d z),
		\end{aligned}
	\end{equation*}
	where $\rho_1\in\Pc_2(\R^n)$ stands for the first marginal of $\rho$ and $\hat\alpha$ is the regular conditional distribution of the second variable given the first variable (or more precisely, it is the Radon-Nikodym derivative of $\rho$ w.r.t. $\mu$ when $x\in\R^n$ is given).
\end{remark}

\subsection{Equivalence between the original problem and the controlled Fokker-Planck problem}
This section is devoted to showing $J_{ol}^*(t, \vartheta) = J_{fp}^*(t, \mu)$ and hence two value functions share the same HJB equation. 
\subsubsection{From the original problem to the controlled Fokker-Planck problem}\label{LCP}\label{sec:original-less-the controlled Fokker-Planck problem}
In this section, we prove the easier part $J_{fp}^*(t,\mu)\leq J^*_{ol}(t,\vartheta)$.  The proof is mainly based on the disintegration of the conditional joint law, and the fact that every open-loop control $\alpha \in \U$ in the original problem induces a kernel-valued control $\hat\alpha \in \widehat \U$.

\begin{lemma}\label{lem:inequality-from-op-fp} For every $\alpha \in \U$, there exists a kernel-valued control $\hat \alpha \in \widehat \U$ in the form \begin{equation}\label{eq:hatalpha}
		\hat\alpha_s(x)(\d u) := \Law(\alpha_s|\G_s,X_s=x)(\d u),
	\end{equation} such that $J_{ol}(t, \vartheta; \alpha) = J_{fp}(t, \mu; \hat\alpha)$. Consequently,  $J_{fp}^*(t,\mu)\leq J^*_{ol}(t,\vartheta)$.
\end{lemma}
\noindent{\it Proof.}\quad
Let $\mu_s$ be the regular conditional distribution of $X_s$ in \eqref{extended-MFC-dynamics} given $\G_s$ under $\alpha \in \U$. we first prove that $\mu_s$ satisfies a stochastic Fokker-Planck equation controlled by $\hat\alpha$. For any $\phi\in { C_b^2(\R^n)}$, applying It\^o's formula to $\phi(X_s)$ from $t$ to $s\in[t,T]$, and taking the conditional expectation on $\G_s$, we obtain
\begin{equation}\label{ito-g}
	\begin{aligned}
		&\E[\phi(X_s)|\G_s] - \E[\phi(X_{t})|\G_t]\\
		&\quad = \int_{t}^s \E\left[\left(b(X_r, \Law((X_r, \alpha_r)|\G_r), \alpha_r)\pa_x \phi(X_r)- \langle \gamma(X_r, \Law((X_r, \alpha_r)|\G_r),\alpha_{r} ,\cdot),\lambda \rangle\right) \pa_x \phi(X_r) \right.\\
		&\qquad\left. + \frac{1}{2} \tr\left[\sigma\sigma^{\T}(X_r,  \Law((X_r, \alpha_r)|\G_r), \alpha_r)  \pa_{xx}\phi(X_r)\right]\Big)\middle|\G_r\right]\d r \\
		& \qquad + \int_{0}^t \int_{Z} \E\left[\left(\phi(X_{r-} + \gamma(X_{r-},  \Law((X_{r-}, \alpha_{r-})|\G_r),\alpha_{r-}, z)) - \phi(X_{r-})\right)\middle|\G_{r-}\right] N(\d r, \d z).
	\end{aligned}
\end{equation}
Observing that the right-hand side of \eqref{ito-g} is expressed in terms of $ \Law((X_r, \alpha_r)|\G_r)$, we arrive at
\begin{equation*}
	\begin{aligned}
		& \E\left[\left(b(X_r, \Law((X_r, \alpha_r)|\G_r), \alpha_r)  \pa_x \phi(X_r)- \langle \gamma(X_r, \Law((X_r, \alpha_r)|\G_r), \cdot),  \lambda \rangle \right)\pa_x \phi(X_r) \right.\nonumber\\
		&\qquad\qquad\left. + \frac{1}{2} \tr\left(\sigma\sigma^{\T}(X_r,  \Law((X_r, \alpha_r)|\G_r), \alpha_r)  \pa_{xx}\phi(X_r)\right)\Big)\middle|\G_r\right]\\
		&\quad= \int_{\R^n \times U} \Big\{\left(b(x,\Law((X_r, \alpha_r)|\G_r), u)
		- \langle \gamma(x, \Law((X_r, \alpha_r)|\G_r), u,\cdot), \lambda\rangle\right)\pa_x\phi(x)\\
		&\qquad+ \frac{1}{2}\tr\left[\sigma\sigma^{\T}(x, \Law((X_r, \alpha_r)|\G_r), u)\pa_{xx}\phi(x)\right] \Big\}\Law((X_r, \alpha_r)|\G_r)(\d x, \d u),
	\end{aligned}
\end{equation*}
and
\begin{equation*}
	\begin{aligned}
		&\E\left[\left(\phi(X_{r-} + \gamma(X_{r-}, \Law((X_{r-}, \alpha_{r-})|\G_{r-}), \alpha_{r-},z)) - \phi(X_{r-})\right)\middle|\G_{r-}\right]\\
		&\qquad= \int_{\R^n \times U}\left(\phi(x + \gamma(x, \Law((X_{r-}, \alpha_{r-})|\G_{r-}),u, z)) - \phi(x)\right)\Law((X_{r-}, \alpha_{r-})|\G_{r-})(\d x,\d u),
	\end{aligned}
\end{equation*}
where we recall $\langle\gamma(x,\mu_r\cdot\hat\alpha_r,u,\cdot),\lambda\rangle=\int_Z\gamma(x,\mu_r\cdot\hat\alpha_r,u,\cdot)\lambda(\d z)$. 
It results from the Bayes' formula that
\begin{equation}\label{eq:Bayesfor}
	\Law((X_s, \alpha_s)|\G_s)(\d x,\d u) = \mu_s(\d x) \cdot \hat\alpha_s(x)(\d u).
\end{equation}
In view of \equref{aggregate}, we have for $r\in[t,s]$ and $\phi\in C_b^2(\R^n)$,
\begin{equation*}
	\begin{aligned}
		&\E\left[\left(b(X_r, \Law((X_r, \alpha_r)|\G_r), \alpha_r)  \pa_x \phi(X_r)- \langle \gamma(X_r, \Law((X_r, \alpha_r)|\G_r), \cdot),  \lambda \rangle\right) \pa_x\phi(X_r) \right.\\
		&\qquad\left. + \frac{1}{2} \tr\left[\sigma\sigma^{\T}(X_r,  \Law((X_r, \alpha_r)|\G_r), \alpha_r)  \pa_{xx}\phi(X_r)\right]\Big)\middle|\G_r\right]\\
		&\quad = \int_{\R^n \times U} \Big\{\big(b(x,\Law((X_r, \alpha_r)|\G_r), u)
		- \langle \gamma(x, \Law((X_r, \alpha_r)|\G_r), u,\cdot), \lambda\rangle\big)\pa_x\phi(x)\\
		&\qquad + \frac{1}{2}\tr\left[\sigma\sigma^{\T}(x, \Law((X_r, \alpha_r)|\G_r), u)\pa_{xx}\phi(x)\right] \Big\}\Law((X_r, \alpha_r)|\G_r)(\d x, \d {u})\\
		&\quad=\int_{\R^n} \left\{\left(\hat b(x,\mu_s,\hat \alpha_s)-\langle\hat\gamma(x,\mu_s,\hat\alpha_s,\cdot),\lambda\rangle\right)\pa_x\phi(x)+\frac12\tr\left[\hat\sigma\hat\sigma^{\T}(x,\mu_r,\hat\alpha_r)\pa_{xx}\phi(x)\right] \right\}\mu_r(\d x)\\
		&\quad=\langle \phi,A_0^{\hat\alpha}\mu_r\rangle,
	\end{aligned}
\end{equation*}
as well as
\begin{equation*}
	\begin{aligned}
		&\E\left[\left(\phi(X_{s-} + \gamma(X_{r-}, \Law((X_{r-}, \alpha_{r-})|\G_{r-}), \alpha_{r-},z)) - \phi(X_{r-})\right)\middle|\G_{r-}\right]\\
		&\quad = \int_{\R^n \times U} \left(\phi(x + \gamma(x, \Law((X_{r-}, \alpha_{r-})|\G_{r-}),u, z)) - \phi(x)\right)\Law((X_{r-}, \alpha_{r-})|\G_{r-})(\d x, \d u)\\
		&\quad = \int_{\R^n} \left(I^{\mu_{r-},\hat\alpha_{r-},z}(\phi)(x)-\phi(x)\right)\mu_{r-}(\d x) =\langle \phi,I^{\mu_{r-},\hat\alpha_{r-}, z, *}\mu_{r-}-\mu_{r-}\rangle=\langle \phi,A_1^{\hat\alpha}\mu_{r-}\rangle.
	\end{aligned}
\end{equation*}
Furthermore, by tower property, for $(t,\vartheta,\alpha)=[0,T]\times L^2(\Omega;\R^n)\times\U$, we have
\begin{equation}\label{value-reformulation}
	\begin{aligned}
		J_{ol}(t, \vartheta; \alpha)
		&= \E\left[\int_t^T \E\left[f(X_s,\Law((X_s, \alpha_s)|\G_s), \alpha_s)\middle|\G_s\right]\d s + \E\left[g(X_T, \Law(X_T|\G_T))\middle|\G_T\right]\right] \\
		&= \E\left[\int_t^T \hat f\big(\mu_s, \hat\alpha_s\big)ds + \hat g(\mu_T)\right] = J_{fp}(t, \mu; \hat\alpha).
	\end{aligned}
\end{equation}
The desired conclusion holds.
\hfill$\Box$

\subsubsection{From the controlled Fokker-Planck problem to the original problem}\label{sec:the controlled Fokker-Planck problem-less-original}
In the previous subsection, we can construct a new controlled  Fokker-Planck problem from an original problem. However, for each kernel-valued control $\hat\alpha \in \widehat \U$, it is challenging to find a corresponding open-loop control in $\U$. Our strategy is to focus on feedback kernel-valued controls and to show that for each near optimal feedback kernel-valued control, there exists an associated piecewise constant control that yields the reverse inequality $J^*_{ol}(t, \vartheta) \leq J^*_{fp}(t, \mu)$. To this end, we will impose the next assumptions.
\begin{ass}[Regularity]\label{Smoothness}
	There exists a classical solution $\hat{J}\in C^{1;1,1}([0,T]\times \Pc_2(\R^n))$ to the HJB equation $\pa_t J(t,\mu)+\mathcal{T}J(t,\mu)=0$ with $J(T,\mu)=\hat g(\mu)$.
\end{ass}

We also make the following assumption:
\begin{ass}[Measurability]\label{measurable}
	For any $\epsilon>0$, there exists a  $\pi_{\epsilon}(t, \mu)(\cdot)$ is an $\epsilon/T$-optimal control of $\hat J$ for some measurable mapping $\pi_\epsilon: [0, T] \times \mathcal{P}(\R^n) \to \hat\U(\R^n)$ in the sense that $\mathcal{T}^{\pi_{\epsilon}(t, \mu)}\hat J(t,\mu)<\mathcal{T}\hat J(t,\mu)+\epsilon/T$ for all $(t,\mu)\in[0,T]\times\Pc_2(\R^n)$.
\end{ass}

When DPP for the  controlled Fokker-Planck problem holds and the new value function $J_{fp}^*$ is smooth in the sense that $J_{fp}^*\in C^{1;1,1}([0,T]\times\Pc_2(\R^n))$, \assref{Smoothness} is automatically fulfilled.  \assref{measurable} holds if there exists a measurable mapping  $\pi^*$$[0, T] \times \mathcal{P}(\R^n) \to \hat\U(\R^n)$ attaining the infimum in the sense that $\mathcal{T} \hat J(t, \mu) = \mathcal{T}^{\pi^*(t, \mu)} \hat J(t, \mu)$. It is easy to verify that the LQ-type MFC problem satisfies both \assref{Smoothness} and \assref{measurable} (c.f. Section~\ref{sec:example}). As a consequence, we obtain the verification theorem for the controlled Fokker-Planck problem.

\begin{prop} \label{thm:verification} Under {\rm \assref{Smoothness}} and {\rm \assref{measurable}}, the control in feedback form $(\pi_{s, \epsilon})_s = (\pi_\epsilon(s, \mu_s))_s \in \widehat\U$ is an $\epsilon$-optimal control of the controlled Fokker-Planck problem, i.e. $J_{fp}(t, \mu; \pi_\epsilon) \leq J_{fp}^*(t, \mu) + \epsilon$.
\end{prop}
\noindent{\it Proof.}\quad
Under \assref{Smoothness}, by applying It\^o's rule to $\hat J(t,\mu_t)$, where $(\mu_s)_{s\in[t,T]}$ satisfies the controlled dynamics \eqref{equ:extended-FP-mu_t} with any $\hat\alpha\in\widehat\U$ , one can easily conclude that \begin{equation}\label{eq:valueineq}
	\hat J(t,\mu)\leq J_{fp}^*(t,\mu),\quad \forall (t,\mu)\in [0,T]\times\Pc_2(\R^{n}).
\end{equation}
On the other hand,  it follows from It\^o's rule  to $\hat J(t, \mu_t)$, where $\mu_t$ satisfies \eqref{equ:extended-FP-mu_t} under the feedback kernel-valued control $\pi_{\epsilon}$
\begin{equation}\label{eq:itoformula}
	\begin{aligned}
		&\E\left[\hat J(T,\mu_T)-\hat J(t,\mu)\right]=\E\left[\int_t^T\left(\pa_s\hat J(s,\mu_s)+\mathcal{T}^{\pi_{\epsilon}(s, \mu_s)}\hat J(s,\mu_s)-\hat f(s, \mu_s,\pi_{\epsilon}(s, \mu_s))\right)\d s\right]\\
		&=\E\left[\int_t^T\left(\pa_s\hat J(s,\mu_s)+\mathcal{T}\hat J(s,\mu_s)\right)\d s\!\!+\!\!\int_t^T\left(\mathcal{T}^{\pi_{\epsilon}(s, \mu_s)}\hat J(s,\mu_s)-\mathcal{T}\hat J(s,\mu_s)\right)\d s\!\!-\!\!\int_t^T\hat f(\mu_s,\pi_{\epsilon}(s, \mu_s))\d s\right]\\
		&<\epsilon+\E\left[J_{fp}(T,\mu_T; \pi_{\epsilon})\right]-J_{fp}(t,\mu; \pi_{\epsilon}),
	\end{aligned}
\end{equation}
where in the last inequality we have exploited the fact that
\begin{equation*}
	J_{fp}(t,\mu; \pi_{\epsilon})
	=\E\left[\int_t^T\hat f(\mu_s,\pi_{\epsilon}(s, \mu_s))\d s+J_{fp}(T,\mu_T; \pi_{\epsilon})\right].
\end{equation*}
By noting $J_{fp}(T,\mu_T; \pi_{\epsilon})=\hat g(\mu_T)=\hat J(T,\mu_T)$, we have $J_{fp}(t, \mu; \pi_{\epsilon}) < \hat J(t, \mu) + \epsilon$ by \eqref{eq:itoformula}.
\hfill $\Box$

\medskip 

Thanks to Theorem \ref{thm:verification}, we focus on a subset of $\widehat \U$ that consists of all feedback controls $\pi=(\pi_t)_{t\in[0,T]}$ with $\pi_t=\pi(t,\mu_t)$ for some measurable mapping $\pi: [0, T] \times \mathcal{P}(\R^n) \to \U(\R^n)$. There may be some confusion between the control process $\pi=(\pi_t)_{t\in[0,T]}$ and { the merely measurable} mapping $\pi:[0, T] \times \Pc_2(\R^n)\mapsto\hat \U(\R^n)$, but readers can easily tell apart when we use this notation. Indeed, for these feedback controls in $\widehat \U$, we first borrow the idea from Motte and Pham \cite{Pham} to introduce the lifted randomized feedback policy.
	
	{\begin{definition}[Ideal Lifted Randomized Feedback Policy] \label{def:ideal_randomized_feedback} A feedback kernel-valued control $\pi\in L([0, T] \times \Pc_2(\R^n);\hat \U(\R^n))$ is called an ideal lifted randomized feedback policy if there exists a measurable mapping $a\in L([0, T] \times \Pc_2(\R^n)\times\R^n\times[0,1];U)$, such that for all $(t, \mu)\in [0, T] \times\Pc_2(\R^n)$ $a(t, \vartheta, \mu, I_t) \sim \pi(t, \mu)$ with $(I_t)_{t \in [0, T]}$ a continuum of i.i.d. uniform random variables, independent of $\vartheta \sim \mu$.
		\end{definition}
		
		By Definition ~\ref{def:ideal_randomized_feedback}, if $\pi$ is an ideal lifted randomized feedback policy, then $a(s, X_s, \mu_s, I_s)$, $t \leq s \leq T$ is an open-loop control in $\U$. In this case, we have $J_{fp}(t, \mu; \pi) = J_{ol}(t, \vartheta; a)$.  However, the ideal lifted randomized policy in Definition \ref{def:ideal_randomized_feedback} requires uncountably-many independent random variables $\{I_s\}_{t \leq s \leq T}$. Corollary 4.3 in Sun \cite{Sun2006} points out that when the random variables $\{I_s\}_{s \in [t, T]}$ are  pairwise independent, the sample path $s \mapsto I_s$ fails to be Lebesgue measurable. This implies the integral $\int_t^T f(s, X_s, \mu_s, a(s, X_s, \mu_s, I_s)))\d s$  involved in the original problem is ill-defined under standard Lebesgue framework, which raises some measure-theoretical issues.  While Sun \cite{Sun2006} solves this issue by extending the Lebesgue measure (in $s \in [t, T]$) via a Fubini extension so as to accommodate uncountable pairwise independent random variables, this approach relies on a nonstandard measure-theoretical framework, under which all (stochastic) integrals should be redefined. As a result, Sun \cite{Sun2006}'s Fubini extension method is not suitable for the classical stochastic control framework. Motivated by \cite{Jiazhou2023, Jiaetal2025} in a RL context, we can instead associate the feedback kernel-valued control with a sequence of piecewise constant controls, which only require countable independent random variables and lead to the following definition.}
	
	\begin{definition}[Lifted Randomized Feedback Policy]\label{randomized_feedback}
		A policy $\pi\in L([0, T] \times \Pc_2(\R^n);\hat \U(\R^n))$ is called a lifted randomized feedback policy if there exists a measurable mapping $a\in L([0, T] \times \Pc_2(\R^n)\times\R^n\times[0,1];U)$, such that for any discretized time grid $\mathcal{D} = \{0 = t_0 < \cdots < t_K = T\}$, $a(t_k,\mu,\vartheta,I_k)\sim\pi(t_k, \mu)$, for all $(t_k, \mu)\in \mathcal{D} \times\Pc_2(\R^n)$ with $(I_k)_{k=0}^K$ a sequence of i.i.d. uniform random variables, independent of $\vartheta \sim \mu$.
	\end{definition}


	Using {\rm \defref{randomized_feedback}}, we aim to show in the next result ({\rm \propref{prop:diff-ol-fp}}) that if an $\epsilon/T$-optimal control $\pi_\epsilon$ for the controlled Fokker-Planck problem is given in the form of  a lifted randomized feedback policy, then there exists a sequence of piecewise constant open-loop controls such that the corresponding cost functionals of the original control problem approximate that of the controlled Fokker–Planck problem arbitrarily well. {Our proof of {\rm \propref{prop:diff-ol-fp}} is based on a PDE approach, similar to that developed in Theorem 4.1 in Jia et al. \cite{Jiaetal2025}. As a consequence, additional regularity conditions are required, as specified in the following assumption.
	}
	
	\begin{ass}\label{ass:regularity-Jh} {Let $\pi_\epsilon$ be $\epsilon/T$-optimal policy given in {\rm \assref{measurable}}, and assume further that it is also lifted randomized feedback policy in Definition \ref{randomized_feedback}.} 
		\begin{itemize}
			\item [{\rm (i)}] $\hat f(\cdot, \cdot, \pi_\epsilon(\cdot, \cdot)) \in C^{1;1, 1}(\mathcal{P}_2(\R^n))$ and $\hat g \in C^{1, 1}(\mathcal{P}(\R^n))$, with $\hat f$ and $\hat g$ defined in \eqref{running-cost}.
			\item[{\rm(ii)}]
			$b$, $\sigma$, $\gamma$ and $\lambda$ are sufficiently regular, such that for each $h(\cdot) \in \{\hat f(\cdot,\pi_{\epsilon}(\cdot,\cdot)), \hat g(\cdot)\}$ and every $t' \in [0, T]$, the PDE 
			\begin{equation}\label{PDE:Jh}
				\pa_t J(t,\mu)+{\rm B}^{\pi_\epsilon(t, \mu)}J(t,\mu)=0, t \in [0, t'], \;
				J(t', \mu) = h(t', \mu)
			\end{equation}
			has a classical solution $J_{h} \in C^{1, 1}([0, t'] \times \mathcal{P}(\R^n))$ satisfying
			\begin{equation*}
				\begin{aligned}
					&|J_h(t, \mu) - J_h(t, \mu')| + |\pa_t J_h(t, \mu) - \pa_t J_h(t, \mu')| \leq C\big({\mathcal{W}_{2;\R^n}(\mu,\mu')^2} + |x- x'|^2\big),\\
					&|\pa_\mu J_h(t, \mu)(x) - \pa_\mu J(t, \mu')(x')|  + |\pa_x\pa_\mu(t, \mu)(x) - \pa_x\pa_\mu(t, \mu')(x')| \leq  C\big({\mathcal{W}_{2;\R^n}(\mu,\mu')} + |x- x'|\big).
				\end{aligned}
			\end{equation*}
			Here, the operator  ${\rm B}^{\hat\alpha}$ is defined on $C^{1; 1, 1}([0, T] \times \mathcal{P}( \R^n))$ by 
			$B^{\hat\alpha} J(t, \mu) = \mathcal{T}^{\hat\alpha} J(t, \mu) - \hat f(t, \mu, \hat\alpha)$, where the operator $\mathcal{T}^{\hat\alpha}$ is given in \eqref{operator-T-hatalpha}.
		\end{itemize}
	\end{ass}
	
	{
		For LQ type MFC problem (c.f. Section~\ref{sec:example}), the optimal feedback policy $\pi^*$ exists and is given by a Dirac measure $\delta_{\alpha^*(t, x, \mu)}$. Consequently, \assref{ass:regularity-Jh} is satisfied by taking $\pi_\epsilon$ as $\pi^*$. In general, (ii) in \assref{ass:regularity-Jh} corresponds to a linear PDE on the Wasserstein space and admits a Feynman-Kac representation. It is expected that a strategy based on Malliavin calculus, similar to that of Buckdahn et al. \cite{BLPR2017} can be adapted to ensure regularity of $J_h$ in our setting with Poissonian common noise, provided that the near-optimal policy $\pi_\epsilon$ and all coefficients are sufficiently regular. In addition, under {\rm \assref{Smoothness}}, {\rm \assref{measurable}}, and under some structural assumption on the operator $\mathcal{T}$, the near-optimal policy $\pi_\epsilon$ inherits the regularity from the classical solution of HJB equation in {\rm \assref{Smoothness}}.
	}
	\begin{prop}\label{prop:diff-ol-fp}
		Let {\rm \assref{ass}}, {\rm \assref{measurable}} and {\rm \assref{ass:regularity-Jh}} hold. For any $\epsilon >0$,  assume that $\pi^\epsilon$ is a $\epsilon/T$-optimal lifted randomized feedback policy. Then there exists a time grid $\mathcal{D}$ of $[t, T]$ with mesh size $K$ and a piecewise constant control $$a^{K, \epsilon}_t = \sum_{k=0}^{K-1} a^\epsilon(t_k, X_{t_k}^{\mathcal{D}}, \Law(X_{t_k}^{\mathcal{D}}|\G_{t_k}), I_k) {\bf 1}_{t \in [t_k, t_{k+1}]} \in \U$$ such that $|J_{fp}(t, \mu; \pi_\epsilon) - J_{ol}(t, \vartheta; a^{K, \epsilon})| < \epsilon$. Here $X=(X_t^{\mathcal{D}})_{t\in[0,T]}$ is the state process in \eqref{extended-MFC-dynamics} under $a^{K, \epsilon}$ and for any $k=0,1,\ldots,K$, $I_k$ is independent of $\vartheta$, $W$, $N$, $(I_l)_{l=0}^{k-1}$.
	\end{prop}
	The proof of {\rm \propref{prop:diff-ol-fp}} is reported in Appendix \ref{sec:appendix-B}.
	
	\medskip 
	
	Combining Lemma \ref{lem:inequality-from-op-fp}, Proposition \ref{thm:verification} and Proposition \ref{prop:diff-ol-fp}, we obtain the main result of this section.

	\begin{theorem}\label{equivalencevalue}
		Let {\rm \assref{ass}},  {\rm \assref{Smoothness}}, {\rm \assref{measurable}},  and {\rm \assref{ass:regularity-Jh}} hold. Assume further that $\pi_\epsilon$ in {\rm \assref{measurable}} is a lifted randomized policy for any $\epsilon >0$. 
			Then,  we have $J_{ol}^*(t,\vartheta)=J_{fp}^*(t,\mu) = \hat J(t, \mu)$. Moreover, the feedback kernel-valued control $\pi_{t, \epsilon} \in \widehat\U$ defined by $\pi_{s, \epsilon} = \pi_\epsilon(s, \mu_s)$, $s \in [t, T]$ is also an $\epsilon$-optimal control for ${J}_{fp}$. And there exists a time grid $\mathcal{D}$ of $[t, T]$ with size $K$ such that the piecewise constant control $\alpha^{\epsilon, K}_s= \sum_{k=0}^Ka^{\epsilon}(t_k, X_{t_k}^{\mathcal{D}}, \Law(X_{t_k}^{\mathcal{D}}|\G_{t_k}),I_k) {\bf 1}_{s \in [t_k, t_{k+1}]} \in \U$ is an $2\epsilon$-optimal control for $J_{ol}$ (here $X^{\mathcal{D}}=(X_t^{\mathcal{D}})_{t\in[0,T]} $ is the dynamics controlled by $\alpha^{\epsilon, K}$). 
		\end{theorem}

		\begin{remark}\label{vlsequivalence}
			As a byproduct of \thmref{equivalencevalue}, the conditional law invariance holds for the original MFC problem, and hence we can denote by $J^*_{ol}(t,\mu)$ the original value function instead of $J^*_{ol}(t,\vartheta)$ and it holds that $\pa_t J^*(t,\mu)+\mathbb{T} J^*(t,\mu)=0$.
			
		\end{remark}

		\subsection{Relationship between HJB equation and SMP}
		
		In this subsection, we investigate the connection of the HJB equation to the BSDE induced by the SMP. A simple calculation, together with \ref{measure_shift}, { yields the next result}, whose proof is omitted.
		\begin{lemma}\label{newLdif}
			For any $\rho\in\Pc_2(\R^n\times U)$, $z\in Z$ and $L\in C^{1;1}(\Pc_2(\R^n))$, let $(X,\alpha)$ be an $\R^n\times U$-valued random variable defined on some probability space such that $\mu=\Law(X)$, $\hat\alpha(x)(\d u)=\Law(\alpha|X=x)(\d u)$ and $\rho=\Law(X,\alpha)$. Define $L_1:\Pc_2(\R^n\times U)\to\R$ as $L_1(\rho):=L(I^{\mu,\hat\alpha,z, *}\mu)$. Then, $\pa_{\mu}L_1$ (recall \defref{Ldif}) exists, and for all $(x,u)\in\R^n\times U$,
			\begin{equation*}
				\pa_{\mu}L_1(\rho)(x,u)=\pa_{\mu}L(I^{\mu,\hat{\alpha},z, *}\mu)(x+\gamma(x,\rho,u,z))\left\{1+\pa_x\gamma(x,\rho,u,z)+\E\left[\pa_{\mu}\gamma(X,\rho,\alpha,z)(x,u)\right]\right\}.
			\end{equation*}
		\end{lemma}
		
		The next theorem is the main result of this subsection.
		\begin{theorem}\label{smp_HJB}
			Let both \assref{Smoothness} and \assref{measurable} hold. Suppose that the original extended MFC problem \eqref{eq:svalue1}-\eqref{eq:seq1} has an optimal control $\alpha^*=(\alpha_t^*)_{t\in[0,T]}\in\U$, and let $X^*=(X_t^*)_{t\in[0,T]}$ be the resulting state process. Define $\mu_t^*=\Law(X_t^*|\G_t)$ and $\rho_t^*(\cdot)=\Law((X_t^*,\alpha_t^*)|\G_t)$ for $t\in[0,T]$. Then, by \thmref{equivalencevalue},  let $J^*\in C^{1;2,1}([0,T]\times\Pc_2(\R^n))$ solve the following PDE, $\d t\times\d\Pb\text{-}\as$,
			\begin{equation}\label{eq:HJB_J}
				\begin{aligned}
					0&=\int_{\R^n\times U}\Big(f(x,\rho_t^*,u)+\Big(b(x,\rho_t^*,u)-\langle\gamma(x,\rho_t^*,u,\cdot),\lambda\rangle\pa_{\mu}J(t,\mu_t^*)(x)\\
					&\quad+\frac{1}{2}\tr\left(\sigma\sigma^{\T}(x,\rho_t^*,u)\pa_x\pa_{\mu}J(t,\mu_t^*)(x)\right)\Big)\bigg)\rho_t^*(\d x,\d u) \\
					&\quad+\int_Z\left(J(t,I^{\mu_t^*,\hat{\alpha}_t^{*},z,*}\mu_t^*)-J(t,\mu_t^*)\right)\lambda(\d z)+\pa_tJ^*(t,\mu_t^*),
				\end{aligned}
			\end{equation}
			where $\hat\alpha_t^*(x)(\d u):=\Law(\alpha_t^*|\G_t,X_t^*=x)(\d u)$. Moreover, consider the processes defined by
			\begin{equation*}
				\begin{cases}
					\displaystyle p_t=\pa_{\mu}J^*(t,\mu^*_t)(X^*_t),\quad P_t=\pa_x\pa_{\mu}J^*(t,\mu_t^*)(X_t^*)\sigma(X_t^*,\rho_t^*,\alpha^*_t),\\
					\displaystyle K_t=\pa_{\mu}J^*(t,I^{\mu_t^*,\hat{\alpha}_t^{*},z,*}\mu_t^*)(X_t^*+\gamma(X_t^*,\mu_t^*,\alpha^*_t,z))-\pa_{\mu}J^*(t,\mu_t^*)(X_t^*).
				\end{cases}
			\end{equation*}
			The triplet $(p,P,K)=(p_t,P_t,K_t)_{t\in[0,T]}$ is the unique solution to the BSDE:
			\begin{equation*}
				\begin{cases}
					\displaystyle \d p_t=-\left\{\pa_xH(X_t^*,\alpha_t^*,\rho_t^*,p_t,P_t,K_t)+\E'\left[\pa_{x'}\delta H(X_t^{*'},\alpha_t^{*'},\rho_t^*,X_t^*,\alpha_t^*,p_t',P_t',K_t')\right] \right\}\d t\\
					\displaystyle\qquad\qquad+P_t\d W_t+\int_ZK_{t-}\tilde{N}(\d t,\d z),\\
					\displaystyle p_T=\pa_xg(X_T^*,\mu_T^*)+{\E'\left[\partial_{\mu} g({X_T^*}', \mu_T^*)(X_T^*)\right]}.
				\end{cases}
			\end{equation*}
		\end{theorem}
		
		\noindent{\it Proof.}\quad
		The first assertion follows directly from \lemref{lemma1} and \remref{joint_law}. We next verify the second assertion. Note that $J^*(T,\mu)=\hat g(\mu)=\int_{\R^n}g(x,\mu)\mu(\d x)$, and hence $p_T=\pa_{\mu}J^*(T,\mu_T^*)(X_T^*)=\pa_xg(X_T^*,\mu_T^*)+\E'\left[\pa_{\mu}g(X_T^{*'},\mu_T^*)(X_T^*)\right]$. As a result, the terminal condition is fulfilled.
		Differentiating on both sides of \equref{eq:HJB_J} with respect to $\mu$ and evaluating at $(t,X_t^*,\mu_t^*)$, one can conclude by \lemref{newLdif} that
		\begin{equation*}
			\begin{aligned}
				0&=\pa_t\pa_{\mu}J^*(t,\mu_t^*,X_t^*)+\pa_xb(X_t^*,\rho_t^*,\alpha^*_t)p_t+b(X_t^*,\rho_t^*,\alpha^*_t)\pa_x\pa_{\mu}J^*(t,\mu_t^*)(X_t^*)\\
				&\quad +\E'\left[\pa_{\mu}b(X_t^{*'},\rho_t^*,\alpha^{*'}_t)(X_t^*,\alpha_t^*)p_t'+b(X_t^{*'},\rho_t^*,\alpha^{*'}_t)\pa_{\mu}\pa_{\mu}J^*(t,\mu_t^*)(X_t^*,X_t^{*'}) \right]\\
				&\quad +\tr\left(P_t\pa_x\sigma(X_t^*,\rho_t^*,\alpha^*_t)^{\T}+\frac{1}{2}\sigma\sigma^{\T}(X_t^*,\rho_t^*,\alpha^*_t)\pa_{xx}\pa_{\mu}J^*(t,\mu_t^*)(X_t^*)\right)\\
				&\quad+\E'\left[\tr\left(P_t'\pa_{\mu}\sigma(X_t^{*'},\rho_t^*,\alpha^{*'}_t)(X_t^*,\alpha_t^*) +\frac12\sigma\sigma^{\T}(X_t^{*'},\rho_t^*,\alpha^{*'}_t)\pa_{\mu}\pa_x\pa_{\mu}J^*(t,\mu_t^*)(X_t^*,X_t^{*'})\right)\right]\\
				&\quad+\pa_x f(X_t^*,\rho_t^*,\alpha^*_t)+\E'\left[\pa_{\mu}f(X_t^{*'},\rho_t^*,\alpha^{*'}_t)(X_t^*,\alpha_t^*)\right]-\pa_{\mu}J^*(t,\mu_t^*)(X_t^*)\\
				&\quad+\int_Z\pa_{\mu}J^*(t,I^{\mu_t^*,\hat{\alpha}_t^*,z,*}\mu_t^*)(X_t^*+\gamma(X_t^*,\mu_t^*,\alpha^*_t,z))(1+\pa_x\gamma(X_t^*,\rho_t^*,\alpha_t^*,z))\lambda(\d z)\\
				&\quad+\int_Z\E'\left[\pa_{\mu}J^*(t,I^{\mu_t^*,\hat{\alpha}_t^*,z,*}\mu_t^*,X_t^{*'}+\gamma(X_t^{*'},\mu_t^*,\alpha^{*'}_t,z))\pa_{\mu}\gamma(X_t^{*'},\rho_t^*,\alpha_t^{*'},z)(X_t^*,\alpha_t^*)\right]\lambda(\d z)\\
				&\quad-\int_Z\left(\pa_x\gamma(X_t^*,\rho_t^*,\alpha_t^*,z)p_t+\gamma(X_t^*,\rho_t^*,\alpha_t^*,z)\pa_x\pa_{\mu}J^*(t,\mu_t^*)(X_t^*)\right)\lambda(\d z)\\
				&\quad-\E'\left[\int_Z\left(\pa_{\mu}\gamma(X_t^{*'},\rho_t^*,\alpha_t^{*'},z)(X_t^*,\alpha_t^*)p_t'+\gamma(X_t^{*'},\rho_t^*,\alpha_t^{*'},z)\pa_{\mu}\pa_{\mu}J^*(t,\mu_t^*)(X_t^*,X_t^{*'})\right)\lambda(\d z) \right].
			\end{aligned}
		\end{equation*}
		On the other hand, applying It\^{o}'s formula (c.f. Theorem 2.7 in Guo and Zhang~\cite{Guoxin}) to $p_t=\pa_{\mu}J^*(t,\mu_t^*,X_t^*)$, we arrive at
		\begin{align*}
			\d p_t&=\pa_t\pa_{\mu}J^*(t,\mu_t^*,X_t^*)\d t+b(X_t^*,\rho_t^*,\alpha^*_t)\pa_x\pa_{\mu}J^*(t,\mu_t)(X_{t-})\d t\\
			&\quad-\int_Z\gamma(X_{t-}^*,\rho_{t-}^*,\alpha_{t-}^*,z)\pa_x\pa_{\mu}J^*(t,\mu_t)(X_{t-})\lambda(\d z)\d t\\
			&\quad+\frac12\tr\left(\sigma\sigma^{\T}(X_t^*,\rho_t^*,\alpha^*_t)\pa_{xx}\pa_{\mu}J^*(t,\mu_t^*)(X_t^*)\right)\d t+\pa_x\pa_{\mu}J^*(t,\mu_t^*)(X_t^*)\sigma(X_t^*,\rho_t^*,\alpha^*_t)\d W_t\\
			&\quad+\int_Z\left(\pa_{\mu}J^*(t,\mu_t^*)(X_t^*)-\pa_{\mu}J^*(t,\mu_t^*)(X_{t-}^*)\right)N(\d t,\d z)\\
			&\quad+\E'\left[b(X_t^{*'},\rho_t^*,\alpha^{*'}_t)\pa_{\mu}\pa_{\mu}J^*(t,\mu_{t-}^*)(X_{t-}^*,X_{t-}^{*'})\right]\d t\\
			&\quad+\frac{1}{2}\E'\left[\sigma\sigma^{\T}(X_t^{*'},\rho_t^*,\alpha^{*'}_t)\pa_{\mu}\pa_x\pa_{\mu}J^*(t,\mu_{t-}^*)(X_{t-}^*,X_{t-}^{*'})\right]\d t\\
			&\quad+\int_Z\left(\pa_{\mu}J^*(t,\mu_t^*)(X_{t-}^*)-\pa_{\mu}J^*(t,\mu_{t-}^*)(X_{t-}^*)\right)N(\d t,\d z)\\
			&\quad-\E'\left[\int_Z\gamma(X_t^{*'},\rho_t^*,\alpha_t^*,z)\pa_{\mu}\pa_{\mu}J^*(t,\mu_t^*)(X_t^*,X_t^{*'})\lambda(\d z)\right] \d t.
		\end{align*}
		Combining the above two terms, one can conclude that
		\begin{align*}
			& \d p_t =-\left\{\pa_xH(X_t^*,\alpha_t^*,\rho_t^*,p_t,P_t,K_t)+\E'\left[\pa_{x'}\delta H(X_t^{*'},\alpha_t^{*'},\rho_t^*,X_t^{*},\alpha_t^*,p_t',P_t',K_t')\right] \right\}\d t+P_t\d W_t\nonumber\\
			&+\int_ZK_{t-}\tilde{N}(\d t,\d z)-\Bigg\{\int_Z\pa_{\mu}J^*(t,I^{\mu_t^*,\hat{\alpha}_t^*,z,*}\mu_t^*)(X_t^*+\gamma(X_t^*,\rho_t^*,\alpha_t^*,z))(1+\pa_x\gamma(X_t^*,\rho_t^*,\alpha_{t-}^*,z))\lambda(\d z)\nonumber\\
			&\quad+\int_Z\E'\left[\pa_{\mu}J^*(t,I^{\mu_t^*,\hat{\alpha}_t^*,z,*}\mu_t^*)(X_t^{*'}+\gamma(X_t^{*'},\mu_t^*,\alpha^{*'}_t,z))\pa_{\mu}\gamma(X_t^{*'},\rho_t^*,\alpha_t^{*'},z)(X_t^*,\alpha_t^*)\right]\Bigg\}\lambda(\d z)\d t\\
			&-\pa_{\mu}J^*(t,\mu_t^*)(X_t^*)-\int_Z\E'\left[\pa_{\mu}J^*(t,\mu_t^*)(X_t^{*'})\pa_{\mu}\gamma(X_t^{*'},\rho_t^*,\alpha_t^{*'},z)(X_t^*,\alpha_t^*)\right]\lambda(\d z)\\
			&+\int_Z\left(\pa_{\mu}J^*(t,\mu_t^*)(X_t^*)-\pa_{\mu}J^*(t,\mu_{t-}^*)(X_{t-}^*)\right)\lambda(\d z)\d t\!\!+\!\!\int_Z\pa_{\mu}J^*(t,\mu_t^*)(X_t^*)\pa_x\gamma(X_{t-}^*,\rho_{t-}^*,\alpha_{t-}^*,z)\lambda(\d z)\d t\nonumber\\
			&+\int_Z\E'\left[\pa_{\mu}J^*(t,\mu_t^*)(X_t^{*'})-\pa_{\mu}J^*(t,\mu_{t-}^*)(X_{t-}^{*'})\right]\lambda(\d z)\\
			&=-\left\{\pa_xH(X_t^*,\alpha_t^*,\rho_t^*,p_t,P_t,K_t)+\E'\left[\pa_{x'}\delta H(X_t^{*'},\alpha_t^{*'},\rho_t^*,X_t^{*},\alpha_t^*,p_t',P_t',K_t')\right] \right\}\d t+P_t\d W_t\nonumber\\
			&\quad+\int_ZK_{t-}\tilde{N}(\d t,\d z).
		\end{align*}
		The last equality holds by using the fact that the jump time is at most countable, and hence a.s. $\{t\in[0,T];~X_t^*\neq X_{t-}^*\}$ forms a Lebesgue zero measure set while the remainders vanish on the set $\{t\in[0,T];~X_t^*=X_{t-}^*\}$. Thus, we complete the proof of the theorem. \hfill$\Box$
		
		\begin{remark}
			Let $Y_t^*=X_t^*+\gamma(X_t^*,\rho_t^*,\alpha_t^*,z)$ for $t\in[0,T]$. Then $K_t=\pa_{\mu}J^*(t,\Law(Y_t^*|\G_t))(Y_t^*)-$ $\pa_{\mu}J^*(t,\Law(X_t^*|\G_t))(X_t^*)$ includes both jumps in the conditional law and the state process. The jump of the measure term comes from the existence of jump common noise. While in the jump diffusion case, the law of $X_t$ is continuous in $t$, and hence the jumps only come from the state (c.f. \remref{jumps}).
		\end{remark}

		\section{An Example of Linear-Quadratic Extended MFC}\label{sec:example}
		
		In this section, we study an example of LQ-type extended MFC problem to further illustrate the relationship between SMP and HJB equation obtained in last two sections (\thmref{smp_HJB}) and the difference between the adjoint BSDE \equref{SBSDE} and \equref{SBSDE_jump}. In particular, we will apply both methods to solve the LQ MFC problem. For comparison purpose, we also present the results under a jump diffusion dynamics with the same coefficients to illustrate the distinctions in the adjoint processes. To ease the presentation, we only consider one dimensional state process with the control space $U=\R$.
		
		Fix $t\in[0,T]$, for any $s\in[t,T]$,
		the controlled state process is governed by
		\begin{equation}\label{eq:LQ_state}
			\d X_s=\left\{b_1\E[X_s|\G_s]+b_2\E[\alpha_s|\G_s]+b_3\alpha_s\right\}\d s+\sigma X_s\d W_s+\int_Z\gamma(z)\alpha_{s-}\tN(\d s,\d z),~
			X_t=\vartheta,
		\end{equation}
		where $b_1,b_2,b_3,\sigma\in\R$, $\gamma(\cdot)\in L^2((Z,\mathscr{Z},\lambda);\R)$ and other notations remain the same as before. We aim to minimize the cost functional over $\alpha=(\alpha_s)_{s\in[t,T]}\in\U$:
		\begin{equation}\label{eq:LQ_cost}
			J(t,\vartheta;\alpha)=\frac{1}{2}\E\left[\int_{ t}^T\alpha_s^2\d s+c\left|X_T-\E[X_T|\G_T]\right|^2 \right].
		\end{equation}
		Here, we recall that $\Gb=(\G_t)_{t\in[0,T]}$ is the natural extensions of $\Fb^N=(\F_t^N)_{t\in[0,T]}$ to $\Omega$ given in Section~\ref{sec:extendedMFCPoiCommon}, and $c\geq 0$ is a constant parameter. The value function is defined by
		\begin{equation}\label{eq:LQ_value}
			J^*(t,\vartheta)=\inf_{\alpha\in\U}J(t,\vartheta;\alpha).
		\end{equation}
		
		We first apply the SMP method to solve the LQ MFC problem \eqref{eq:LQ_state}-\eqref{eq:LQ_value}.  Recall  \equref{SHamiltonian} and \equref{SdHamiltonian}, the Hamiltonian and delta Hamiltonian can be written as, for $(x,u,\rho,p,P,K)\in\R\times U\times\Pc_2(\R\times U)\times\R\times\R\times L^2((Z,\mathscr{Z},\lambda);\R)$,
		\begin{equation*}
			\begin{aligned}
				&H(x,u,\rho,p,P,K)=\left(\int_{\R^n\times U}(b_1x+b_2{ u'})\rho(\d x,\d { u'})+b_3 u\right)p+\sigma xP+\int_Z\gamma(z)uK(z)\lambda(\d z)+\frac12 u^2,\\
				&\qquad \delta H(x,u.\rho,x',u',p,P,K)=(b_1x'+b_2u')p.
			\end{aligned}
		\end{equation*}
		The extension of $H$ clearly satisfies the $L$-convexity stated in \defref{L-convex}, and hence we can apply \thmref{SSUF} to conclude that $\alpha\in\U$ is an optimal control if
		\begin{equation}\label{LQ_optimal}
			\alpha_s=-\left(\int_Z\gamma(z)K_s(z)\lambda(\d z)+b_2\E[p_s|\G_s]+b_3p_s\right),\quad \forall s\in[t,T],
		\end{equation}
		where  the adjoint process $(p,P,K(z))=(p_s,P_s,K_s(z))_{s\in[t,T]}$ is the unique solution to the following BSDE that, for $s\in[t,T]$,
		\begin{equation}\label{LQ_BSDE}
			\displaystyle \d p_s=-\{\sigma P_s+b_1\E[p_s|\G_s]\}\d s+P_s\d W_s+\int_ZK_s(z)\tN(\d s,\d z),~~
			p_T=c\{X_T-\E[X_T|\G_T]\}.
		\end{equation}
		Consider $p=(p_s)_{s\in[t,T]}$ with the form given by
		\begin{equation}\label{LQ_p}
			p_s=\beta_sX_s+\eta_s\E[X_s|\G_s],\quad \forall s\in [t,T],
		\end{equation}
		with the deterministic functions $\beta,\eta\in C([0,T];\R)$ that $\beta_T=c$ and $\eta_T=-c$. Differentiating on both sides of \equref{LQ_p} leads to that, for $s\in[t,T]$,
		\begin{equation*}
			\begin{aligned}
				\d p_s
				&=\dot{\beta}_sX_s\d s+\beta_s\left\{(b_1\E[X_s|\G_s]+b_2\E[\alpha_s|\G_s]+b_3\alpha_s)\d s+\sigma X_s\d W_s+\int_Z\gamma(z)\alpha_s\tN(\d s,\d z)\right\}\\
				&\quad+\dot{\eta}_s\E[X_s|\G_s]\d s+\eta_s\left\{[b_1\E[X_s|\G_s]+(b_2+b_3)\E[\alpha_s|\G_s]]\d s+\int_Z\gamma(z)\E[\alpha_s|\G_s]\tN(\d s,\d z)\right\},
			\end{aligned}
		\end{equation*}
		where $\dot{\beta}_s:=\frac{\d\beta_s}{\d s}$ and $\dot{\eta}_s:=\frac{\d\eta_s}{\d s}$. Comparing the equation \equref{LQ_BSDE}, we have
		\begin{equation}\label{LQ_adjoint}
			\begin{cases}
				\displaystyle -\{\sigma P_s+b_1\E[p_s|\G_s]\}=\dot{\beta}_sX_s+\dot{\eta}_s\E[X_s|\G_s]+(\beta_s+\eta_s)\{b_1\E[X_s|\G_s]+b_2\E[\alpha_s|\G_s]\}\\
				\displaystyle\qquad\qquad\qquad\qquad\qquad +b_3\{\beta_s\alpha_s+\eta_s\E[\alpha_s|\G_s]\},\\
				\displaystyle P_s=\beta_s\sigma X_s,\quad K_s(z)=\gamma(z)\{\beta_s\alpha_s+\eta_s\E[\alpha_s|\G_s]\}.
			\end{cases}
		\end{equation}
		Combining with \equref{LQ_optimal} and taking the conditional expectation, we deduce that
		\begin{align}
			\displaystyle   & \E[\alpha_s|\G_s]=-\frac{(b_2+b_3)(\beta_s+\eta_s)}{1+\int_Z\gamma^2(z)\lambda(\d z)(\beta_s+\eta_s)}\E[X_s|\G_s],\label{LQ_alphacon}\\
			\displaystyle    &  \alpha_s =-\frac{(b_2+b_3)(\beta_s+\eta_s)}{1+\int_Z\gamma^2(z)\lambda(\d z)(\beta_s+\eta_s)}\E[X_s|\G_s]-\frac{b_3\beta_s}{1+\int_Z\gamma^2(z)\lambda(\d z)\beta_s}\{X_s-\E[X_s|\G_s]\}.\label{LQ_alpha}
		\end{align}
		Substituting $\alpha_s$ and $\E[\alpha_s|\G_s]$ into the first equation of \equref{LQ_adjoint}, one can derive the  Riccati equations that, for $s\in[t,T]$,
		\begin{equation}\label{riccati}
			\begin{cases}
				\displaystyle \dot{\beta}_s+\sigma^2\beta_s-\frac{b_3^2}{1+\int_Z\gamma^2(z)\lambda(\d z)\beta_s}\beta_s^2=0;\\[1.4em]
				\displaystyle \dot{\eta}_s+\frac{b_3^2}{1+\int_Z\gamma^2(z)\lambda(\d z)\beta_s}\beta_s^2+\left(2b_1-\frac{(b_2+b_3)^2(\beta_s+\eta_s)}{1+\int_Z\gamma^2(z)\lambda(\d z)(\beta_s+\eta_s)}\right)(\beta_s+\eta_s)=0
			\end{cases}
		\end{equation}
		with terminal conditions $\beta_T=c$ and $\eta_T=-c$.
		
		In fact, Eq.~\equref{riccati} is a decoupled Riccati equation. From Theorem 7.2 in Chapter 6 of Yong and Zhou~\cite{Zhou}, it follows that \equref{riccati} admits a unique (smooth) solution $(\beta,\eta)=(\beta_s,\eta_s)_{s\in[t,T]}$ satisfying $1+\int_Z\gamma^2(z)\lambda(\d z)\beta_s>0,~1+\int_Z\gamma^2(z)\lambda(\d z)(\beta_s+\eta_s)>0$ and $\beta_T=c,~\eta_T=-c$. Moreover, using the uniqueness of the solution to BSDE \equref{LQ_BSDE}, we deduce that Eq.~\equref{riccati} has a unique solution  and the optimal control is given by \equref{LQ_alpha}.
		
		Before deriving the HJB equation, we need a preparation result whose proof is straightforward and omitted.
		\begin{lemma}\label{minimization}
			Let $a,b,c,d\in\R$ satisfy $a>0$ and $a+c>0$. For a given squared-integrable random variable $X$ on some probability space $(\Omega,\F,\mathbb{P})$, introduce the  functional $F:L^2((\Omega,\F,P);\R)\mapsto\R$ by
			\begin{equation*}
				F(\vartheta):=a\E[\vartheta^2]+b\E[\vartheta X]+c\left|\E[\vartheta]\right|^2+d\E[\vartheta],\quad \forall \vartheta\in L^2((\Omega,\F,P);\R).
			\end{equation*}
			Then, the functional $F$ admits a unique minimizer given by
			\begin{equation*}
				\vartheta^*=-\frac{b\E[X]+d}{2(a+c)}-\frac{b}{2a}\{X-\E[X]\}.
			\end{equation*}
			Moreover, the minimum of the functional $F$ is given by
			$F(\vartheta^*)=-\frac{b^2}{4a}{\rm Var}(X)-\frac{|b\E[X]+d|^2}{4(a+c)}$.
		\end{lemma}
		It follows from \equref{extended-HJB-1-poisson-common-noise} and \ref{lemma1}, the HJB equation can be written as:
		\begin{equation}\label{LQ_HJB}
			\pa_tJ({t},\mu)+\mathbb{T}J({t},\mu)=0,~J(T,\mu)=\frac{c}{2}\left[\int_{\R} x^2\mu(\d x)-\left(\int_{\R}x\mu(\d x)\right)^2\right],
		\end{equation}
		where the operator $\mathbb{T}$ is defined by
		\begin{align}\label{TJ}
			\mathbb{T}J(t,\mu)&=\inf_{\alpha\in L(\Omega;U)}\Bigg\{\E\Bigg[
			\frac12\alpha^2+\left(b_1\E[X]+b_2\E[\alpha]+b_3\alpha-\int_Z\gamma(z)\lambda(\d z)\alpha\right)\pa_{\mu}J(t,\mu)(X)\\
			&+\frac{1}{2}\sigma^2X^2\pa_x\pa_\mu J(t,\mu)(X)
			\bigg]+\int_Z \left\{J(t,\Law(X+\gamma(z)\alpha))-J(t,\mu)\right\}\lambda(\d z)\Bigg\},\quad X\sim\mu.\nonumber
		\end{align}
		We assume  that Eq.~\eqref{LQ_HJB} has a classical solution $\hat J$ in the form of
		\begin{equation}\label{LQ_hatJ}
			\hat J({t},\mu)=\frac{1}{2}\left\{\beta_t\E[ X^2]+\eta_t\left|\E[X]\right|^2\right\},~~X\sim\mu,~~\forall t\in[0,T].
		\end{equation}
		Here, $\beta,\eta\in C([0,T];\R)$ satisfy $1+\int_Z\gamma^2(z)\lambda(\d z)\beta_t>0$ and $1+\int_Z\gamma^2(z)\lambda(\d z)(\beta_t+\eta_t)>0$ for all $t\in[0,T]$. Hence, we also have $\beta_T=c$ and $\eta_T=-c$. Thus, it holds that
		\begin{equation*}
			\hat J(t,\Law(X+\gamma(z)\alpha))-\hat J(t,\mu)=\frac12\left\{2\gamma(z)(\beta_t\E[\alpha X]+\eta_t\E[\alpha]\E[ X])+\gamma^2(z)(\beta_t\E[\alpha^2]+\eta_t(\E[\alpha])^2)\right\}.
		\end{equation*}
		Moreover, it holds that
		\begin{equation*}
			\pa_{\mu}\hat J(t,\mu)(X)=\beta_t X+\eta_t \E[X],\quad \pa_x\pa_{\mu}\hat J(t,\mu)(X)=\beta_t.
		\end{equation*}
		Inserting the above equations into \equref{TJ}, and by applying \lemref{minimization}, we conclude that the optimizer of operator $\mathbb{T}$ in \eqref{TJ} satisfies
		\begin{equation}\label{LQ_optimalcontrol}
			\alpha_{t}=-\frac{(b_2+b_3)(\beta_t+\eta_t)}{1+\int_Z\gamma^2(z)\lambda(\d z)(\beta_t+\eta_t)}\E[X]-\frac{b_3\beta_t}{1+\int_Z\gamma^2(z)\lambda(\d z)\beta_t}\{X-\E[X]\},
		\end{equation}
		which resembles the optimal control of the form in \equref{LQ_alpha}. The HJB equation then becomes
		\begin{equation*}
			\begin{aligned}
				0&=\left[\dot{\eta}_t+\frac{b_3^2\beta_t^2}{1+\int_Z\gamma^2(z)\lambda(\d z)\beta_t}+\left(2b_1-\frac{(b_2+b_3)^2(\beta_t+\eta_t)}{1+\int_Z\gamma^2(z)\lambda(\d z)(\beta_t+\eta_t)}\right)(\beta_t+\eta_t)\right]\left|\E[X]\right|^2\\
				&\quad+\left(\dot{\beta}_t+\sigma^2\beta_t-\frac{b_3^2\beta_t^2}{1+\int_Z\gamma^2(z)\lambda(\d z)\beta_t}\right)\E[X^2],
			\end{aligned}
		\end{equation*}
		which leads to the Riccati equations, for $t\in[0,T]$,
		\begin{equation}\label{riccati_HJB}
			\begin{cases}
				\displaystyle \dot{\beta}_t+\sigma^2\beta_t-\frac{b_3^2\beta_t^2}{1+\int_Z\gamma^2(z)\lambda(\d z)\beta_t}=0,\\[1.2em]
				\displaystyle \dot{\eta}_t+\frac{b_3^2\beta_t^2}{1+\int_Z\gamma^2(z)\lambda(\d z)\beta_t}+\left(2b_1-\frac{(b_2+b_3)^2(\beta_t+\eta_t)}{1+\int_Z\gamma^2(z)\lambda(\d z)(\beta_t+\eta_t)}\right)(\beta_t+\eta_t)=0
			\end{cases}
		\end{equation}
		with $1+\int_Z\gamma^2(z)\lambda(\d z)\beta_t>0$, $1+\int_Z\gamma^2(z)\lambda(\d z)(\beta_t+\eta_t)>0$ and $\beta_T=c,~\eta_T=-c$. Using Theorem 7.2 in Chapter 6 of Yong and Zhou \cite{Zhou} again, we can claim its well-posedness. One can find that the above Riccati equation resembles \equref{riccati}. Now, $\hat J$ defined in \equref{LQ_hatJ} with $(\beta,\eta)=(\beta_t,\eta_t)_{t\in[0,T]}$ the unique solution to Eq.~\equref{riccati_HJB} is indeed a classical solution to \equref{LQ_HJB}, and hence \assref{Smoothness} is fulfilled. Moreover, \equref{LQ_optimalcontrol} ensures the existence of optimal lifted (randomized) feedback policy, and thus \assref{measurable} holds. Furthermore, we can insert \equref{LQ_optimalcontrol} into the HJB equation and apply \thmref{equivalencevalue} to conclude the value function $J^*(t,\mu)=\hat J(t,\mu)$. Consequently, \thmref{smp_HJB} is applicable in this LQ context.

		To compare with the case with Poissonian common noise, we also discuss the MFC problem with Poisson idiosyncratic noise that, for $s\in[t,T]$,
		\begin{equation}\label{jump_diffsuion}
			\d X_s=\{b_1\E[X_s]+b_2\E[\alpha_s]+b_3\alpha_s\}\d s+\sigma X_s\d W_s+\int_Z\gamma(z)\alpha_{s-}\tN(\d s,\d z),~X_t=\vartheta\in\R.
		\end{equation}
		In contrast to \eqref{eq:LQ_cost}, we consider the objective functional without common noise that
		\begin{equation}\label{eq:LQ_cost_jump}
			J(t,\vartheta;\alpha)=\frac12\E\left[\int_{ t}^T\alpha_s^2\d s+c\left|X_T-\E X_T\right|^2 \right].
		\end{equation}
		The value function is thus defined by
		\begin{equation}\label{eq:LQ_value_jump}
			J^*(t,\vartheta)=\inf_{\alpha\in\U}J(t,\vartheta;\alpha).
		\end{equation}
		
		Similar to the case with Poissonian common noise, by \corref{SNECj} and \corref{SSUFj}, the Hamiltonian and delta Hamiltonian here can be written as, for $(x,u,\rho,p,P,K)\in\R\times U\times\Pc_2(\R\times U)\times\R\times\R\times L^2((Z,\mathscr{Z},\lambda);\R)$,
		\begin{equation*}
			\begin{aligned}
				& H(x,u,\rho,p,P,K)=\left(\int_{\R^n\times U}(b_1x+b_2u)\rho(\d x,\d u)+b_3\right)p+\sigma xP+\int_Z\gamma(z)uK(z)\lambda(\d z)+\frac12 u^2,\\
				&\qquad \delta H(x,u,\rho,x',u',p,P,K)=(b_1x'+b_2u')p.
			\end{aligned}
		\end{equation*}
		The extension of $H$ clearly satisfies $L$-convexity stated in \defref{L-convex}, and hence we are ready to apply \thmref{SSUF} to conclude that $\alpha\in\U$ is an optimal control if
		\begin{equation}\label{LQ_optimalj}
			\alpha_s=-\left(\int_Z\gamma(z)K_s(z)\lambda(\d z)+b_2\E[p_s]+b_3p_s \right),\quad \forall s\in[t,T].
		\end{equation}
		Here, the adjoint process $(p,P,K(z))=(p_s,P_s,K_s(z))_{s\in[t,T]}$ is the unique solution to the following BSDE, for $s\in[t,T]$,
		\begin{equation}\label{LQ_BSDEj}
			\d p_s=-\{\sigma P_s+b_1\E[p_s]\}\d s+P_s\d W_s+\int_ZK_s(z)\tN(\d s,\d z),~~~
			p_T=c\{X_T-\E[X_T]\}.
		\end{equation}
		Let us consider $p_s$ in the form given by
		\begin{equation}\label{LQ_pj}
			p_s=\beta_sX_s+\eta_s\E[X_s],\quad \forall s\in [t,T]
		\end{equation}
		with deterministic coefficients $\beta,\eta\in C([0,T];\R)$ such that $\beta_T=c$ and $\eta_T=-c$. Repeating the same procedure as in the case with common noise, 
		we can get
		\begin{equation}\label{LQ_adjointj}
			\begin{cases}
				\displaystyle -\{\sigma P_s+b_1\E[p_s]\}=\dot{\beta}_sX_s+\dot{\eta}_s\E[X_s]+(\beta_s+\eta_s)\{b_1\E[X_s] +b_2\E[\alpha_s]\}\\
				\qquad\qquad\qquad\quad\qquad +b_3\{\beta_s\alpha_s+\eta_s\E[\alpha_s]\},\\
				\displaystyle P_s=\beta_s\sigma X_s,\quad K_s(z)=\gamma(z)\beta_s\alpha_s.
			\end{cases}
		\end{equation}
		as well as
		\begin{equation}\label{LQ_alphaconj}
			\begin{cases}
				\displaystyle \E[\alpha_s]=-\frac{(b_2+b_3)(\beta_s+\eta_s)}{1+\int_Z\gamma^2(z)\lambda(\d z)\beta_s}\E[X_s],\\[1.2em]
				\displaystyle\alpha_s=-\frac{(b_2+b_3)(\beta_s+\eta_s)}{1+\int_Z\gamma^2(z)\lambda(\d z)\beta_s}\E[X_s]-\frac{b_3\beta_s}{1+\int_Z\gamma^2(z)\lambda(\d z)\beta_s}\{X_s-\E[X_s]\},
			\end{cases}
		\end{equation}
		where coefficients $\beta_s$ and $\eta_s$ satisfy the Riccati equations, for $s\in[t,T]$,
		\begin{equation}\label{riccatij}
			\begin{cases}
				\displaystyle \dot{\beta}_s+\sigma^2\beta_s-\frac{b_3^2}{1+\int_Z\gamma^2(z)\lambda(\d z)\beta_s}\beta_s^2=0;\\[1.2em]
				\displaystyle \dot{\eta}_s+\frac{b_3^2}{1+\int_Z\gamma^2(z)\lambda(\d z)\beta_s}\beta_s^2+\left[2b_1-\frac{(b_2+b_3)^2(\beta_s+\eta_s)}{1+\int_Z\gamma^2(z)\lambda(\d z)\beta_s}\right](\beta_s+\eta_s)=0
			\end{cases}
		\end{equation}
		with $\beta_T=c$ and $\eta_T=-c$. One can easily observe that \equref{riccatij} is again well-posed thanks to Theorem 7.2 in Chapter 6 of Yong and Zhou~\cite{Zhou}, and the optimal control is given by \equref{LQ_alphaconj}.
		
		\begin{remark}
			We note that the last line of \eqref{LQ_adjointj} differs substantially from that of \eqref{LQ_adjoint} in the case with Poissonian common noise mainly because $s\mapsto\E[X_s|\G_s]$ is discontinuous but $s\mapsto \E[X_s]$ is continuous here. This  leads to a different adjoint process $K_s(z)$ in this case with Poissonian idiosyncratic noise whose jumps only stems from $X_s$.
		\end{remark}
		
		\ \\
		\textbf{Acknowledgements.} The authors are grateful to the Associate Editor and the referee for their helpful comments. L. Bo and J. Wang are supported by National Natural Science of Foundation of China (grant 12471451), Natural Science Basic Research Program of Shaanxi (grant 2023-JC-JQ-05) and the Shaanxi Fundamental Science Research Project for Mathematics and Physics (grant 23JSZ010). X. Wei is supported by National Natural Science Foundation of China grant under no.12201343. X. Yu is supported by the Hong Kong RGC General Research Fund (GRF) under grant no. 15306523 and grant no. 15211524.
		
		\ \\
		\appendix
		\section{Proofs of Auxiliary Results for Section~\ref{sec:model}}\label{appendix_for_sec2}
		The section collects proofs of some lemmas from Section~\ref{sec:model}.
		
		\begin{proof}[Proof of \lemref{metriceq}]
			On one hand, the restriction of $f\in\Lip_1(V)$ to $\R^n\times\Pc(U)$ belongs to $\Lip_1(\R^n\times\Pc(U))$. On the other hand, any $f\in\Lip_1(\R^n\times\Pc(U))$ can be extended to be an element in $\Lip_1(V)$ according to McShane~\cite{McShane}. 
		\end{proof}

		\begin{proof}[Proof of  \lemref{extensioneq}]
			In view of the definition in \eqref{eq:extensionh}, it suffices to show that $\ps:{\cal P}_2(\R^n\times {\cal P}(U))\mapsto{\cal P}_2(\R^n\times U)$ given by \eqref{eq:affinemaping} is Lipschitz continuous. We have, for any $\xi_1,\xi_2\in\mathcal{P}_2(V)$,
			\begin{equation*}
				\begin{aligned}
					\lv\ps(\xi_1)-\ps(\xi_2)\rv_{K,{\rm FM}}&=\sup_{\substack{f\in\Lip_1(K)\\ \lv f\rv_{\infty}\leq 1}}\left(\int_K\int_{\mathcal{M}(U)}f(x,u)q(\d u)\left(\xi_1(\d x,\d q)-\xi_2(\d x,\d q)\right)\right)\\
					&=\sup_{\substack{f\in\Lip_1(K)\\ \lv f\rv_{\infty}\leq 1}}\left(\int_V\int_U f(x,u)q(\d u)(\xi_1(\d x,\d q)-\xi_2(\d x,\d q))\right)\\
					&\leq\sup_{F\in\Lip_1(\R^n\times\Pc(U))}\left(\int_{\R^n}\int_{\Pc(U)}F(x,q)(\xi_1(\d x,\d q)-\xi_2(\d x,\d q))\right)\\
					&=d_{\rm KR}(\xi_1,\xi_2).
				\end{aligned}
			\end{equation*}
			The last inequality above holds because the mapping $F_f:\R^n\times\Pc(U)\mapsto\R$ defined by
			\begin{equation}\label{eq:Ff}
				F_f(x,q):=\int_Uf(x,u)q(\d u),\quad\forall (x,q)\in\R^n\times\Pc(U)
			\end{equation}
			belongs to the space $\Lip_1(\R^n\times\Pc(U))$ for each fixed $f\in{\rm Lip}_1(K)$ satisfying $\|f\|_{\infty}\leq1$. In fact, it follows from \eqref{eq:Ff} that, for any $(x_i,q_i)\in\R^n\times\Pc(U)$ with $i=1,2$,
			\begin{equation*}
				\begin{aligned}
					\left|F_f(x_1,q_1)-F_f(x_2,q_2)\right|
					&\leq\int_U\left|f(x_1,u)-f(x_2,u)\right|q_1(\d u)+\int_Uf(x_2,u)(q_1(\d u)-q_2(\d u))\\
					&\leq\int_U|x_1-x_2|q_1(\d u)+\sup_{\substack{g\in\Lip_1(U)\\ \lv g\rv_{\infty}\leq 1}}\left(\int_U g(u)(q_1(\d u)-q_2(\d u))\right)\\
					&= |x_1-x_2|+\lv q_1-q_2\rv_{U,{\rm FM}}\nonumber\\
					&=\lv(x_1,q_1)-(x_2,q_2)\rv_V.
				\end{aligned}
			\end{equation*}
			Thus, we complete the proof of the lemma.
		\end{proof}
		
		\begin{proof}[Proof of \lemref{lem:condexpect}]
			The claim can be easily verified by noting the equivalence \begin{align*}
				\int_{\M(U)}q(\d u)\Law((X,\delta_{\alpha})|\G)(\d x,\d q)=\int_U\delta_{u'}(\d u)\Law((X,\alpha)|\G)(\d x,\d u')=\Law((X,\alpha)|\G)(\d x,\d u).
			\end{align*}
		\end{proof}
		
		\begin{proof}[Proof of \lemref{extensionLdif}]\quad
			First of all, we have from \eqref{eq:affinemaping} that, for any $\xi_1,\xi_2\in\Pc_2(\R^n\times\Pc(U))$,
			\begin{equation*}
				\begin{aligned}
					\tilde{h}(\xi_2)-\tilde{h}(\xi_1)&=h(\ps(\xi_2))-h(\ps(\xi_1))\\
					&=\int_0^1\int_{\R^n}\int_{B}\frac{\delta h}{\delta\rho}(\ps(\xi_1)+\lambda(\ps(\xi_2)-\ps(\xi_1)))(x,u)(\ps(\xi_2)-\ps(\xi_1))(\d x,\d u)\d\lambda\\
					&=\int_0^1\int_{\R^n}\int_{B}\int_{\M(U)}\frac{\delta h}{\delta\rho}(\ps(\xi_1+\lambda(\xi_2-\xi_1)))(x,u)q(\d u)(\xi_2-\xi_1)(\d x,\d q)\d\lambda\\
					&=\int_0^1\int_{\R^n}\int_{\M(U)}\left(\int_U\frac{\delta h}{\delta\rho}(\ps(\xi_1+\lambda(\xi_2-\xi_1)))(x,u)q(\d u)\right)(\xi_2-\xi_1)(\d x,\d q)\d\lambda,
				\end{aligned}
			\end{equation*}
			where we have applied Fubini theorem to interchange the order of integral in the last line and recalled that $B$ stands for the Banach space introduced in Subsection \ref{subsecB}. In lieu of \defref{def:derivativelinear}, the representation \eqref{eq:chain} holds.  It remains to show that the growth condition in \defref{def:derivativelinear} is satisfied by the { linear functional derivative} $\frac{\delta \tilde{h}}{\delta \xi}$.
			
			According to \defref{def:derivativelinear},  there exists a constant $C>0$ such that $\left|\frac{\delta h}{\delta\rho}(\ps(\xi))(x,u)\right|\leq C(1+\|(x,u)\|_{K}^2)$ for all $(x,u)\in\bigcup\limits_{\eta\in I}\supp(\eta)$ with $I:={\cal P}_2(\R^n\times U)$. Then, by \eqref{eq:chain} and the compactness of the control space $U$, for all $(x,q)\in\bigcup_{\xi\in\Pc_2(\R^n\times\Pc(U))}\supp(\xi)=\R^n\times\Pc(U)$, it holds that
			\begin{equation*}
				\begin{aligned}
					\left|\frac{\delta \tilde{h}}{\delta \xi}(\xi)(x,q)\right|&=\left|\int_U \frac{\delta h}{\delta\rho}(\ps(\xi))(x,u)q(\d u)\right|\leq\left|\int_U C(1+|x|^2+\|u\|^2)q(\d u)\right|\nonumber\\
					&\leq C_U(1+|x|^2+\lv q\rv_{U,{\rm FM}})=C_U(1+\lv (x,q)\rv_V^2),
				\end{aligned}
			\end{equation*}
			where $C_U$ is a positive constant that depends on $U$. Thus, the proof is completed.
		\end{proof}
		
		\section{Proofs of Auxiliary Results for Section~\ref{sec:SMP}}\label{appendix}
		In this appendix, we collect proofs of some auxiliary results in Section~\ref{sec:SMP} of the paper.

		\begin{proof}[Proof of \lemref{dif}]
			The well-posedness of $V=(V_t)_{t\in [0,T]}$ follows from a similar argument of Lemma 3.9 in Bo et al. \cite{BWY1}.
			Let us set $Y_t:=\frac{X_t^{\epsilon}-X_t}{\epsilon}-V_t$ for $t\in[0,T]$. Then, it holds that
			{\small\begin{align*}
					\d Y_t&=\int_0^t\int_U\left(\frac{1}{\epsilon}(\tb(X_s^{\epsilon},\xi_s^{\epsilon},u)-\tb(X_s,\xi_s^{\epsilon},u)-\pa_x{\tb(X_s,\xi_s,u)V_s}\right)q_s(\d u)\d s\\
					&\quad+\int_0^t\int_U\left(\frac{1}{\epsilon}(\tb(X_s,\xi_s^{\epsilon},u)-\tb(X_s,\xi_s,u))-\E'[{\eta_{b,s}V_s'+\zeta_{b,s}}]\right)q_s(\d u)\d s\\
					&\quad+\int_0^t\int_U\left(\frac{1}{\epsilon}(\ts(X_s^{\epsilon},\xi_s^{\epsilon},u)-\ts(X_s,\xi_s^{\epsilon},u)-\pa_x{\ts(X_s,\xi_s,u)V_s}\right)q_s(\d u)\d s\\
					&\quad+\int_0^t\int_U\left(\frac{1}{\epsilon}(\ts(X_s,\xi_s^{\epsilon},u)-\ts(X_s,\xi_s,u))-\E'[{\eta_{\sigma,s}V_s'+\zeta_{\sigma,s}}]\right)q_s(\d u)\d W_s\\
					&\quad+\int_0^t\int_Z\int_U\left(\frac{1}{\epsilon}(\tg(X_{s-},\xi_{s-}^{\epsilon},u,z)-\tg(X_{s-},\xi_{s-}^{\epsilon},u,z))-\pa_x\tg(X_{s-},\xi_{s-},u,z)V_t\right)q_{s-}(\d u)\tN(\d s,\d z)\\
					&\quad+\int_0^t\int_Z\int_U\left(\frac{1}{\epsilon}(\tg(X_{s-},\xi_{s-}^{\epsilon},u,z)-\tg(X_{s-},\xi_{s-},u,z))-\E'[\eta_{\gamma,s-}V_{s-}'+\zeta_{\gamma,s-}]\right)q_{s-}(\d u)\tN(\d s,\d z)\\
					&=\int_0^t\int_U\left\{\pa_x\tb(X_s,\xi_s,u)Y_s+\E'  \left[\pa_x\left(\frac{\delta\tb}{\delta\xi}(X_s,\xi_s,u)\right)(X_s',q_s')Y_s'\right]\right\}q_s(\d u)\d s+\int_0^t\{\kappa_{b,s}+\E'[\delta_{b,s}]\}\d s\\
					&\quad+\int_0^t\int_U\left\{\pa_x\ts(X_s,\xi_s,u)Y_t+\E'\left[ \pa_x\left(\frac{\delta\ts}{\delta\xi}(X_s,\xi_s,u)\right)(X_s',q_s') Y_s'\right]\right\}q_s(\d u)\d W_s +\int_0^t\{\kappa_{\sigma,s}+\E'[\delta_{\sigma,s}]\}\d W_s\\
					&\quad+\int_0^t\int_U\int_Z\left\{\pa_x\tg(X_{s-},\xi_{s-},u,z)Y_s+\E' \left[\pa_x\left(\frac{\delta\tg}{\delta\xi}(X_{s-},\xi_{s-},u,z)\right)(X_{s-}',q_{s-}') Y_{s-}'\right]\right\}q_{s-}(\d u)N(\d s,\d z)\\
					&\quad+\int_0^t\{\kappa_{\gamma,s-}+\E'[\delta_{\gamma,s-}]\}\tN(\d s,\d z),
			\end{align*}}
			where $\kappa_{\phi,t},\delta_{\phi,t}$ are defined as follows\begin{align*}
				\kappa_{\phi,t}&=\frac{1}{\epsilon}\left(\tilde\phi(X_t^{\epsilon},\xi_t^{\epsilon},u)-\tilde\phi(X_t,\xi_t^{\epsilon},u)-\pa_{x}\tilde\phi(X_t,\xi_t,u)(X_t^{\epsilon}-X_t)\right),\\
				\delta_{\phi,t}&=\frac{1}{\epsilon}\left(\tilde\phi(X_t,\xi_t^{\epsilon},u)-\tilde\phi(X_t,\xi_t,u)-\E'\left[\eta_{b,t}(X_t^{\epsilon'}-X_t')\right]\right),~\forall \phi\in \{b,\sigma,\gamma(\cdot)\}.
			\end{align*}
			Then according to \lemref{diff}, \assref{ass} and \equref{eq:dif}, it holds that
			\begin{equation}\label{eq:remainder}
				\lim_{\epsilon\downarrow0}\sum_{\phi\in\{b,\sigma,\gamma(\cdot)\}}\sup_{t\in[0,T]}\E\left[|\kappa_{\phi,t}|^2+\left|\E'[\delta_{\phi,t}]\right|^2\right]=0.
			\end{equation}
			In lieu of the boundedness of derivatives in \assref{ass}, we can derive, a.s.
			\begin{equation*}
				|Y_t|^2\leq C\int_0^t|Y_s|^2\d s+C\int_0^t|Y_{s}|^2\d s+\rho_t,\quad \forall t\in[0,T],
			\end{equation*}
			{where $\rho_t$ satisfies $\rho_t\to 0$ as $\epsilon\downarrow 0$ in light of \equref{eq:remainder}}. Thus,
			similar to the proof of \lemref{statedif}, we can conclude with the help of Gronwall's inequality that
			\begin{equation*}
				\lim_{\epsilon\downarrow 0}\sup_{t\in[0,T]}\E\left[\left|\frac{X_t^{\epsilon}-X_t}{\epsilon}-V_t\right|^2\right]=\lim_{\epsilon\downarrow 0}\sup_{t\in[0,T]}\E\left[|Y_t|^2\right]=0.
			\end{equation*}
		\end{proof}
		\begin{proof}[Proof of \lemref{valuedif}]\quad
			We have from the definition of the relaxed cost functional $\J$ that
			\begin{equation*}
				\begin{aligned}
					0&\leq \J(q^{\epsilon})-\J(q)\\ &=\E\left[g(X_T^{\epsilon},\mu_T^{\epsilon})-g(X_T,\mu_T)+\int_0^T\left(\int_Uf(X_t^{\epsilon},\xi_t^{\epsilon},u)q_t^{\epsilon}(\d u)-\int_Uf(X_t,\xi_t,u)q_t(\d u)\right)\d t \right].
				\end{aligned}
			\end{equation*}
			Using \lemref{diff}, \lemref{dif} and \assref{ass}, one can derive
			\begin{equation*}
				\begin{aligned}
					0&\leq o(\epsilon)+\epsilon\E\left\{\pa_xg(X_T,\mu_T)\cdot V_T+\int_0^T\int_U\pa_xf(X_t,\xi_t,u)\cdot V_tq_t(\d u)\d t \right.\\
					&\quad +\int_0^T\left(\int_U f(X_t,\xi_t,u)v_t(\d u)-\int_0^T\int_Uf(X_t,\xi_t,u)q_t(\d u)\right)\d t\\
					&\quad+\E'\left[\pa_x\left(\frac{\delta g}{\delta\mu}(X_T,\mu_T)\right)(X_T')\cdot V_T'+\int_0^T\int_U\pa_x\left(\frac{\delta\tf}{\delta\xi}(X_s,\xi_s,u)\right)(X_s',q_s')\cdot V_t'q_t(\d u)\d t\right.\\
					&\quad+\left.\left.\int_0^T\int_U\pa_q\left(\frac{\delta\tf}{\delta\xi}(X_t,\xi_t,u)\right)(X_t',q_t')(v_t'-q_t')q_t(\d u)\d t\right]\right\}.
				\end{aligned}
			\end{equation*}
			The desired result holds by taking $\epsilon\downarrow0$ on both sides of the inequality above.
		\end{proof}
		
		{
			
			\begin{proof}[Proof of \lemref{law_representation}]
				By the classical SDE theory, the process $\hat X=(\hat X_t)_{t\in[0,T]}$ solves \equref{SDE_hat} and it only suffices to show that $\xi$ corresponds to the conditional joint law. Let $h_1:D([0,T]\times\R^n)\times\mathcal{Q}\to\R$ and $h_2:\mathscr{N}(\Pi_Z)\to\R$ (see Page 6 of Bo et al. \cite{BWWY} for the detailed definition of $\mathscr{N}(\Pi_Z)$) be (Borel) measurable mappings. By the identical law condition, we deduce that
				{\small \begin{align*}
						&\E^{\hat P}\left[\int_{D([0,T];\R^n)}\int_{\mathcal{Q}}h_1(x,q)\hat\xi(\d x,\d q)h_2(\hat N)\right]=\E\left[\int_{D([0,T];\R^n)}\int_{\mathcal{Q}}h_1(x,q)\xi(\d x,\d q)h_2(N)\right]\\
						&\qquad=\E\left[\E\left[h_1(X,q)\middle|N\right]h_2(N)\right]=\E^{\hat P}\left[\E^{\hat P}\left[h_1(\hat X,\hat q)\middle|\hat N\right]h_2(\hat N)\right],
				\end{align*}}where, in the last equality, we have used $\Law(X,q,N)=\Law^{\hat P}(\hat X,\hat q,\hat N)$ as well as $\Law((X,q)|N)=\Law^{\hat P}((\hat X,\hat q)|\hat N)$ for $\Law(N)(\Law^{\hat P}(\hat N))$-$\as$ $n\in\mathscr{N}(\Pi_z)$. Note that the conditional law can be expressed as a Borel mapping of $n\in\mathscr{N}(\Pi_Z)$ and the above equality should therefore be understood as an identity between these Borel mappings on the canonical space $\mathscr{N}(\Pi_Z)$ of $N(\hat N)$. Consequently, by the arbitrariness of $h_1$ and $h_2$, we get 
				\begin{align*}
					\hat\xi(\d x,\d q)=\Law^{\hat{P}}((\hat X,\hat q)|\hat N)(\d x,\d q),~ \hat P\text{-}\as.
				\end{align*}
				Furthermore, if $q$ is continuous, we can choose a countable subset of $[0,T]$ and use the c\`{a}dl\`{a}g property of $\hat X$ to derive
				$\hat P\left(\hat\xi_t=\Law^{\hat P}((X_t,q_t)|N),\forall t\in [0,T]\right)=1$. It remains to show that $\hat\xi_t$ is $\hat\G_t$-measurable. To this end, we define 
				\begin{align*}
					\mathcal{E}:=\left\{E=E_1\cap E_2: E_1\in \G_t,E_2=\{\hat N((t,s]\times A)\in \mathbb{T}\},0\leq t<s\leq T, A\in\mathscr{Z}, \mathbb{T}\in\mathcal{B}([0,T])\right\}.
				\end{align*}
				Note that $N((t,s]\times A)$ is $P$-independent of $\sigma(X_r,q_r,W_r,N_r;r\leq t)$ for $0\leq t<s\leq T$ and $A\in\mathscr{Z}$. Hence $\hat N((t,s]\times A)$ is also $\hat P$-independent of $\sigma(\hat X_r,\hat q_r,\hat W_r,\hat N_r;r\leq t)$. It then holds that
				{\small\begin{align*}
						\E^{\hat P}\left[\xi_t(D)\boldsymbol{1}_E\right]&=\E^{\hat P}\left[\E^{\hat P}\left[\boldsymbol{1}_D(\hat X_t,\hat q_t)\middle|\hat N\right]\boldsymbol{1}_E\right]=\E^{\hat P}\left[\boldsymbol{1}_D(\hat X_t,\hat q_t)\boldsymbol{1}_{E_1}\boldsymbol{1}_{E_2}\right]=\E^{\hat P}\left[\boldsymbol{1}_D(\hat X_t,\hat q_t)\boldsymbol{1}_{E_1}\right]\hat P(E_2)\\
						&=\E^{\hat P}\left[\E^{\hat P}\left[\boldsymbol{1}_D(\hat X_t,\hat q_t)\middle|\hat\G_t\right]\boldsymbol{1}_{E_1}\right]\hat P(E_2)=\E^{\hat P}\left[\E^{\hat P}\left[\boldsymbol{1}_D(\hat X_t,\hat q_t)\middle|\hat\G_t\right]\boldsymbol{1}_E\right],
				\end{align*}}for every $D\subset D([0,T];\R^n)\times\mathcal{Q}$ and $E\in\mathcal{E}$ being Borel measurable. Here, in the third and last equality, we have used the fact that $E_2$ is $\hat P$ independent of $\sigma(X_r,q_r,W_r,N_r;r\leq t)$. As $\mathcal{E}$ forms a $\pi$-system and generates $\sigma(\hat N)$, we conclude by $\pi$-$\lambda$ theorem and the arbitrariness of $D,E$  that
				\begin{align}\label{xi_t_representation}
					\hat\xi_t(\d t,\d u)=\Law^{\hat{P}}((\hat X_t,\hat q_t)|\hat\G_t),~\hat{P}\text{-}\as.
				\end{align}
				That is, $\hat\xi_t$ is $\hat\G_t$-measurable, and the proof is complete.
			\end{proof}
			
			\begin{proof}[Proof of \lemref{continuous_control}]\quad
				Let $q^n$ be constructed as in Equation (3.8) of Ma and Yong \cite{Majin}. By the same argument, we deduce that outside a $\tt m\otimes \Pb$-null set $\mathcal{N}\subset[0,T]\times \Omega$, $\ie$ $\forall (t,\omega)\notin\mathcal{N}$, \begin{align}\label{pointwise_convergence}
					\lim_{n\to\infty}\int_U f(u)q_t^n(\omega,\d u)=\int_U f(u)q_t(\omega,\d u)
				\end{align}
				holds for all $f\in C(U;\R^l)$ and $l\in\mathbb{N}$. Let $\xi_t(\omega)=\Law((X_t,q_t)|\G_t)(\omega)$, $\mu_t(\omega)=\Law(X_t|\G_t)$ and $\xi_t^n(\omega)=\Law((X_t^n,q_t^n)|\G_t)(\omega)$, $\mu_t^n(\omega)=\Law(X_t^n|\G_t)$. It follows from \eqref{pointwise_convergence} that, for all
				$(t,\omega)\notin\mathcal N$,
				\begin{align}
					\lim_{n\to\infty}
					\int_U \tilde b\bigl(X_t(\omega),\xi_t(\omega),u\bigr)\,
					q_t^n(\omega,\d u)
					=
					\int_U \tilde b\bigl(X_t(\omega),\xi_t(\omega),u\bigr)\,
					q_t(\omega,\d u).
				\end{align}
				Similar results hold for $\ts$, $\tg(\cdot,z)$, and $\tf$. Dominated convergence theorem implies that
				\begin{align*}
					\lim_{n\to\infty}\int_U\phi(X_t,\xi_t,u)q_t^n(\d u)=\int_U\phi(X_t,\xi_t,u)q_t(\d u)~~~\text{in}~L^2([0,T]\times\Omega) 
				\end{align*}
				holds for $\phi\in\{\tb,\ts,\tg(\cdot,z),\tf\}$. We then define 
				\begin{equation*}
					\left\{\begin{aligned}
						\Delta\tb(t)&:=\tb (X_t^n,\xi_t^n,u)-\tb (X_t,\xi_t,u),~~\Delta\ts(t):=\ts(X_t^n,\xi_t^n,u)-\ts (X_t,\xi_t,u),\\
						\Delta\tg(t,z)&:=\tg (X_t^n,\xi_t^n,u,z)-\tg (X_t,\xi_t,u,z),~~\Delta\tf(t):=\tf (X_t^n,\xi_t^n,u)-\tf (X_t,\xi_t,u).
					\end{aligned}\right.
				\end{equation*}
				Thanks to \assref{extra_ass}, it holds that
				\begin{align}
					|\Delta \tb(t)|&\leq\left|b_1(X_t^n,\mu_t^n,u)-b_1(X_t,\mu_t,u)\right|\nonumber\\
					&\quad+\left|\int_{\R^n\times\Pc(U)}\int_U(b_2(X_t^n,\mu_t^n,u)(x',u')-b_2(X_t,\mu_t,u)(x',u'))q'(\d u')\xi_t^n(\d x',\d q')\right|\nonumber\\
					&\quad+\left|\int_{\R^n\times\Pc(U)}\int_Ub_2(X_t,\mu_t,u)(x',u')q'(\d u')(\xi_t^n-\xi_t)(\d x',\d q')\right|\nonumber\\
					&:= I_1+I_2+I_3.\label{Delta_estimate}
				\end{align}
				By \assref{extra_ass} again, we have \begin{align}\label{I_12_estimate}
					I_1+I_2\leq L\left(|X_t^n-X_t|+\mathcal{W}_{2,\R^n}(\mu^n,\mu_t)\right),
				\end{align}
				and
				\begin{align}
					I_3&=\left|\E'\left[\int_Ub_2(X_t,\mu_t,u)(X_t',u')q_t'(\d u)-\int_Ub_2(X_t,\mu_t,u)({X_t^n}',u'){q_t^n}'(\d u)\right]\right|\nonumber\\
					&\leq L\E\left[\left|X_t^n-X_t\right|\middle|\G_t\right]+\E'\left[\int_Ub_2(X_t,\mu_t,u)(X_t',u')(q_t'-{q_t^n}')(\d u')\right]\nonumber\\
					&:=I_{31}+I_{32}\label{I_3_estimate}.
				\end{align}
				Here and in the sequel, the generic constant $L>0$ may vary from line to line. Recall that $q'$, ${q^n}'$ are independent copies of $q$, $q^n$ respectively and therefore we derive $I_{32}\to 0$ uniformly due to the uniform continuity of $b_2$ with respect to $u'$. Similar results also hold for $\Delta\ts(t)$ and $\Delta\tg(t,z)$.
				
				Applying It\^{o}'s formula to $|X_t^n-X_t|^2$, we get 
				{\small\begin{align*}
						|X_t^n-X_t|^2&=2\int_0^t\left(\int_U(X_s^n-X_s)^{\T}\Delta\tb(s) q_s^n(\d u)+\int_U(X_s^n-X_s)^{\T}\tb(X_s,\xi_s,u)(q_s^n-q_s)(\d u))\right)\d s\\
						&\quad+2\int_0^t\left(\int_U(X_s^n-X_s)^{\T}\Delta\ts(s) q_s^n(\d u)+\int_U(X_s^n-X_s)^{\T}\ts(X_s,\xi_s,u)(q_s^n-q_s)(\d u))\right)\d W_s\\
						&\quad+\int_0^t\left(\int_U\Delta\ts(s)q_s^n(\d u)+\int_U\ts(X_s,\xi_s,u)(q_s^n-q_s)(\d u)\right)^2\d s\\
						&\quad+\int_0^t\int_Z\left(\int_U\Delta\tg(s,z)q_s^n(\d u)+\int_U\tg(X_s,\xi_s,u,z)(q_s^n-q_s)(\d u)\right)^2N(\d z,\d s)\\
						&\quad-2\int_0^t\int_Z\left(\int_U(X_s^n-X_s)^{\T}\Delta\tg(s,z)q_s^n(\d u)\right.\\
						&\qquad\left.+\int_U(X_s^n-X_s)^{\T}\tg(X_s,\xi_s,u,z)(q_s^n-q_s)(\d u)\right)\lambda(\d z)\d s.
				\end{align*}}Applying Burkholder-Davis-Gundy inequality, Cauchy-Schwartz inequality and Gronwall inequality and combing \equref{Delta_estimate}, \equref{I_12_estimate} and \equref{I_3_estimate} altogether, we can  conclude that
				\begin{align*}
					\lim_{n\to\infty}\E\left[\sup_{t\in [0,T]}|X_t-X_t^n|^2\right]=0.
				\end{align*}
				For the value function equivalence, on one hand, note that by \equref{local_Lip} in \assref{extra_ass}, the estimation for $\Delta\tf(t)$ holds:
				{\small\begin{align*}
						|\Delta\tf(t)|&\leq|f_1(X_t^n,\mu_t^n,u)-f_1(X_t,\mu_t,u)|\\
						&\quad+\left|\int_{\R^n\times\Pc(U)}\int_U(f_2(X_t^n,\mu_t^n,u)(x',u')-f_2(X_t,\mu_t,u))(x',u')q'(\d u')\xi_t^n(\d x',\d q')\right|\nonumber\\
						&\quad+\left|\int_{\R^n\times\Pc(U)}\int_Uf _2(X_t,\mu_t,u)(x',u')q'(\d u')(\xi_t^n-\xi_t)(\d x',\d q')\right|\nonumber\\
						&\leq L(1+|X_t^n|+|X_t|+M_2(\mu_t^n)+M_2(\mu_t))\left(|X_t^n-X_t|+\mathcal{W}_{1,\R^n}(\mu_t^n,\mu_t)+\E\left[|X_t^n-X_t|\middle|\G_t\right]\right)\nonumber\\&
						\quad+\underbrace{\E'\left[\int_Ub_2(X_t,\mu_t,u)(X_t',u')(q_t'-{q_t^n}')(\d u')\right]}_{I_f(t)}.
				\end{align*}}Similar to the case of $\Delta\tb$, $\int_UI_f(t)q_t^n(\d u)\to 0$ in $L^1[0,T]\times\Omega)$ (by the growth condition in \assref{extra_ass}, \equref{pointwise_convergence} and dominated convergence theorem).
				On the other hand, since $X^n=(X_t^n)_{t\in [0,T]}$ converges to $X=(X_t)_{t\in [0,T]}$ in $L^2(\Omega;D([0,T];\R^n))$ and $D([0,T];\R^n)$ is separable, we can extract a subsequence $X^{n_k}=(X_t^{n_k})_{t\in [0,T]}$ such that $X^{n_k}\to X$ in $D([0,T];\R^n)$, $\Pb$-$\as$, which implies $X^{n_k}_T\to X_T$ and $\mathcal{W}_{1;\R}(\mu_T^{n_k},\mu_T)\to 0$ since every continuous strictly increasing bijection $\lambda$ on $[0,T]$ must satisfy $\lambda(T)=T$. Simple calculations with \equref{local_Lip} in \assref{extra_ass} yield that
				\begin{align*}    &\left|\int_0^T\int_U\tf(X_t^{n_k},\mu_t^{n_k},u)q_t^{n_k}(\d u)\d t-\int_0^T\int_U\tf(X_t,\mu_t,u)q_t(\d u)\d t+g(X_T^{n_k},\mu_T^{n_k})-g(X_T,\mu_T)\right|\\
					\leq&\int_0^T\int_U\left|\Delta \tf(t)\right|q_t^{n_k}(\d u)\d t+\left|\int_0^T\int_U\tf(X_t,\mu_t,u)(q_t^{n_k}-q_t)(\d u)\d t\right|+\left|g(X_T^{n_k},\mu_T^{n_k})-g(X_T,\mu_T)\right|\\
					\leq&L\int_0^T(1+|X_t^{n_k}|+|X_t|+M_2(\mu_t^{n_k})+M_2(\mu_t))\left(|X_t^n-X_t|+\mathcal{W}_{2,\R^n}(\mu_t^n,\mu_t)+\E\left[|X_t^n-X_t|\middle|\G_t\right]\right)\d t\\
					&+\left|\int_0^T\int_U I_f(t)q_t^n(\d t)\d t\right|+\left|\int_0^T\int_U\tf(X_t,\mu_t,u)(q_t^{n_k}-q_t)(\d u)\d t\right|\\
					&+L(1+|X_T^{n_k}|+|X_T|+M_2(\mu_T^{n_k})+M_2(\mu_T))(|X_T^{n_k}-X_T|+\mathcal{W}_{1,\R^n}(\mu_T^{n_k},\mu_T))\to 0,~\Pb\text{-}\as.
				\end{align*}
				Consequently, by dominated convergence theorem, we have\begin{align*}
					\lim_{n\to\infty}|\mathcal{J}(q^n)-\mathcal{J}(q)|=0,
				\end{align*}
				which verifies the equivalence between value functions in \equref{value_continuity}. 
			\end{proof}
			
			\begin{proof}[Proof of \lemref{L2con}]\quad
				{
					Let $\Omega_0\subset\Omega$ be the set of sample paths such that for every $\omega\in\Omega_0$, the jumping times of the Poisson random measure $N(\omega)$ are finite and are denoted by $0\leq \tau_1<\tau_2<\cdots<\tau_k\leq T$, where $\{\tau_i\}_{i=1}^k$ and $n$ may depend on $\omega$. Due to the finite intensity of $N$, $\Pb(\Omega_0)=1$. Note that $q=(q_t)$ has continuous paths and $X_t$ is continuous in $t\in [\tau_i,\tau_{i+1})$ for every $\omega\in\Omega_0$ and hence $\xi_t(\omega)$ is also continuous in  $t\in [\tau_i,\tau_{i+1})$. We first show that for any $\phi\in L^2[0,T]$, the following convergence holds for every $\omega\in\Omega_0$\begin{align}\label{L2weakcon}
						\int_0^T M_{\tilde b}^n(t,\omega)\phi(t)\d t\to \int_0^T M_{\tilde b}(t,\omega)\phi(t)\d t,~n\to\infty,
					\end{align}
					which is equivalent to assert that for every $\omega\in\Omega_0$, $M_{\tilde b}^n(\cdot,\omega)\overset{w}{\to} M_{\tilde b}(\cdot,\omega)$ in $L^2([0,T])$. By definition and (B.3) in Subsection~\ref{sec:extendedMFCPoiCommon}, we derive that\begin{align}\label{uniform_bound}
						\sup_n\left|M_{\tilde b}^n(t,\omega)\right|\vee \left|M_{\tilde b}(t,\omega)\right|\leq K(1+|X_t(\omega)|+M_2(\xi_t(\omega))).
					\end{align}
					By standard moment estimation, we can derive that\begin{align}
						\E\left[\sup_{t\in[0,T]}|X_t|+\sup_{t\in [0,T]}M_2(\xi_t)\right]<\infty.
					\end{align}
					Therefore, there exists a subset $\Omega_1\subset\Omega$ with $\Pb(\Omega_1)$, such that for all $\omega\in\Omega_1$\begin{align}\label{sup_finite}
						\sup_{t\in [0,T]}|X_t(\omega)|+\sup_{t\in [0,T]}M_2(\xi_t(\omega))<\infty.
					\end{align}
					Furthermore, for any $\epsilon>0$, by Lusin's theorem, there exists a continuous function $\psi\in C([0,T];\R)$ and a subset $E\subset [0,T]$ with $\int_{[0,T]\backslash E}|\phi(t)|\d t<\epsilon$ such that $\psi|_E=\phi|_E$ and $|\psi(t)|\leq |\phi(t)|,~t\in [0,T]$. Simple calculations yield that\begin{align}
						\left|\int_0^T \left(M_{\tilde b}^n(t,\omega)-M_{\tilde b}(t,\omega)\right)\phi(t)\d t\right|&\leq \int_0^T\left(M_{\tilde b}^n(t,\omega)-M_{\tilde b}(t,\omega)\right)\left(\phi(t)-\psi(t)\right)\d t\nonumber\\
						&\quad+\int_0^T\left(M_{\tilde b}^n(t,\omega)-M_{\tilde b}(t,\omega)\right)\psi(t)\d t\nonumber\\
						&=I_1+I_2.\label{L2_distance}
					\end{align}
					For $I_2$, note that\begin{align}
						I_2=\sum_{i=1}^{k-1}\int_{\tau_i}^{\tau_{i+1}}\left(M_{\tilde b}^n(t,\omega)-M_{\tilde b}(t,\omega)\right)\psi(t)\d t.
					\end{align}
					As $q^n\to q$ in $\mathcal{Q}$ $\as$ and $\psi(t),M_{\tilde b}^n(t,\omega),M_{\tilde b}(t,\omega)$ are continuous in $[\tau_i,\tau_{i+1}]$, there exists $N\in\mathbb{N}_+$ such that for $n>N$, $|I_2|<\epsilon$. \\
					For $I_1$, just note that\begin{align*}
						|I_1|&\leq \int_0^T\left|M_{\tilde b}^n(t,\omega)-M_{\tilde b}(t,\omega)\right|\left|\phi(t)-\psi(t)\right|\d t\\
						&\leq2\left(\sup_{t\in[0,T]}|X_t|+\sup_{t\in [0,T]}M_2(\xi_t)\right)\int_0^T|\phi(t)-\psi(t)|\d t\\
						&\leq 4\left(\sup_{t\in[0,T]}|X_t|+\sup_{t\in [0,T]}M_2(\xi_t)\right)\int_{[0,T]\backslash E}|\phi(t)|\d t\\
						&\leq 4\left(\sup_{t\in[0,T]}|X_t|+\sup_{t\in [0,T]}M_2(\xi_t)\right)\epsilon.
					\end{align*}
					Combining the above estimations together, we can deduce that\begin{align}
						|I_1|+|I_2|\leq 4\left(\sup_{t\in[0,T]}|X_t|+\sup_{t\in [0,T]}M_2(\xi_t)+\frac14\right)\epsilon,
					\end{align}
					for all $n>N$, which yields the desired result \eqref{L2weakcon} for every $\omega\in\Omega_0\cap\Omega_1$. Then for any $\phi\in L^2([0,T]\times\Omega)$,  there exists a subset $\Omega_2\subset\Omega$ with $\Pb(\Omega_2)=1$ such that $\phi(\cdot,\omega)\in L^2[0,T]$ for $\omega\in\Omega_2$. Then for $\omega\in\Omega_0\cap\Omega_1\cap\Omega_2$, it holds by \eqref{L2weakcon} that\begin{align*}
						\int_0^T M_{\tilde b}^n(t,\omega)\phi(t,\omega)\d t\to \int_0^T M_{\tilde b}(t,\omega)\phi(t,\omega)\d t,~n\to\infty.
					\end{align*}
					Consequently, by the dominated convergence theorem, we can conclude $M_{\tilde b}^n\overset{w}{\to}M_{\tilde b}$ in $L^2([0,T]\times\Omega)$ as $n\to\infty$. The same argument can be applied to $\sigma,\gamma$ and  the proof is complete.}
				
			\end{proof}
			\begin{proof}[Proof of \propref{valueeq}]
				In light of the fact that $\delta(\U)\subset\Qb$ and  \equref{value_continuity} in \lemref{continuous_control}, it only suffices to show that\begin{align}\label{continuous_control_value}
					\inf_{q\in \Qb_c}\mathcal{J}(q)\leq\inf_{\alpha\in\U}J(\alpha).
				\end{align}  
				To begin with, consider $q\in\Qb_c$ and $(\alpha^n)_{n\geq1}$ given in \lemref{L2con}. For any $n\geq1$, denote by $X^n=(X_t^n)_{t\in[0,T]}$,  $\xi^n$, $\xi_t^n$ and $\mu_t^n$ the state processes \eqref{eq:seq1} under $\alpha^n$, the joint conditional law $\Law((X^n,\delta_{\alpha_t^n}(\d u)\d t)|N)$,  $\Law((X_t^n,\delta_{\alpha_t^n})|\G_t)$ and the conditional state law $(\Law(X_t^n|\G_t))_{t\in [0,T]}$, respectively.  
				Similar to the proof of Lemma 5.1 in Bo et al. \cite{BWY} and Lemma 3.10 in Bo et al. \cite{BWWY}, we can deduce the tightness of $\{\Pb\circ (X^n)^{-1}:n\in\mathbb{N}\}$ in $\Pc(D([0,T];\R^n))$. Moreover, by Proposition A.2 in Carmona et al. \cite{Carmona2}, we can also derive the tightness of $\{\Pb\circ (\xi^n)^{-1}:n\in\mathbb{N}\}$ in $\Pc(\Pc_2(D([0,T];\R^n))\times\mathcal{Q}))$ Thanks to the compactness of $U$, we can further derive the tightness of $\mathcal{K}=\{\Pb\circ (X^n,\delta_{\alpha_t^n}(\d u)\d t,W,N,\xi^n)^{-1},n\in\mathbb{N}\}$ in $\Pc(D([0,T];\R^n)\times \mathcal{Q}\times C([0,T];\R^d)\times\mathscr{N}(\Pi_Z)\times\Pc_2(D([0,T];\R^n)\times\mathcal{Q}))$. Hence, we can extract a subsequence $\{\Pb\circ (X^{n_k},\delta_{\alpha_t^{n_k}}(\d u)\d t,W,N,\xi^{n_k})^{-1};~n_k\in\mathbb{N},\}\subset\mathcal{K}$ converging weakly to some probability measure $\hat{P}$ on $D([0,T];\R^n)\times \mathcal{Q}\times C([0,T];\R^d)\times\mathscr{N}(\Pi_Z)\times \Pc_2(C([0,T]\times\mathcal{Q})$. 
				
				By Skorokhod representation theorem, we can construct a probability space $(\hat\Omega,\hat{\F},\hat\Pb)$ on which a sequence of processes $\{(\hat{X}^{n_k},\hat{q}^{n_k},\hat{W}^{n_k},\hat{N}^{n_k},\hat\xi^{n_k}),n_k\in\mathbb{N}\}$ and $(\hat X,\hat q,\hat W,\hat N,\hat\xi)$ are defined, such that {\rm (i)} $\hat{\Pb}\circ (\hat{X}^{n_k},\hat{q}^{n_k},\hat{W}^{n_k},\hat{N}^{n_k},\hat\xi^{n_k})^{-1}=\Pb\circ (X^{n_k},\delta_{\alpha_t^{n_k}}(\d u)\d t,W,N,\xi^{n_k})^{-1}$, $\hat P=\hat{\Pb}\circ (\hat X,\hat q,\hat W,\hat N,\hat\xi)^{-1}$; {\rm (ii)} $(\hat{X}^{n_k},\hat{q}^{n_k},\hat{W}^{n_k},\hat{N}^{n_k},\hat\xi^{n_k})\to (\hat X,\hat q,\hat W,\hat N,\hat\xi)$ in $D([0,T];\R^n)\times \mathcal{Q}\times C([0,T];\R^d)\times\mathscr{N}(\Pi_Z)\times\Pc_2(D([0,T];\R^n)\times\mathcal{Q})$ $\hat P$-$\as$.  By \lemref{law_representation}, we know that, $\hat P$-$\as$, $\hat X^{n_k}=(\hat X_t^{n_k})_{t\in [0,T]}$ solves the SDE
				\begin{align*}
					\d \hat{X}_t^{n_k}&=\int_U\tilde b(\hat{X}_t^{n_k},\xi_t^{n_k},u)\hat{q}_t^{n_k}(\d u)\d t+ \int_U\ts(\hat{X}_t^{n_k},\xi_t^{n_k},u)\hat{q}_t^{n_k}(\d u)\d \hat{W}^{n_k}_t\\
					&\quad+\int_U\int_Z\tg(\hat{X}_t^{n_k},\xi_t^{n_k},u,z)\hat{q}_t^{n_k}(\d u)\tilde{\hat{N}}^{n_k}(\d t,\d z),
				\end{align*}
				and $\xi_t^{n_k}=\Law^{\hat{P}}((\hat{X}_t^{n_k},\hat{q}^{n_k}_t)|\hat{\G}_t^{n_k})$, $\mu_t^{n_k}=\Law^{\hat P}(\hat X_t^{n_k}|\hat\G_t^{n_k})$, where $\hat\G_t^{n_k}=\sigma(\hat N^{n_k}((0,s]\times A):s\leq t,A\in\mathscr{Z}\}$. Furthermore, by lemma 2.1 and the proof of Theorem 3.6 in Ma and Yong \cite{Majin}, each $\hat{q}^{n_k}$ has the form \begin{align*}
					\hat q_t^{n_k}(\d u)=\delta_{\hat\alpha_t^{n_k}}(\d u),~\forall t\in [0,T],    
				\end{align*}
				for some $U$-valued process $\hat\alpha^{n_k}=(\hat\alpha_t^{n_k})_{t\in [0,T]}$.  
				
				Define\begin{align*}
					\hat M_{\phi}^{n_k}(t,\hat\omega)=\tb(\hat X_t(\hat\omega),\hat\xi_t(\omega),\hat\alpha_t^{n_k}(\hat\omega)),~\hat M_{\phi}(t,\hat\omega)=\int_U\tb(\hat X_t(\hat\omega),\hat\xi_t(\omega),u)\hat q_t(\hat\omega,\d u).
				\end{align*}
				\lemref{L2con} yields that $M_{\phi}^{n_k}\overset{w}{\to}M_{\phi}$ as $k\to\infty$ for each $\phi\in\{\tb,\ts,\tg(\cdot,z)\}$. Therefore, by Mazur's Theorem (c.f. Brezis \cite{Brezis}), there exist $k$ positive real numbers $\{\lambda_1,\cdots,\lambda_k\}$ satisfying $\sum_{i=1}^{k}\lambda_i=1$ (here the sequence $(\lambda_i)_{i=1}^{n_k}$  may depend on $k$, but we omit it for simplicity) such that, for all $z\in Z$,
				\begin{align}\label{L2strongcon}
					\lim_{k\to\infty}\sum_{\phi\in\{\tb,\ts,\tg(\cdot,z)\}}\left\|\sum_{i=1}^{k}\lambda_{i}\hat M_{\phi}^{n_i}-\hat M_{\phi}\right\|_{L^2([0,T]\times\Omega)}^2=0.
				\end{align}

				Similar to Equations (3.24) and (3.25) in Ma and Yong \cite{Majin}, we can derive \begin{align}\label{delta_WN}
					\lim_{k\to\infty}\E^{\hat P}\left[\sup_{t\in [0,T]}\left|\sum_{i=1}^k\lambda_i\Delta^i(\hat W)(t)\right|^2+\sup_{t\in [0,T]}\left|\sum_{i=1}^k\lambda_i\Delta^i(\hat N)(t)\right|^2\right]=0,
				\end{align}
				where $\Delta^i(\hat W)$ and $\Delta^i(\hat N)$ are defined by
				{\small\begin{align*}
						\Delta^i(\hat W)(t)&:=\int_0^t\int_U\ts(\hat{X}^{n_i}_s,\hat{\xi}_s^{n_i},u)\hat q^{n_i}_s(\d u)\d \hat W_s^{n_i}-\int_0^t\int_U\ts(\hat{X}^{n_i}_s,\hat{\xi}_s^{n_i},u)\hat q^{n_i}_s(\d u)\d \hat W_s,\\
						\Delta^i(\hat N)(t)&:=\int_0^t\int_U\int_Z\tg(\hat X_s^{n_i},\hat\xi_s^{n_i},u)\hat q^{n_i}_s(\d u)\hat N^{n_i}(\d s,\d z)-\int_0^t\int_U\int_Z\tg(\hat X_s^{n_i},\hat\xi_s^{n_i},u)\hat q^{n_i}_s(\d u)\hat N(\d s,\d z).
				\end{align*}}Denote by $\hat\xi_t$ the $t$-marginal of $\hat\xi$ and $\hat\mu_t$ the first marginal of $\hat\xi_t$, $\ie$, $\hat\mu_t=\hat\xi_t|_{\R^n}$. 
				By the separation condition in \assref{extra_ass}, it holds that
				{\small\begin{align*} 
						&\sum_{i=1}^{k}\lambda_i\left(\tb(\hat X_t^{n_i},\hat\xi_t^{n_i},\hat\alpha_t^{n_i})-\hat M_{\tb}^{n_i}(t)\right)=\sum_{i=1}^{k}\lambda_i\left(b_1(\hat X_t^{n_i},\hat\mu_t^{n_i},\alpha_t^{n_i})-b_1(\hat X_t,\hat \mu_t,\hat\alpha_t^{n_i})\right)\\
						&\quad+\sum_{i=1}^{k}\lambda_i\int_{\R^n\times\Pc(U)}\int_Ub_2(\hat X_t,\hat \mu_t,\hat\alpha_t^{n_i})(x',u')q'(\d u')(\hat\xi_t^{n_i}-\hat\xi_t)(\d x',\d q')\\
						&\quad+\sum_{i=1}^{k}\lambda_i\int_{\R^n\times\Pc(U)}\int_U\left(b_2(\hat X_t^{n_i},\hat\mu_t^{n_i},\hat\alpha_t^{n_i})(x',u')-b_2(\hat X_t,\hat\mu_t,\hat\alpha_t^{n_i})(x',u')\right)q'(\d u')\hat\xi_t^{n_i}(\d x',\d q')\\
						&:=I_1+I_2+I_3.
				\end{align*}}
				On one hand, by the Lipschitz property of $b_1$ and $b_2$, we conclude that
				\begin{align}\label{I_13_estimate}
					|I_1|+|I_3|\leq L\sum_{i=1}^{k}\lambda_i\left(\left|\hat X_t^{n_i}-\hat X_t\right|+\mathcal{W}_1(\hat\mu_t^{n_i},\hat\mu_t)\right).
				\end{align}
				Here and in the sequel, $L>0$ is a generic constant that may vary from line to line. As $\hat\xi^{n_i}\to\hat\xi$ in $\Pc_2(D([0,T];\R^n)\times\mathcal{Q})$ $\hat P$-$\as$ and $D([0,T];\R^n)\times\mathcal{Q}\in(x,q)\mapsto \int_{\mathbb{T}}\int_Ub_2(\hat X_t,\hat\mu_t,\hat\alpha_t^{n_i}(x(t),u)q_t(\d u)\d t$ is continuous for any  Borel measurable $\mathbb{T}\subset [0,T]$ with positive Lebesgue measure (see the proof of Lemma 5.2 in Bo et al. \cite{BWY}), we deduce from Fubini' theorem and dominated convergence theorem that, $\hat P$-$\as$,
				{\small\begin{align}\label{tb_convergence}    &\left|\int_{\mathbb{T}}\sum_{i=1}^{k}\lambda_i\left(\tb(\hat X_t^{n_i},\hat\xi_t^{n_i},\hat\alpha_t^{n_i})-\hat M_{\tb}^{n_i}(t)\right)\d t\right|\nonumber\\
						&\quad\leq L\int_{\mathbb{T}}\sum_{i=1}^k\left(\left|\hat X_t^{n_i}-\hat X_t\right|+\mathcal{W}_1(\hat\mu_t^{n_i},\hat\mu_t)\right)\d t+\left|\int_{\mathbb{T}}I_2\d t\right|\to 0,\quad k\to\infty,
				\end{align}}Similarly, the convergence \equref{tb_convergence} also holds for $\ts$ and $\tg(\cdot,z)$. Consequently, similar to the proof of Theorem 3.4 in Ma and Yong \cite{Majin}, it follows from \equref{L2strongcon}, \equref{delta_WN} and \equref{tb_convergence}  that $\hat X$ satisfies:
				\begin{align*}
					\d \hat X_t&=\int_U\tb(\hat X_t,\hat\xi_t,u)\hat q_t(\d u)\d t+\int_U\ts(\hat X_t,\hat\xi_t,u)\hat q_t(\d u)\d\hat W_t+\int_U\tg(\hat X_t,\hat\xi_t,u,z)\hat q_t(\d u)\tilde{\hat{N}}(\d t,\d z).
				\end{align*}
				We next prove 
				\begin{align}\label{hatxi_t_representation}
					\hat P\left(\hat\xi_t=\Law^{\hat P}\left((\hat X_t,\hat q_t)\middle|\hat\G_t\right),\forall t\in [0,T]\right)=1,
				\end{align}
				where $\hat\G_t=\sigma(\hat N((0,s]\times A),0<s\leq T,~A\in\mathscr{Z})$ for $t\in[0,T]$. Let $h_1:D([0,T];\R^n)\times\mathcal{Q}\to\R$ and $h_2:\mathscr{N}(\Pi_Z)\to\R$ be bounded continuous mappings. It holds that
				\begin{align*}
					\E^{\hat P}\left[\int_{D([0,T];\R^n)}\int_{\mathcal{Q}}h_1(x,q)\hat\xi(\d x,\d q)h_2(\hat N)\right]&=\lim_{k\to\infty}\E^{\hat P}\left[\int_{D([0,T];\R^n)}\int_{\mathcal{Q}}h_1(x,q)\hat\xi^{n_k}(\d x,\d q)h_2(\hat N^{n_k})\right]\\
					&=\lim_{k\to\infty}\E^{\hat P}\left[\E^{\hat P}\left[h_1(\hat X^{n_k},\hat q^{n_k})\middle|\hat N^{n_k}\right]h_2(\hat N^{n_k})\right]\\
					&=\E^{\hat P}\left[\E^{\hat P}\left[h_1(\hat X,\hat q)\middle|\hat N\right]h_2(\hat N)\right],
				\end{align*}
				where in the first equality we have applied the dominated convergence theorem and in the last equality we have exploited the weak convergence of $\Law^{\hat P}(\hat X^{n_k},\hat q^{n_k},\hat N^{n_k})$ to $\Law^{\hat P}(\hat X,\hat q,\hat N)$. Thanks to the arbitrariness of $h_1,h_2$, we conclude $\hat P(\hat\xi=\Law^{\hat P}((\hat X,\hat q)|\hat N))=1$. Moreover, by an argument similar to that used in the proof of \lemref{law_representation}, we can show that $\hat N((t,s]\times A)$ is $\hat P$-independent of $\sigma(\hat X_r,\hat q_r,\hat W_r;r\leq t)$ for $0\leq t<s\leq T$ and $A\in\mathscr{Z}$. Consequently, by the continuity of $\hat q$ and the same proof of \equref{xi_t_representation}, we deduce \equref{hatxi_t_representation}.
				
				Finally, in light of the weak uniqueness (c.f. Definition 4.4 in EI Karoui et al. \cite{Karoui}), we get by dominated convergence theorem that
				\begin{align*}
					\mathcal{J}(q)&=\E^{\hat P}\left[\int_0^T\int_U\tf(\hat X_t,\hat\xi_t,u)\hat q_t(\d u)\d t+g(\hat X_T,\hat\mu_T)\right]\\
					&=\lim_{k\to\infty}\E^{\hat P}\left[\int_0^T\int_U\tf(\hat X_t^{n_k},\hat\xi_t^{n_k},u)\hat q_t^{n_k}(\d u)\d t+g(\hat X_T^{n_k},\hat\mu_T^{n_k})\right]=\lim_{k\to\infty}J(\alpha^{n_k}),
				\end{align*}
				which verifies \equref{continuous_control_value}. The proof is then complete.
			\end{proof}

		}

		\section{Proofs of Auxiliary Results for Section~\ref{sec:HJB}}\label{sec:appendix-B}
		In this appendix, we collect proofs of some auxiliary results in Section~\ref{sec:HJB}.
		
		\medskip 
		
		{\noindent{\it Proof of \lemref{measure_shift}}.  
			Note that $C_0(\R^n)\subset C_b(\R^n)$ and thus $I^{\mu,\hat u,z}\nu$ can be interpreted as a signed Radon measure on $\R^n$ by Riesz representation theorem.  Define $\langle g,\mu\rangle:=\int_{\R^n}g(x)\mu(\d x)$ for any $\mu\in\Pc_2(\R^n)$. 
			
			Let us verify that $I^{\mu, \hat u, z, *}$ maps each $\nu\in\Pc_2(\R^n)$ into $\Pc_2(\R^n)$ for any $\hat u\in\hat \U(\R^n)$. In fact, it holds that $\langle I^{\mu, \hat u, z}(g),\nu\rangle=\langle g,I^{\mu, \hat u, z, *}\nu\rangle$ for all $g\in C_b(\R^n)$. As a result, $I^{\mu, \hat u, z, *}\nu$ is a (positive) measure. By considering $g\equiv 1$, it is then a probability measure on $\R^n$. Thus, it remains to show that $I^{\mu, \hat u, z, *}\nu$ belongs to $\Pc_2(\R^n)$ whenever $\nu\in\Pc_2(\R^n)$. However, this can be easily verified by considering $g_n(x)=|x|^2\wedge n$ for $x\in\R^n$, the linear growth of $\gamma$ with respect to $x$ and sending $n\to\infty$. 
			
			For the second assertion, note that, for any $\phi\in C_b(\R^n)$, $\Pb$-$\as$.
			\begin{equation*}
				\begin{aligned}
					\langle \phi, I^{\mu,\hat\alpha, z, *}\mu\rangle &= \langle  I^{\mu,\hat\alpha, z}(\phi),\mu\rangle=\int_{\R^n}\int_U \phi(x+\gamma(x,\rho,u,z))\hat\alpha(x,\d u)\mu(\d x)\\
					&=\int_{\R^n}\E\left[\phi(x+\gamma(x,\rho,\alpha,z))|\G,X=x\right]\mu(\d x)=\E\left[\phi(X+\gamma(X,\rho,\alpha,z))|\G\right].
				\end{aligned}
			\end{equation*}
			The desired result then follows from the arbitrariness of $\phi\in C_b(\R^n)$. 
			
			\hfill$\Box$
		}

		\medskip 

		\medskip 
		
		We next prove \propref{prop:diff-ol-fp}. To do it, we need \lemref{lem:diff-ol-fp} below. Recall that $(\mu_s)_{s \geq 0}$ is the solution to~\eqref{equ:extended-FP-mu_t} with $\mu_0 = \mu$ under the lifted randomized policy $\pi_\epsilon$, and that $X^{\mathcal{D}}$ is the solution to~\eqref{extended-MFC-dynamics} under the piecewise constant randomized policy $a^{K, \epsilon}_t = \sum_{k=0}^{K- 1} a_{t_k} {\bf 1}_{t \in [t_k, t_{k+1}]}$ with $a_{t_k}: = a^\epsilon(t_k, X_{t_k}^{\mathcal{D}}, \Law(X_{t_k}^{\mathcal{D}}|\G_{t_k}), I_k) \sim \pi_\epsilon(t_k, X_{t_k}^{\mathcal{D}}, \Law(X_{t_k}^{\mathcal{D}}|\G_{t_k}))$, i.e., for $s \in [t_k, t_{k+1}]$,
		\begin{equation}\label{eq:SDE-constant}
			\begin{aligned}
				X_s^{\mathcal{D}} =& X_{t_k}^{\mathcal{D}} + \int_{t_k}^s  b(X_r^{\mathcal{D}}, \Law((X_r^{\mathcal{D}}, a_{t_k})|\G_r), a_{t_k})\d r + \int_{t_k}^s\sigma(X_r^{\mathcal{D}}, \Law((X_r^{\mathcal{D}}, a_{t_k})|\G_r), a_{t_k}) \d W_r\\ & + \int_{t_k}^s \int_Z \gamma(X_{r-}^{\mathcal{D}}, \Law((X_{r-}^{\mathcal{D}}, a_{t_k-}|\G_{r-}),a_{t_k-}, z) \widetilde N(\d r, \d z).
			\end{aligned}
		\end{equation}

		\begin{lemma}\label{lem:diff-ol-fp} Let {\rm \assref{ass}} and {\rm \assref{ass:regularity-Jh}} hold. Then, for any $h \in C^{1, 1}(\mathcal{P}(\R^n))$ in \assref{ass:regularity-Jh}, there exists a constant $C$ depending only on $T$, $b, \sigma, \gamma, \lambda$ and $h$ such that, for time grid $\mathcal{D}$, $\sup_{t \in [0, T]}\left|\E[h(\Law(X_t^{\mathcal{D}}|\G_t)) -h(\mu_t)]\right| \leq C |\mathcal{D}|^{1/2}$ with $|\mathcal{D}|: = \max_{0 \leq k < K-1}|t_{k+1} - t_k|$.
		\end{lemma}

		\noindent{\it Proof of \lemref{lem:diff-ol-fp}.}\quad From now on, we consider only the time grid points $t_k$ without loss of generality, and  use $C$ to denote a generic constant, which may change from line to line but only depends on $T$, $b, \sigma, \gamma, \lambda$ and $h$.  As $J_h$ satisfies the linear PDE~\eqref{PDE:Jh} on Wasserstein space, the Feynman-Kac's formula yields $\E[h(\mu_{t_k})] = J_{h}(0, \mu)$. It then follows that
		\begin{equation*}
			\begin{aligned}
				\E\left[h(\Law(X_{t_k}^{\mathcal{D}}|\G_{t_k})) - h(\mu_{t_k})\right] & = \E\left[J_h(t_k, \Law(X_{t_k}^{\mathcal{D}}|\G_{t_k}))) - J_h(0, \mu)\right]\\
				& = \sum_{i=0}^{k-1}\E\left[J_h(t_{i+1}, \Law(X_{t_{i+1}}^{\mathcal{D}}|\G_{t_{i+1}})) - J_h(t_i, \Law(X_{t_i}^{\mathcal{D}}|\G_{t_i}))\right] =: \sum_{i=0}^{k-1} e_i.
			\end{aligned}
		\end{equation*}
		It thus suffices to estimate the term $e_i$. By applying It\^o's lemma (c.f. Theorem 2.7 in Guo and Zhang~\cite{Guoxin}) to $J_h(r, \Law(X_{r}^{\mathcal{D}}|\G_{r}))$ between $t_i$ and $t_{i+1}$, it holds that
		\begin{equation*}
			\begin{aligned}
				e_i =& \E\bigg[\int_{t_i}^{t_{i+1}} \bigg( \pa_t J_h(r, \Law(X_{r}^{\mathcal{D}}|\G_{r})) + (b(X_r^{\mathcal{D}}, \Law((X_r^{\mathcal{D}}, a_{t_i})|\G_r), a_{t_i}) \\
				&\quad -\langle\gamma(X_r^{\mathcal{D}}, \Law((X_r^{\mathcal{D}}, a_{t_i})|\G_r), a_{t_k}), \lambda\rangle) \pa_\mu J_h(r,\Law(X_{r}^{\mathcal{D}}|\G_{r}))(X_{r}^{\mathcal{D}}) \\
				& \quad + \frac{1}{2} \tr\left(\sigma\sigma^{\T}(X_r^{\mathcal{D}}, \Law((X_r^{\mathcal{D}}, a_{t_i})|\G_r), a_{t_i}) \pa_x\pa_\mu J_h(r,\Law(X_{r}^{\mathcal{D}}|\G_{r}))(X_{r}^{\mathcal{D}})\right)\\
				& \quad + \int_Z \left(J(r,\Law(X_{r-}^{\mathcal{D}}+\gamma(X_{r-}^{\mathcal{D}},\Law((X_{r-}^{\mathcal{D}}, a_{t_i})|\G_{r-}),a_{t_i},z))-J(r,\Law(X_{r-}^{\mathcal{D}})|\G_{r-})\right)\lambda(\d z)\bigg)\d r\bigg].
			\end{aligned}
		\end{equation*}
		Note that $\E[\pa_t J_h(t_i,\Law(X_{t_i}^{\mathcal{D}}|\G_{r}))+{\rm B}^{\pi(t_i, \Law(X_{t_i}^{\mathcal{D}}|\G_{t_i}))}J_h(t_i,\Law(X_{t_i}^{\mathcal{D}}|\G_{t_i}))]=0$ for $0 \leq i \leq k-1$. Subtracting it on both sides of the above equation, we have
		\begin{align*}
			e_i =& \E\bigg[\int_{t_i}^{t_{i+1}} \bigg( \pa_t J_h(r, \Law(X_{r}^{\mathcal{D}}|\G_{r})) - \pa_t J_h(t_i, \Law(X_{t_i}^{\mathcal{D}}|\G_{t_i}))\big)\d r\bigg]\\
			&\quad + \E\bigg[\int_{t_i}^{t_{i+1}}\bigg(b(X_r^{\mathcal{D}}, \Law((X_r^{\mathcal{D}}, a_{t_i})|\G_r), a_{t_i})  \pa_\mu J_h(r,\Law(X_{r}^{\mathcal{D}}|\G_{r}))(X_{r}^{\mathcal{D}})\\
			& \quad \quad \quad- (b(X_{t_i}^{\mathcal{D}}, \Law((X_{t_i}^{\mathcal{D}}, a_{t_i})|\G_{t_i}), a_{t_i})  \pa_\mu J_h(t_i,\Law(X_{t_i}^{\mathcal{D}}|\G_{t_i}))(X_{t_i}^{\mathcal{D}})\bigg)\d r\bigg] \\
			&\quad +\E\bigg[\int_{t_i}^{t_{i+1}} \bigg(-\langle\gamma(X_r^{\mathcal{D}}, \Law((X_r^{\mathcal{D}}, a_{t_i})|\G_r), a_{t_i}), \lambda\rangle) \pa_\mu J_h(r,\Law(X_{r}^{\mathcal{D}}|\G_{r}))(X_{r}^{\mathcal{D}}) \\
			&\quad \quad \quad +\langle\gamma(X_{t_i}^{\mathcal{D}}, \Law((X_{t_i}^{\mathcal{D}}, a_{t_i})|\G_i), a_{t_i}), \lambda\rangle) \pa_\mu J_h(t_i,\Law(X_{t_i}^{\mathcal{D}}|\G_{t_i}))(X_{t_i}^{\mathcal{D}}) \bigg)\d r\bigg]\\
			& \quad + \frac{1}{2}\E\bigg[\int_{t_i}^{t_{i+1}} \bigg(\tr\left(\sigma\sigma^{\T}(X_r^{\G^k}, \Law((X_r^{\mathcal{D}}, a_{t_i})|\G_r), a_{t_i}) \pa_x\pa_\mu J_h(r,\Law(X_{r}^{\mathcal{D}}|\G_{r}))(X_{r}^{\mathcal{D}})\right)\\
			& \quad \quad \quad -  \tr\left(\sigma\sigma^{\T}(X_{t_i}^{\G^k}, \Law((X_{t_i}^{\mathcal{D}}, a_{t_i})|\G_r), a_{t_i}) \pa_x\pa_\mu J_h(t_i,\Law(X_{t_i}^{\mathcal{D}}|\G_{r}))(X_{t_i}^{\mathcal{D}})\right)\bigg)\d r\bigg]\\
			& \quad +\E\bigg[\int_{t_i}^{t_{i+1}}\int_Z \bigg(J(r,\Law(X_{r-}^{\mathcal{D}}+\gamma(X_{r-}^{\mathcal{D}},\Law(X_{r-}^{\mathcal{D}}, a_{t_i})|\G_{r-}),a_{t_i},z))-J(r,\Law(X_{r-}^{\mathcal{D}}|\G_{r-}))\\
			& \quad-  J(t_i,\Law(X_{t_i-}^{\mathcal{D}}+\gamma(X_{t_i-}^{\mathcal{D}},\Law((X_{t_i-}^{\mathcal{D}}, a_{t_i})|\G_{t_i-}),a_{t_i},z))+J(t_i,\Law(X_{t_i-}^{\mathcal{D}}|\G_{t_i-}))\bigg) \lambda(\d z)\d r\bigg].
		\end{align*}
		By standard estimates of~\eqref{eq:SDE-constant}, we obtain, for $r \in [t_i, t_{i+1}]$,
		\begin{equation}\label{equ:mean-square-estimate-XG}
			\E\left[\|\Law((X_r^{\mathcal{D}}, a_{t_i})|\G_r)-\Law((X_{t_i}^{\mathcal{D}}, a_{t_i})|\G_{t_i})\|_{K, \rm FM}\right] \leq \E\left[|X_r^{\mathcal{D}} - X_{t_i}^{\mathcal{D}}|^2\right]^{1/2} \leq C |\mathcal{D}|^{1/2}.
		\end{equation}
		By~\eqref{equ:mean-square-estimate-XG}, Lipschtiz continuity on $b$, $\sigma$, $\gamma$ in \assref{ass}, and continuity on $J_h$, $\pa_t J_h$, $\pa_\mu J_h$, $\pa_x\pa_\mu J_h$ in \assref{ass:regularity-Jh}, we conclude that
		$|e_i| \leq C (t_{i+1} - t_i) |\mathcal{D}|^{1/2}$. The desired result follows from $\sum_{i=0}^{k-1} |e_i| \leq C T |\mathcal{D}|^{1/2}$.

		\hfill$\Box$

		The above {\rm \lemref{lem:diff-ol-fp}} provides the key estimate needed to establish {\rm \propref{prop:diff-ol-fp}}. We now prove {\rm \propref{prop:diff-ol-fp}}. 
		
		\noindent{\it Proof of \propref{prop:diff-ol-fp}.}\quad
		Note that
		\begin{equation*}
			\begin{aligned}
				&J_{ol}(0, \vartheta; a^{K, \epsilon}) - J_{fp}(0, \mu; \pi_\epsilon)
				= \E\left[g\left(X_T^{\mathcal{D}}, \Law(X_T^{\mathcal{D}}|\G_T)\right)\right] - \E\left[\hat g(\mu_T)\right]\\
				&\qquad + \sum_{k=0}^{K-1} \E\left[\int_{t_k}^{t_{k+1}} \left(f(X_s^{\mathcal{D}}, \Law((X_s^{\mathcal{D}}, a_{t_k})|\G_s), a_{t_k}) - \hat f(\mu_s, \pi_{\epsilon, s}(\mu_s))\right)ds\right] =: (I) + \sum_{k=0}^{K-1} (II^k).
			\end{aligned}
		\end{equation*}
		By using {\rm \lemref{lem:diff-ol-fp}} for $h = \hat g$, we have
		\begin{equation}\label{estimate-I}
			\begin{aligned}
				|(I)| = \left|\E[g(X_T^{\mathcal{D}}, \Law(X_T^{\mathcal{D}}|\G_T))] - \E[\hat g(\mu_T)]\right| = \left|\hat g(\Law(X_T^{\mathcal{D}}|\G_T)) - \E[\hat g(\mu_T)]\right| \leq C |\mathcal{D}|^{1/2}.
			\end{aligned}
		\end{equation}
		For the second term $\sum_{k=0}^{K-1} (II^k)$, note that
		\begin{equation}\label{difference-ol-fp}
			\begin{aligned}
				(II^k) 
				= & \E\left[ \int_{t_k}^{t_{k+1}} \big(f(X_s^{\mathcal{D}}, \Law((X_s^{\mathcal{D}}, a_{t_k})|\G_s), a_{t_k}) - f(X_{t_k}^{\mathcal{D}}, \Law((X_{t_k}^{\mathcal{D}}, a_{t_k})|\G_{t_k}), a_{t_k})\big)\d s\right]\\
				& +  \E\left[ \int_{t_k}^{t_{k+1}} \big(\hat f(\Law(X_{t_k}^{\mathcal{D}}|\G_{t_k}), \pi_{\epsilon, t_k}(\Law(X_{t_k}^{\mathcal{D}}|\G_{t_k}))) - \hat f(\mu_{t_k}, \pi_{\epsilon, t_k}(\mu_{t_k}))\big)\d s\right]\\
				& + \E\left[\int_{t_k}^{t_{k+1}} \big(\hat f(\mu_{t_k}, \pi_{\epsilon, t_k}(\mu_{t_k})) - \hat f(\mu_s, \pi_{\epsilon, s}(\mu_s))\big)\d s \right]
				=: (II^k_1) + (II^k_2) + (II^k_3),
			\end{aligned}
		\end{equation}
		where we use $\E[f(X_{t_k}^{\mathcal{D}}, \Law((X_{t_k}^{\mathcal{D}}, a_{t_k})|\G_{t_k}), a_{t_k})]= \hat f(\Law(X_{t_k}^{\mathcal{D}}|\G_{t_k}), \pi_{\epsilon, t_k}(\Law(X_{t_k}^{\mathcal{D}}|\G_{t_k})))$ due to the fact that $a_{t_k} = a(t_k, X_{t_k}^{\mathcal{D}}, \Law((X_{t_k}^{\mathcal{D}}|\G_{t_k}), I_k) \sim \pi_{\epsilon,t_k}(\Law((X_{t_k}^{\mathcal{D}}|\G_{t_k}))$ and that the random variable $I_k$ is independent of $\G_{t_k}$. 
		For the term $(II_1^k)$, we have from \assref{ass} 
		\begin{equation}\label{estimate-IIk1}
			\begin{aligned}
				&\left|E\left[ \int_{t_k}^{t_{k+1}} \big(f(X_s^{\mathcal{D}}, \Law((X_s^{\mathcal{D}}, a_{t_k})|\G_s), a_{t_k}) - f(X_{t_k}^{\mathcal{D}}, \Law((X_{t_k}^{\mathcal{D}}, a_{t_k})|\G_{t_k}), a_{t_k})\big)\d s\right]\right| \\
				& \quad \leq  L \int_{t_k}^{t_{k+1}}\left|\E\left[|X_s^{\mathcal{D}} - X_{t_k}^{\mathcal{D}}| + \|\Law((X_s^{\mathcal{D}}, a_{t_k})|\G_s)- \Law((X_{t_k}^{\mathcal{D}}, a_{t_k})|\G_{t_k})\|_{K, FM}\right]\right|\d s\\
				& \quad \leq  2L \int_{t_k}^{t_{k+1}}\E\left[|X_s^{\mathcal{D}} - X_{t_k}^{\mathcal{D}}|^2\right]^{1/2} ds\leq C(t_{k+1} - t_k)|\mathcal{D}|^{1/2},
			\end{aligned}
		\end{equation}
		where we have used the Kantorovich duality of $\|\cdot\|_{K, {\rm FM}}$ and Cauchy-Schwarz inequality in the second inequality, and used ~\eqref{equ:mean-square-estimate-XG} in the last inequality. It then follows from \lemref{lem:diff-ol-fp} with $h(\cdot) = \hat f(\cdot, \pi_\epsilon(t_k,\cdot))$ that
		{\small\begin{equation}\label{estimate-IIk2}
				\begin{aligned}
					|(II_2^K)| = \left|\E\left[ \int_{t_k}^{t_{k+1}} \big(\hat f(\Law(X_{t_k}^{\mathcal{D}}|\G_{t_k}), \pi_{\epsilon, t_k}(\Law(X_{t_k}^{\mathcal{D}}|\G_{t_k})))) - \hat f(\mu_{t_k}, \pi_{\epsilon, t_k}(\mu_{t_k}))\big)ds\right]\right| \leq C (t_{k+1} - t_k) \cdot |\mathcal{D}|^{1/2}.
				\end{aligned}
		\end{equation}}
		For the term $(II_3^k)$, it follows from \lemref{lem:diff-ol-fp} and \equref{equ:mean-square-estimate-XG} that
		\begin{equation}\label{estimate-IIk3}
			\begin{aligned}
				&\left|\E\left[\int_{t_k}^{t_{k+1}} \big(\hat f(\mu_{t_k}, \pi_{\epsilon, t_k}(\mu_{t_k})) - \hat f(\mu_s, \pi_{\epsilon, s}(\mu_s))\big)\d s\right] \right|\\
				\leq &\left|\E\left[\int_{t_k}^{t_{k+1}} \big(\hat f(\mu_{t_k}, \pi_{\epsilon, t_k}(\mu_{t_k})) - \hat f(\Law(X_{t_k}^{\mathcal{D}}|\G_{t_k}), \pi_{\epsilon, t_k}(\Law(X_{t_k}^{\mathcal{D}}|\G_{t_k})))\big)\d s\right] \right|\\
				&+\left|\E\left[\int_{t_k}^{t_{k+1}} \big(\hat f(\mu_{s}, \pi_{\epsilon, s}(\mu_{s})) - \hat f(\Law(X_{s}^{\mathcal{D}}|\G_{s}), \pi_{\epsilon, s}(\Law(X_{s}^{\mathcal{D}}|\G_{s})))\big)\d s\right] \right|\\
				&+\left|\E\left[\int_{t_k}^{t_{k+1}} \big(\hat f(\Law(X_{s}^{\mathcal{D}}|\G_{s}), \pi_{\epsilon, s}(\Law(X_{s}^{\mathcal{D}}|\G_{s}))) - \hat f(\Law(X_{t_k}^{\mathcal{D}}|\G_{t_k}), \pi_{\epsilon, t_k}(\Law(X_{t_k}^{\mathcal{D}}|\G_{t_k})))\big)\d s\right] \right|\\
				\leq &C(t_{k+1}-t_k)|\mathcal{D}|^{1/2}.
			\end{aligned}
		\end{equation}
		Combining~\eqref{estimate-I},~\eqref{estimate-IIk1},~\eqref{estimate-IIk2} and~\eqref{estimate-IIk3}, we obtain the desired result.
		\hfill$\Box$
		\ \\
		{}
		
	\end{document}